\documentclass{amsart}
\usepackage[margin=1.5in]{geometry}
\usepackage{amsthm}
\usepackage{amssymb}
\usepackage[non-compressed-cites]{amsrefs}
\usepackage{esint}
\usepackage{caption}
\usepackage{enumitem}
\usepackage{mathtools}
\usepackage{xcolor}
\usepackage{etoolbox}
\usepackage{booktabs}
\usepackage{microtype}
\usepackage{hyperref}

\allowdisplaybreaks

\newcommand{\R}{\mathbb{R}}
\newcommand{\Sph}{\mathbb{S}}
\newcommand{\del}{\partial}
\newcommand{\Ric}{\operatorname{Ric}}
\renewcommand{\div}{\operatorname{div}}
\newcommand{\Rm}{\operatorname{Rm}}
\newcommand{\tr}{\operatorname{tr}}
\newcommand{\dvol}{\mathrm{dvol}}
\newcommand\inner[2]{\langle #1, #2 \rangle}
\newcommand{\norm}[1]{\left\lVert #1 \right\rVert}
\renewcommand{\O}{\mathcal{O}}
\renewcommand{\L}{\mathcal{L}}
\newcommand{\Z}{\mathbb{Z}}
\newcommand{\CP}{\mathbb{CP}}
\newcommand{\RP}{\mathbb{RP}}
\newcommand{\I}{\mathcal{I}}
\newcommand{\C}{\mathbb{C}}
\renewcommand{\epsilon}{\varepsilon}

\newcommand{\bfi}{\mathbf{i}}
\newcommand{\bfj}{\mathbf{j}}
\newcommand{\bfk}{\mathbf{k}}

\newtheorem{theorem}{Theorem}[section]
\newtheorem{corollary}[theorem]{Corollary}
\newtheorem{lemma}[theorem]{Lemma}
\newtheorem{definition}[theorem]{Definition}
\newtheorem{proposition}[theorem]{Proposition}

\newtheorem{remark}[theorem]{Remark}
\numberwithin{equation}{section}
 
\hypersetup{
    colorlinks,
    linkcolor={red!60!black},
    citecolor={blue!80!black},
}

\patchcmd{\subsection}
  {-.5em}{.5\baselineskip}{}{}
\patchcmd{\subsubsection}
  {-.5em}{.5\baselineskip}{}{}
\patchcmd{\section}
  {\normalfont\scshape\centering}
  {\normalfont\large\bfseries\centering}
  {}{}

\setlist[itemize]{leftmargin=3em}
\setlist[enumerate]{leftmargin=3em}
\mathtoolsset{showonlyrefs}

\title[On steady and expanding Ricci solitons with asymptotic symmetries]{On steady and expanding Ricci solitons with asymptotic symmetries}
\author[Michael B. Law]{Michael B. Law}
\address{MIT, Department of Mathematics, 77 Massachusetts Avenue, Cambridge, MA 02139, USA.}
\email{mikelaw@mit.edu}

\begin{document}

\begin{abstract}
    We establish a symmetry principle for asymptotically cylindrical steady gradient Ricci solitons and asymptotically conical expanding gradient Ricci solitons. Our symmetry principle bounds the loss of symmetry from the asymptotic link to the interior by the dimension of the weighted $L^2$ kernel of the stability operator. For homogeneous links, suitable kernel bounds imply that the soliton has an isometric action of cohomogeneity one. The results apply in particular to strictly linear stable solitons, require no curvature sign assumptions, and allow orbifold singular sets of codimension at least four.

    Applications include characterizations of Bryant solitons and global $U(m)$-symmetry results for solitons with quotient-Berger asymptotic links. In dimension four, we classify asymptotically $SU(2)/\Gamma$-cylindrical steady gradient Ricci solitons with trivial weighted kernel as finite quotients of Bryant or Appleton solitons. Under the same kernel condition, we also establish toric symmetry for asymptotically lens-cylindrical steady solitons.
\end{abstract}

\maketitle
\tableofcontents

\section{Introduction}

Let $(M^n,g)$ be a complete, connected smooth Riemannian orbifold of dimension $n$, and let $f: M \to \R$ be a smooth function. The triple $(M^n,g,f)$ is called a \emph{gradient Ricci soliton} (GRS) if
\begin{align} \label{eq:RS}
    \Ric = \nabla^2 f + \lambda g,
\end{align}
where $\lambda$ determines the soliton type: \emph{shrinking} ($\lambda>0$), \emph{steady} ($\lambda = 0$), or \emph{expanding} ($\lambda<0$). GRSs generate self-similar solutions of the Ricci flow and are fundamental to the analysis of singularities. In particular, shrinking and steady GRSs are special instances of \emph{ancient solutions}, many of which are \emph{singularity models}: blowup limits of Ricci flow on compact manifolds at finite-time singularities. On the other hand, steady and expanding GRSs are expected to play a key role in desingularization by geometric surgeries.

In dimension $3$, all singularity models are \emph{ancient $\kappa$-solutions}: non-flat, $\kappa$-noncollapsed ancient solutions with bounded and nonnegative sectional curvature. The classification of 3-dimensional ancient $\kappa$-solutions was completed by Brendle \cite{brendle-acta} and Brendle--Daskalopoulos--Sesum \cite{bds} in the noncompact and compact cases respectively. On the other hand, expanding GRSs play a less prominent role in dimension 3, and a full classification remains open. For more information on these developments, we refer to the books \cites{chow2023ricci,chowluni} and references therein.

In dimensions $\geq 4$, ancient $\kappa$-solutions no longer model all finite-time singularities: singularity models only have nonnegative scalar curvature in general.
A full classification therefore appears unfeasible; indeed, even a classification of Einstein 4-manifolds remains distant. Recent successes in the theory of generic mean curvature flow (e.g. \cites{cm-generic,mcf-generic-ccms}) suggest the more realistic goal of only classifying GRSs which are \emph{stable}, since these are thought to govern singularity formation and resolution for generic solutions to Ricci flow.

To date, many classification results for GRSs and ancient solutions take place on smooth manifolds under assumptions on both asymptotic geometry and curvature; see for instance \cites{brendle2014rotational,brendle-naff,zz,deng-zhu,chodosh,chodosh-fong,cds}.
Motivated by the discussion above, our goal here is to investigate what remains true in the absence of curvature assumptions. In this paper, we investigate uniqueness and symmetry of \emph{asymptotically cylindrical} steady GRSs and \emph{asymptotically conical} expanding GRSs, replacing curvature assumptions with an analytical condition on the stability operator \eqref{eq:L-defn}.

\textbf{Throughout, we assume that the orbifold singular set of $M$ has real codimension at least four.} This regularity assumption is motivated by Bamler's partial regularity theory of Ricci flows \cites{bamler-a,bamler-b,bamler-c}. Smoothness of functions and tensor fields is understood via lifts to local orbifold charts, and diffeomorphisms are understood in the orbifold sense. By a \emph{manifold} we mean an orbifold with empty orbifold singular set.

In accompanying work \cite{law}, we extend these ideas to \emph{Ricci-flat ALE spaces}, which themselves form a class of steady GRSs, but are also of independent interest.

We now define the asymptotic conditions of interest. Recall that a \emph{spherical space form} of dimension $d$ is a manifold diffeomorphic to $\Sph^d/\Gamma$, where $\Gamma \subset O(d+1)$ is a finite subgroup acting freely on $\Sph^d$.
\begin{definition} \label{def:acyl-quotient}
    Let $n \geq 3$. Let $\mathcal{S}$ be a spherical space form of dimension $n-1$, and let $g_{\mathcal{S}}$ be the round metric on $\mathcal{S}$ with constant sectional curvature $1$. A steady GRS $(M^n,g,f)$ is called \emph{asymptotically $\mathcal{S}$-cylindrical} if the following holds. Let $\Phi_t$ be the flow of $-\nabla f$. Take an arbitrary sequence $p_m$ of points going to infinity and consider the rescaled flows
    \begin{align}
        \hat{g}^{(m)}(t) = r_m^{-1} \Phi_{r_m t}^*g,
    \end{align}
    with $r_m R(p_m) = \frac{n-1}{2} + o(1)$, where $R(p_m)$ is the scalar curvature at $p_m$. We require that the flows $(M,\hat{g}^{(m)}(t),p_m)$ converge as $m \to \infty$ (without passing to a subsequence) in the smooth pointed Cheeger--Gromov sense to a family of shrinking cylinders $(\mathcal{S} \times \R, \bar{g}(t))$, $t \in [0,1)$, where
    \begin{align} \label{eq:gbar-defn}
        \bar{g}(t) = (n-2)(2-2t) g_{\mathcal{S}} + dz \otimes dz.
    \end{align}
\end{definition}

\begin{definition} \label{def:acon}
    Let $n \geq 3$, and let $(\mathcal{S},\bar{g}_\infty)$ be a closed, connected Riemannian manifold of dimension $n-1$. An expanding GRS $(M^n,g,f)$ is called \emph{asymptotically $(\mathcal{S},\bar{g}_\infty)$-conical} if there is a compact set $K \subset M$, a radius $R_0>0$, a diffeomorphism $\phi: (R_0,\infty) \times \mathcal{S} \to M \setminus K$, and a decay rate $\epsilon > 0$ such that
    \begin{gather}
        |\bar{\nabla}^k(\phi^*g - \bar{g})|_{\bar{g}} = \O(r^{-\epsilon-k}) \text{ as } r \to \infty \text{ for all } k \geq 0, \label{eq:acon-diffeo-decay}
    \end{gather}
    where $\bar{g} = dr^2 + r^2 \bar{g}_\infty$ is the conical metric on $(R_0,\infty) \times \mathcal{S}$, and $\bar{\nabla}$ is its Levi-Civita connection.
\end{definition}
\begin{remark}
    \begin{itemize}
        \item Brendle's definition of asymptotic cylindricity \cite{brendle2014rotational} also requires that $\frac{\Lambda_1}{r} \leq R \leq \frac{\Lambda_2}{r}$, where $\Lambda_1, \Lambda_2 > 0$ and $r$ is the distance to a fixed point in $M$. It turns out that Definition \ref{def:acyl-quotient} implies this \cite{CDN}*{\S 3}.
        \item Definition \ref{def:acyl-quotient} requires the asymptotic link to be a round spherical space form, while Definition \ref{def:acon} imposes no such restrictions.
    \end{itemize}
\end{remark}

To formulate the analytical condition underlying our results, define the stability operator on symmetric $2$-tensors by
\begin{align} \label{eq:L-defn}
    Lh := \Delta h + \nabla_{\nabla f} h + 2\Rm(h),
\end{align}
where $\Delta = \tr_g \nabla^2$ is the rough Laplacian and $\Rm(h)_{ik} = \Rm_{ijkl} h^{jl}$ is the action of the curvature tensor. The operator $L$ is formally self-adjoint with respect to the weighted measure $e^f \, \dvol_g$. We write $L^2(e^f)$ for the corresponding space of weighted $L^2$ tensors, i.e. those satisfying $\norm{h}_{L^2(e^f)}^2 := \int_M |h|^2 e^f < \infty$. For a compactly supported symmetric $2$-tensor $h$, define
\begin{align}
    \mathcal{E}[h] := -\int_M \inner{Lh}{h} e^f \, \dvol_g = \int_M \left( |\nabla h|^2 - 2\inner{\Rm(h)}{h} \right) e^f \, \dvol_g, 
\end{align}
where the second equality follows from integrating by parts.

We say that a steady or expanding GRS $(M,g,f)$ is \emph{strictly linearly stable} if there exists $\mu > 0$ such that $$\mathcal{E}[h] \geq \mu\norm{h}_{L^2(e^f)}^2$$ for every smooth compactly supported symmetric $2$-tensor $h$. In particular, by a standard cutoff argument, strict linear stability implies
\begin{align}
    \ker_{L^2(e^f)}(L) = \{0\}.
\end{align}
Our symmetry and classification results will assume upper bounds on the dimension of this kernel and therefore apply, in particular, to strictly linearly stable solitons.

\subsection{A symmetry principle}

Our first result is the following symmetry principle. Given a Riemannian orbifold $(M,g)$, denote by $\I(M,g)$ its isometry group, and $\I^0(M,g)$ its identity component. Recall that $(M,g)$ is \emph{homogeneous} if $\I(M,g)$ acts transitively. If $G$ is a compact Lie group acting on $M$, then the action is \emph{effective} if only the identity element acts trivially. The action has \emph{cohomogeneity one} if its generic orbits are hypersurfaces.

\begin{theorem} \label{thm:general-symmetry-intro}
    Let $(M^n,g,f)$ be an asymptotically $\mathcal{S}$-cylindrical steady GRS or an asymptotically $(\mathcal{S},\bar{g}_\infty)$-conical expanding GRS. In the steady case, set $\bar{g}_\infty = g_{\mathcal{S}}$. Let $\bar{d} = \dim \I(\mathcal{S},\bar{g}_\infty)$, and assume that $d := \dim\ker_{L^2(e^f)}(L) < \infty$. Then:
    \begin{enumerate}[label=(\alph*)]
        \item There exists a closed connected Lie subgroup $G \subseteq \I^0(\mathcal{S},\bar{g}_\infty)$ with dimension $\dim G \geq \bar{d}-d$ acting smoothly, effectively, and isometrically on $(M,g)$. The action preserves $f$, and every sufficiently far level set of $f$ is $G$-equivariantly diffeomorphic to $\mathcal{S}$.
        \item Suppose in addition that $(\mathcal{S},\bar{g}_\infty)$ is homogeneous, and define $d_*(\mathcal{S},\bar{g}_\infty)$ to be the largest nonnegative integer $j$ such that every closed connected Lie subgroup of $\I^0(\mathcal{S},\bar{g}_\infty)$ with dimension at least $\bar{d}-j$ acts transitively on $\mathcal{S}$. If $d \leq d_*(\mathcal{S},\bar{g}_\infty)$, then $G$ acts transitively on every sufficiently far level set of $f$. In particular, its action on $M$ has cohomogeneity one.
    \end{enumerate}
\end{theorem}
The proof generalizes Brendle's Killing field technique pioneered in \cites{brendle2013rotational,brendle2014rotational}. This technique forms the backbone of various uniqueness theorems for Ricci solitons and Ricci flows under positive curvature assumptions \cites{chodosh,chodosh-fong,brendle-acta,brendle-naff,zz}. For solitons, positivity of curvature makes certain maximum principles accessible, which then allow approximate Killing fields arising from symmetries of the asymptotic link $(\mathcal{S},\bar{g}_\infty)$ to be promoted to exact Killing fields. Without curvature assumptions, our argument identifies $\ker_{L^2(e^f)}(L)$ as an obstruction space for this promotion, so bounding its dimension controls the possible loss of symmetry when passing from the asymptotic link to the soliton.

By specializing Theorem \ref{thm:general-symmetry-intro} to certain links $(\mathcal{S},\bar{g}_\infty)$, we obtain symmetry and classification results described below.

\subsection{Characterizations of Bryant solitons} \label{subsec:uniqueness-bryant}

If $(\mathcal{S},\bar{g}_\infty)$ is a round sphere, an application of Theorem \ref{thm:general-symmetry-intro} leads to the following result.
\begin{theorem} \label{thm:bryant-classification-intro}
    Let $(M^n,g,f)$, $n \geq 3$, be either
    \begin{enumerate}[label=(\roman*)]
        \item an asymptotically $\Sph^{n-1}$-cylindrical steady GRS; or
        \item an asymptotically $(\Sph^{n-1}, \alpha g_{\Sph^{n-1}})$-conical expanding GRS for some $\alpha > 0$, where $g_{\Sph^{n-1}}$ is the round metric of constant sectional curvature $1$.
    \end{enumerate}
    Suppose
    \begin{align} \label{eq:bryant-kernel-bound-intro}
        \dim\ker_{L^2(e^f)}(L) \leq \begin{cases}
            n-2 &\text{if } n=3 \text{ or } n \geq 5, \\
            1 &\text{if } n=4.
        \end{cases}
    \end{align}
    Then $M$ is diffeomorphic to $\R^n$, and $g$ is rotationally symmetric. Thus in case (i), $(M,g)$ is isometric to the Bryant steady soliton \cite{bryant}. In case (ii), $(M,g)$ is isometric to the Bryant expanding soliton with parameter $\alpha$ \cite{bryant}.
\end{theorem}
Here and throughout, identifications with named solitons (and their quotients) are understood up to constant rescaling of the metric.

In the steady case, Brendle \cites{brendle2013rotational,brendle2014rotational} established the same uniqueness (on smooth manifolds) under a positive sectional curvature assumption instead of \eqref{eq:bryant-kernel-bound-intro}. His result also holds under weaker curvature conditions \cite{zz} and on smooth orbifolds with positive sectional curvature \cite{deng-orbifold}. Our theorem replaces the curvature assumption by the analytical condition \eqref{eq:bryant-kernel-bound-intro}. The Bryant steady soliton indeed satisfies \eqref{eq:bryant-kernel-bound-intro} as it is strictly linearly stable \cite{deruelle-stability}.

In the expanding case, Chodosh \cite{chodosh} proved the same uniqueness under a positive sectional curvature assumption instead of \eqref{eq:bryant-kernel-bound-intro}, similarly to Brendle's theorem. Positive sectional curvature of Bryant's expanders corresponds to $\alpha < 1$. However, Theorem \ref{thm:bryant-classification-intro} also applies to some nonpositively curved Bryant expanders ($\alpha \geq 1$), as Deruelle \cite{deruelle-stability} showed that there exists $\epsilon > 0$ such that strict linear stability, and hence $\ker_{L^2(e^f)}(L) = \{0\}$, holds when $\alpha \in (0,1+\epsilon)$.

The bounds \eqref{eq:bryant-kernel-bound-intro} arise from maximal dimensions of proper Lie subalgebras of $\mathfrak{so}(n)$. They ensure that the symmetry group supplied by Theorem \ref{thm:general-symmetry-intro} is the full rotation group $SO(n)$.

It remains unknown whether the kernel assumption \eqref{eq:bryant-kernel-bound-intro} can be dropped. The analogous result holds in mean curvature flow for the steady case, where the \emph{bowl soliton} plays the role of the Bryant steady soliton \cites{chhw,hershkovits}.

\subsection{Symmetry and classification with quotient-Berger links} \label{subsec:uniq-appl}

We next consider solitons whose asymptotic links are lens spaces with Berger metrics. For $m \geq 2$ and $k \geq 1$, the principal $U(1)$ Hopf fibration $\pi: \Sph^{2m-1} \to \CP^{m-1}$ induces a free action of the cyclic subgroup $\Z_k \subset U(1)$ on $\Sph^{2m-1}$. This action can be realized by scalar multiplication by a $k$-th root of unity on $\Sph^{2m-1} \subset \C^m$. Equip $\Sph^{2m-1}$ with a Berger metric $\tilde{g}_{a,b}$ for some $a,b>0$, that is
\begin{align}
    \tilde{g}_{a,b} = a^2 \sigma^2 + b^2 \pi^*g_{\mathrm{FS}}
\end{align}
where $\sigma$ is the $1$-form dual to the unit Hopf vector field with respect to the unit round metric, and $g_{\mathrm{FS}}$ is the Fubini--Study metric on $\CP^{m-1}$ with holomorphic sectional curvature $4$. Then the action of $\Z_k$ is isometric on $(\Sph^{2m-1}, \tilde{g}_{a,b})$, and the quotient is a lens space denoted by $(L_{m,k}, g_{a,b})$. When $a=b=\kappa$, the metrics $\tilde{g}_{a,b}$ and $g_{a,b}$ have constant sectional curvature $\kappa^{-2}$.

For each pair of integers $k > m \geq 2$, Appleton \cite{appleton} constructed complete, noncollapsed a steady GRS on the total space of the complex line bundles $\O_{\CP^{m-1}}(-k)$. (He also constructed collapsed examples as well as incomplete ones, but in this paper \emph{Appleton's solitons} refer to the complete, noncollapsed ones.) The complement of the zero section $B \cong \CP^{m-1}$ in $\O_{\CP^{m-1}}(-k)$ is diffeomorphic to $(0,\infty) \times L_{m,k}$ (see Appendix \ref{subsec:line-bundle-geometry}), and on this set the soliton metric and potential have the form
\begin{align} \label{eq:appleton-ansatz-intro}
    g = ds^2 + g_{a(s),b(s)}, \qquad f=f(s),
\end{align}
where $s$ is the distance from $B$.
Appleton's solitons are asymptotically $L_{m,k}$-cylindrical. In another work \cite{law-appleton}, we establish uniqueness of Appleton's solitons among complete noncollapsed steady GRSs within this ansatz, up to scaling and isometry, and prove strict linear stability of the four-dimensional examples on $\O_{\CP^1}(-k)$, $k \geq 3$.

The quotient-Berger ansatz \eqref{eq:appleton-ansatz-intro} also occurs in expanding GRSs, including the K\"ahler examples of Feldman--Ilmanen--Knopf \cite{fik} and further four-dimensional examples constructed by Donovan \cite{donovan}. Our next application of Theorem \ref{thm:general-symmetry-intro} recovers this ansatz from the asymptotic geometry and a dimension bound on $\ker_{L^2(e^f)}(L)$.

\begin{theorem} \label{thm:berger-symmetry-intro}
    Let $m \geq 2$ and $k \geq 3$, and suppose $(M^{2m},g,f)$ is either
    \begin{enumerate}[label=(\roman*)]
        \item an asymptotically $L_{m,k}$-cylindrical steady GRS; or
        \item an asymptotically $(L_{m,k}, g_{\alpha,\beta})$-conical expanding GRS for some $\alpha,\beta > 0$.
    \end{enumerate}
    Suppose
    \begin{align} \label{eq:berger-kernel-bound-intro}
        \dim\ker_{L^2(e^f)}(L) \leq \begin{cases}
            0 &\text{if }m=2,\\
            2 &\text{if }m=4,\\
            2m-3 &\text{if }m=3\text{ or }m\geq5.
        \end{cases}
    \end{align}
    Then either
    \begin{enumerate}[label=(\alph*)]
        \item $M$ is a smooth manifold diffeomorphic to $\O_{\CP^{m-1}}(-k)$; or
        \item $M$ is diffeomorphic to $\C^m/\Z_k$, where $\Z_k$ acts by scalar multiplication.
    \end{enumerate}
    Under this identification, the natural $U(m)$-action on $M$ is isometric with cohomogeneity one. Let $B$ denote the zero section in case (a) and the image of the origin in case (b). On $M \setminus B$, the soliton metric and potential take the form
    \begin{align} \label{eq:berger-ansatz-results-intro}
        g = ds^2 + g_{a(s),b(s)}, \qquad f=f(s),
    \end{align}
    where $s$ is the distance from $B$. In case (i), we have $\lim_{s\to\infty} \frac{a(s)}{b(s)} = 1$, while in case (ii) we have $\lim_{s\to\infty} \frac{a(s)}{s} = \alpha$ and $\lim_{s\to\infty} \frac{b(s)}{s} = \beta$.
\end{theorem}

In dimension four, corresponding to $m=2$, the steady case of Theorem \ref{thm:berger-symmetry-intro} enables a classification. If $M \cong \O_{\CP^1}(-k)$, then \cite{law-appleton}*{Theorem 1.1} identifies the soliton with Appleton's example. If $M \cong \C^2/\Z_k$, we show that the soliton is a $\Z_k$-quotient of the Bryant steady soliton. Conversely, these examples are strictly linearly stable by \cite{law-appleton}*{Theorem 1.3} and \cite{deruelle-stability} respectively. We therefore obtain the following characterization.

\begin{corollary} \label{cor:appleton-characterization-intro}
    Let $k\geq3$, and let $(M^4,g,f)$ be an
    asymptotically $L_{2,k}$-cylindrical steady GRS.
    Then $\ker_{L^2(e^f)}(L)=\{0\}$ if and only if
    $(M,g)$ is isometric, up to scaling, to either
    \begin{enumerate}[label=(\alph*)]
        \item the Appleton soliton on $\O_{\CP^1}(-k)$; or
        \item the $\Z_k$-quotient of the four-dimensional
        Bryant steady soliton, where $\Z_k$ acts on
        $\R^4\cong\C^2$ by scalar multiplication.
    \end{enumerate}
\end{corollary}

\subsection{Four-dimensional steady solitons with spherical space form links}

Corollary \ref{cor:appleton-characterization-intro}
forms part of a classification of asymptotically
$SU(2)/\Gamma$-cylindrical steady GRSs with $\ker_{L^2(e^f)}(L) = \{0\}$, where $\Gamma$ is a finite
subgroup of $SU(2)$.
Here $SU(2)$ is identified with the unit round $\Sph^3$, and $SU(2)/\Gamma$ denotes the quotient by right multiplication. Up to conjugacy, the finite subgroups of $SU(2)$ are:
\begin{enumerate}[label=(\roman*)]
    \item the cyclic groups $\Z_k$, $k \geq 1$;
    \item the binary dihedral groups $\tilde{D}_k$, $k \geq 2$, of order $4k$;
    \item the binary tetrahedral, octahedral, and icosahedral groups $\tilde{T}$, $\tilde{O}$, and $\tilde{I}$, of orders $24$, $48$, and $120$, respectively.
\end{enumerate}
In the cyclic cases, $SU(2)/\Z_k = L_{2,k}$. We describe these groups in \S\ref{subsec:SU2-preliminaries}.

Write $B^4$ for the $4$-dimensional Bryant steady soliton and $A_k$ for the $4$-dimensional Appleton soliton on $\O_{\CP^1}(-k)$, $k \geq 3$. The quotient $B^4/\Gamma$ is defined by radially extending the right-multiplication action of $\Gamma$ on $\Sph^3 \cong SU(2)$. For $k \geq 2$, the soliton $A_{2k}$ also admits a $\Z_2$-action whose quotient is an asymptotically $SU(2)/\tilde{D}_k$-cylindrical steady GRS on the non-orientable $2$-plane bundle over $\RP^2$ with Euler number $\pm k$. We denote this quotient by $A_{2k}/\Z_2$; see Appendix \ref{subsec:appleton-Z2-quotients} for details.

\begin{theorem}
\label{thm:spherical-space-form-classification}
    Let $\Gamma\subset SU(2)$ be a finite subgroup,
    and let $(M^4,g,f)$ be an asymptotically
    $SU(2)/\Gamma$-cylindrical steady GRS.
    Then $\ker_{L^2(e^f)}(L)=\{0\}$ if and only if
    $(M,g)$ is isometric to one of
    the solitons in the corresponding row below.
    \begin{center}
        \begin{tabular}{c|c|l}
            Group $\Gamma$ & Solitons & Reference \\
            \hline \noalign{\vskip 1pt}
            $\{1\}$
            & $B^4$
            & Theorem \ref{thm:bryant-classification-intro} \\

            $\Z_2$
            & $B^4/\Z_2$
            & Proposition \ref{prop:RP3-bryant} \\

            $\Z_k$, $k\geq3$
            & $B^4/\Z_k$ and $A_k$
            & Corollary \ref{cor:appleton-characterization-intro} \\

            $\tilde{D}_k$, $k\geq2$
            & $B^4/\tilde{D}_k$ and $A_{2k}/\Z_2$
            & Corollary \ref{cor:binary-dihedral-classification} \\

            $\tilde{T},\tilde{O},\tilde{I}$
            & $B^4/\Gamma$
            & Proposition \ref{prop:polyhedral-bryant} \\
        \end{tabular}
    \end{center}
\end{theorem}
\begin{remark}
    Shortly before this version of the paper was released, Deng and Wang \cite{deng-wang} announced a classification of complete, non-Ricci-flat, noncollapsed cohomogeneity one steady GRSs on smooth manifolds in dimensions four and five. In dimension four, their result identifies these solitons as Bryant and Appleton solitons up to finite covers. For smooth manifolds, their classification can therefore be combined with Theorem \ref{thm:general-symmetry-intro} to obtain the `only if' direction of Theorem \ref{thm:spherical-space-form-classification}. We retain our proof to keep the paper self-contained and to account for the orbifold cases in our theorem.
\end{remark}

Bamler--Chow--Deng--Ma--Zhang \cite{bamler-3} proved that a four-dimensional steady gradient Ricci soliton singularity model has a unique tangent flow at infinity. If this tangent flow is the shrinking quotient-cylinder $(\Sph^3/\Gamma) \times \R$, then the soliton is asymptotically $\Sph^3/\Gamma$-cylindrical. This motivates the problem of classifying such solitons for all finite subgroups $\Gamma \subset SO(4)$ acting freely on $\Sph^3$. Under the additional assumption $\ker_{L^2(e^f)}(L) = \{0\}$, Theorem \ref{thm:spherical-space-form-classification} settles the cases where $\Gamma$ lies in the image of either $SU(2)$ factor in $SO(4) \cong SU(2) \times_{\Z_2} SU(2)$.

A natural next case concerns lens spaces
\begin{align}
    L(p,q) = \Sph^3/\langle \operatorname{diag}(e^{2\pi i/p}, e^{2\pi iq/p}) \rangle, \qquad \gcd(p,q) = 1.
\end{align}
The cases $q \equiv \pm 1 \pmod p$ correspond to $L_{2,p}$, with opposing orientations. For $q \not\equiv \pm 1 \pmod p$, the round lens space $(L(p,q), g_{L(p,q)})$ has isometry group equal to a finite extension of the $2$-torus $\mathbb{T}^2$ \cite{isoms-3}. Since $\mathbb{T}^2$ is the only compact connected Lie group of dimension $2$, the next result is then immediate from Theorem \ref{thm:general-symmetry-intro}(a).
\begin{corollary} \label{cor:lens-toric-intro}
    Let $(M^4,g,f)$ be an asymptotically $L(p,q)$-cylindrical steady GRS, where $\gcd(p,q) = 1$ and $q \not\equiv \pm 1 \pmod p$. If $\ker_{L^2(e^f)}(L) = \{0\}$, then $M$ admits an effective isometric $\mathbb{T}^2$-action preserving $f$.
\end{corollary}

Besides quotients of Bryant steady solitons, examples satisfying these conditions include cyclic quotients of Appleton's solitons. Appendix \ref{subsec:appleton-cyclic-quotients} presents a family of such quotients with asymptotic links $L(pk,qk+1)$ and each having two isolated orbifold singularities.

Building on the seminal work of Orlik--Raymond \cite{orlik-raymond} on four-manifolds admitting effective $\mathbb{T}^2$-actions, there is a substantial body of recent work addressing questions on existence, uniqueness, and classification of toric gravitational instantons (see e.g. \cites{kunduri-lucietti,li-sun,andersson-dahl-nilsson,araneda-lucietti} and references therein). It would be interesting to see if any of the techniques from those studies can be adapted to classify the steady GRSs appearing in Corollary \ref{cor:lens-toric-intro}.

\begingroup
\renewcommand{\addcontentsline}[3]{}

\subsection*{Structure of the paper} \label{subsec:outline}

In \S\ref{sec:geom-acyl}, we collect geometric preliminaries for asymptotically cylindrical steady GRSs and asymptotically conical expanding GRSs.

In \S\ref{sec:global-symmetry}, we prove the symmetry principle (Theorem \ref{thm:general-symmetry-intro}) and recall relevant structure results for cohomogeneity one orbifolds.

In \S\ref{sec:round-link-classification}, we establish the characterizations of Bryant solitons (Theorem \ref{thm:bryant-classification-intro}).

In \S\ref{sec:berger-link-classification}, we prove the symmetry and classification results for quotient-Berger links (Theorem \ref{thm:berger-symmetry-intro} and Corollary \ref{cor:appleton-characterization-intro}).

In \S\ref{sec:spherical-space-form-classification}, we complete the classification of asymptotically $SU(2)/\Gamma$-cylindrical steady GRSs (Theorem \ref{thm:spherical-space-form-classification}).

Appendix \ref{app:appleton-quotients} describes the geometry of $\O_{\CP^{m-1}}(-k)$ and constructs finite quotients of Appleton's solitons with binary dihedral and lens-space asymptotic links.

\subsection*{Note on this version}

ArXiv versions $1$ and $2$ of this paper contained a fatal error in Lemma 3.11 (old numbering). The author is grateful to Bennett Chow and Yuxing Deng for pointing this out. Consequently, the main claims of those versions do not hold as stated. The present version replaces those arguments and should be regarded as superseding the previous versions.

\subsection*{Statement on AI usage}

Large language models (LLMs) suggested the proofs of Lemmas \ref{lem:berger-no-exceptional-orbit} and \ref{lem:triaxial-round-asymptotics}, and helped to simplify several steps in \S\ref{subsec:integrate-KFs}. All other arguments were developed without AI involvement. LLMs also assisted in proofreading the manuscript. The author takes full responsibility for the correctness of all claims in this paper.

\subsection*{Acknowledgements}

We thank William Minicozzi and Tristan Ozuch for their insights and continual support, as well as Shrey Aryan, Pak-Yeung Chan, Yonghwan Kim, and Yongjia Zhang for inspiring discussions. This work was partially supported by a Croucher Scholarship and by a Simons Dissertation Fellowship in Mathematics.

\par
\endgroup

\section{Geometry of asymptotically cylindrical and conical solitons} \label{sec:geom-acyl}

This section collects known geometric properties of the spaces we study.

\subsection{Asymptotically cylindrical steady GRSs} \label{subsec:grs-conventions}

In this subsection, we assume $(M^n,g,f)$ is an asymptotically $\mathcal{S}$-cylindrical steady GRS (see Definition \ref{def:acyl-quotient}).
Recall that the steady GRS equation is
\begin{align} \label{eq:GRS}
    \Ric = \nabla^2 f.
\end{align}
Denote by $R$ the scalar curvature of $(M,g)$. Standard identities for steady GRSs give $R \geq 0$ and $R + |\nabla f|^2 \equiv c \geq 0$, and if equality holds in either then $\Ric = 0$. Clearly, asymptotically $\mathcal{S}$-cylindrical steady GRSs are not Ricci flat, so
\begin{align} \label{eq:R0}
    R > 0
\end{align}
and by scaling the metric we may arrange that $c=1$; thus
\begin{align} \label{eq:identity-1}
    R + |\nabla f|^2 = 1.
\end{align}
Our convention for the Laplacian on any tensor is $\Delta = \tr_g \nabla^2 = -\nabla^*\nabla$. The drift Laplacian $\Delta_f$ on tensors is defined by
\begin{align}
    \Delta_f = \Delta + \nabla_{\nabla f}.
\end{align}
This is formally self-adjoint with respect to the weighted volume form $e^f \dvol_g$. Under the normalization \eqref{eq:identity-1}, one has the soliton identity
\begin{align} \label{eq:dff}
    \Delta_f f = 1.
\end{align}

\begin{lemma} \label{lem:compact-singular-set}
    The orbifold singular set of $M$ is contained in a compact set.
\end{lemma}
\begin{proof}
    Otherwise, there would be a sequence of orbifold-singular points $p_i$ tending to infinity. By Definition \ref{def:acyl-quotient}, the corresponding rescaled pointed flows converge smoothly to a shrinking $\mathcal{S}$-cylinder. Since this cylinder is a smooth manifold, the pointed convergence implies that for all sufficiently large $i$, $p_i$ is not an orbifold point. This is a contradiction.
\end{proof}

Let $r: M \to \R$ be the distance from a fixed point with respect to $g$.

\begin{lemma}[\cite{CDN}*{\S 2}, also \cite{bamler-3}*{Lemma 3.14}] \label{lem:f-r-comparable}
    We have $\frac{f(x)}{r(x)} \to 1$ as $r(x) \to \infty$.
\end{lemma}
In particular, $f$ is proper and bounded from below. As \eqref{eq:GRS} and \eqref{eq:identity-1} are invariant under translating $f$, we will henceforth assume $\min_M f = 1$.

\begin{lemma} \label{lem:f-regular}
    For sufficiently large $\rho > 0$, we have $\nabla f \neq 0$ on $\{f \geq \rho\}$. Hence, $\{f=\rho\}$ is a smooth compact hypersurface in $M$.
\end{lemma}
\begin{proof}
    By \cite{deng-zhu}*{Lemma 6.5}, asymptotic $\mathcal{S}$-cylindricity implies that $R = o(1)$ as $r \to \infty$.
    Then the claim follows from \eqref{eq:identity-1} and Lemma \ref{lem:f-r-comparable}.
\end{proof}

Next, we state the roundness estimates on the level sets of $f$ which were first established by Brendle \cite{brendle2014rotational} and subsequently improved by Zhao--Zhu \cite{zz}.

\begin{proposition}[\cite{zz}*{Propositions 2.5 and 2.7}] \label{prop:roundness-estimates}
    There exists $\delta > 0$ such that on $M$,
    \begin{align} \label{eq:fR}
        f R &= \frac{n-1}{2} + \O(r^{-\frac{1}{2}-\delta})
    \end{align}
    and
    \begin{align} \label{eq:fRm}
        f \Rm_{ijkl} &= \frac{1}{2(n-2)} (g_{ik} - \del_i f \del_k f)(g_{jl} - \del_j f \del_l f) \\
        &\quad - \frac{1}{2(n-2)} (g_{il} - \del_i f \del_l f)(g_{jk} - \del_j f \del_k f) + \O(r^{-\frac{1}{2}-\delta}).
    \end{align}
\end{proposition}
For the full proof of this result, see Propositions 2.1--2.5 in \cite{brendle2014rotational} followed by Section 2.1 of \cite{zz}. Once $f$ and $r$ are known to be comparable (Lemma \ref{lem:f-r-comparable}), these proofs only use asymptotic cylindricity, do not require any curvature hypotheses, and take place outside a large compact set, where $M$ is a smooth manifold by Lemma \ref{lem:compact-singular-set}. They carry over verbatim to asymptotically $\mathcal{S}$-cylindrical steady GRSs even when $\mathcal{S} \neq \Sph^{n-1}$.

We also record the following estimates, which are useful for adapting the arguments from \cites{brendle2014rotational,zz} in \S\ref{sec:global-symmetry}.
\begin{corollary} \label{cor:shi-steady}
    We have $|\nabla^k {\Rm}| = \O(r^{-1-\frac{k}{2}})$ for each $k \geq 0$. Also, $|\nabla f| = \O(1)$ and $|\nabla^k f| = \O(r^{-\frac{k}{2}})$ for each $k \geq 2$.
\end{corollary}
\begin{proof}
    The curvature bounds follow from Lemma \ref{lem:f-r-comparable}, Proposition \ref{prop:roundness-estimates}, and Shi's estimates. Combining this with \eqref{eq:identity-1} and the equation $\Ric = \nabla^2 f$, the bounds on $f$ follow.
\end{proof}

\begin{theorem} \label{thm:conn-at-infty}
    $M$ has exactly one end.
\end{theorem}
\begin{proof}
    By \cite{mw-smms}*{Corollary 1.1}, a complete steady GRS on a smooth manifold with $R>0$ is connected at infinity. By \cite{li-wang}*{Theorem A.22}, this still holds in our present setting of an orbifold with singular set of codimension at least $4$.
\end{proof}

\begin{corollary} \label{cor:area-noncollapsed}
    The hypersurfaces $\{f=\rho\}$ are diffeomorphic to $\mathcal{S}$ for all large $\rho$.
\end{corollary}
\begin{proof}
    By Lemma \ref{lem:f-r-comparable} and Lemma \ref{lem:f-regular}, the level sets $\{f=\rho\}$ are smooth hypersurfaces bounding the end(s) of $M$. Then by Theorem \ref{thm:conn-at-infty}, $\{f=\rho\}$ is connected for all large $\rho$.

    Take a sequence of points $p_m$ going to infinity, and let $r_m = f(p_m)$, $\hat{g}^{(m)}(t) = r_m^{-1} \Phi_{r_m t}^*g$. By \cite{deng-zhu}*{Lemma 6.5}, the Ricci curvature of the level sets of $f$ decays linearly in $r$. Then by the Bonnet--Myers theorem, the $g$-diameter of $\{f=r_m\}$ is bounded by $C\sqrt{r_m}$, so the $\hat{g}^{(m)}(t)$-diameter of $\{f=r_m\}$ is uniformly bounded over all $m$ and over all $t$ in a compact interval within $(0,1)$.
    
    By asymptotic $\mathcal{S}$-cylindricity, the sequence $(M, \hat{g}^{(m)}(t), p_m)$ converges in the Cheeger--Gromov sense to a shrinking $\mathcal{S}$-cylinder \eqref{eq:gbar-defn} for $t \in (0,1)$. Moreover, the vector fields $\sqrt{r_m} \nabla f$ converge to the axial vector field $\frac{\del}{\del z}$ on $\mathcal{S} \times \R$. Using this, and the assertions on the $\hat{g}^{(m)}(t)$-diameter of $\{f=r_m\}$ from the last paragraph, it is easy to see that $\{f=r_m\}$ is diffeomorphic to $\mathcal{S}$ for all large $m$.
\end{proof}

By the above results, $f$ is a coordinate on the end of $M$ with level sets $\Sigma_\rho := \{f=\rho\}$ diffeomorphic to $\mathcal{S}$. Let $F_\rho: \mathcal{S} \to \Sigma_\rho$ be a one-parameter family of diffeomorphisms such that $\frac{\del}{\del\rho} F_\rho = \frac{\nabla f}{|\nabla f|^2} \circ F_\rho$. Write $g_\rho$ for the metric induced on $\Sigma_\rho$. Then the end of $(M,g)$ is isometric to
\begin{align} \label{eq:IS}
    \left( I \times \mathcal{S} , \frac{d\rho^2}{|\nabla f|^2 \circ F_\rho} + F_\rho^*g_\rho \right),
\end{align}
where $I$ is an interval of the form $[\rho_0,\infty)$. Following \cite{zz}*{\S 2.2}, we have the following convergence and closeness bounds:
\begin{lemma} \label{lem:g-closeness}
    There exists $\delta > 0$ such that for each $k \geq 0$, we have
    \begin{align}
        \norm{\frac{1}{2(n-2)\rho} F_\rho^*g_\rho - \bar{g}_{\infty}}_{C^k(\mathcal{S},\bar{g}_{\infty})} = \O(\rho^{-\frac{1}{2}-\frac{1}{2}\delta})
    \end{align}
    as $\rho \to \infty$ where, after a suitable choice of the initial identification $F_{\rho_0}$, we have $\bar{g}_\infty = g_{\mathcal{S}}$.
\end{lemma}
Thus, $(M,g)$ opens up like a paraboloid, and has volume growth of order $r^{\frac{n+1}{2}}$.

\subsection{Asymptotically conical expanding GRSs} \label{subsec:egs-conventions}

In this subsection, we assume $(M^n,g,f)$ is an asymptotically $(\mathcal{S},\bar{g}_\infty)$-conical expanding GRS (see Definition \ref{def:acon}). After globally rescaling the metric, we may assume that
\begin{align} \label{eq:EGS-equation}
    \Ric = \nabla^2 f - \frac{1}{2}g.
\end{align}
With this normalization, standard identities for expanding GRSs give
\begin{align}
    R + \frac{n}{2} &\geq 0, \label{eq:EGS-1} \\
    R + |\nabla f|^2 &= f + \mu, \label{eq:EGS-2} \\
    \Delta_f f &= f + \frac{n}{2} + \mu, \label{eq:EGS-3}
\end{align}
where $\mu \in \R$ is a constant. Replacing $f$ by $f+\mu$, we henceforth assume $\mu=0$. We denote by $r: M \to \R$ the distance from a fixed point with respect to $g$.

\begin{lemma} \label{lem:EGS-curv-estimates}
    We have $|\nabla^k{\Rm}| = \O(r^{-2-k})$ for each $k \geq 0$.
\end{lemma}
\begin{proof}
    This follows directly from \eqref{eq:acon-diffeo-decay}.
\end{proof}

These estimates allow us to choose asymptotic coordinates within which $f$ is a radial function. Note that the asymptotically conical assumption makes the end of $M$ a smooth manifold, so Deruelle's argument (which takes place outside a compact set) applies.
\begin{proposition}[\cite{deruelle-asymptotic}*{Theorem 1.2 and \S 3.1}] \label{prop:EGS-deruelle-diffeo}
    There exist $s_0 > 0$ and a diffeomorphism $\Psi: (s_0,\infty) \times \mathcal{S} \to \{f > s_0^2/4\}$ such that
    \begin{align}
        f(\Psi(s,\theta)) = \frac{s^2}{4}, \qquad \frac{\del\Psi}{\del s} = \frac{s}{2} \frac{\nabla f}{|\nabla f|^2} \circ \Psi,
    \end{align}
    and for every integer $k \geq 0$,
    \begin{align} \label{eq:EGS-deruelle-coord-decay}
        |\bar{\nabla}^k(\Psi^*g - \bar{g})|_{\bar{g}} = \O(s^{-2-k}),
    \end{align}
    where $\bar{g} = ds^2 + s^2 \bar{g}_\infty$ is the conical metric on $(s_0,\infty) \times \mathcal{S}$.
\end{proposition}

\begin{corollary} \label{cor:f-asymptotics-EGS}
    We have $f(x) \sim \frac{r(x)^2}{4}$ as $r(x) \to \infty$. In particular, $f$ is proper and bounded from below. For sufficiently large $\rho$, we have $\nabla f \neq 0$ on $\{f \geq \rho\}$. Hence, $\{f=\rho\}$ is a smooth hypersurface diffeomorphic to $\mathcal{S}$.
\end{corollary}
\begin{proof}
    This follows directly from Proposition \ref{prop:EGS-deruelle-diffeo}.
\end{proof}

Using the adapted coordinate system $\Psi$ from Proposition \ref{prop:EGS-deruelle-diffeo}, let $F_\rho: \mathcal{S} \to \Sigma_\rho := \{f=\rho\}$ be defined by $F_\rho(\theta) = \Psi(2\sqrt{\rho},\theta)$. Then $\frac{\del}{\del\rho} F_\rho = \frac{\nabla f}{|\nabla f|^2} \circ F_\rho$. Write $g_\rho$ for the metric induced on $\Sigma_\rho$. Then the end of $(M,g)$ is isometric to
\begin{align} \label{eq:IS-EGS}
    \left( I \times \mathcal{S}, \frac{d\rho^2}{|\nabla f|^2 \circ F_\rho} + F_\rho^*g_\rho \right),
\end{align}
where $I = (s_0^2/4,\infty)$. By \eqref{eq:EGS-deruelle-coord-decay}, we also have the following convergence and closeness bounds:

\begin{lemma} \label{lem:g-closeness-EGS}
    For each $k \geq 0$, we have
    \begin{align}
        \norm{\frac{1}{4\rho} F_\rho^*g_\rho - \bar{g}_{\infty}}_{C^k(\mathcal{S},\bar{g}_\infty)} = \O(\rho^{-1}).
    \end{align}
\end{lemma}

\section{A global symmetry principle} \label{sec:global-symmetry}

In this section, we prove Theorem \ref{thm:general-symmetry-intro} and provide relevant preparations for its applications.

\subsection{Killing fields on asymptotically cylindrical steady GRSs}

In this subsection, we fix an asymptotically $\mathcal{S}$-cylindrical steady GRS $(M^n,g,f)$. We first pull back Killing fields from $(\mathcal{S},\bar{g}_\infty)$ to obtain approximate Killing fields on $(M,g)$. Recall the identifications
\begin{align}
    F_\rho: \mathcal{S} \to \Sigma_\rho := \{f=\rho\}
\end{align}
fixed towards the end of \S\ref{subsec:grs-conventions}. Given a Killing field $\bar{U}$ on $(\mathcal{S},\bar{g}_\infty)$, we will say that a vector field $U$ on $M$ \emph{extends $\bar{U}$} if for all sufficiently large $\rho$, $U$ is tangent to $\Sigma_\rho$ and $F_\rho^*(U|_{\Sigma_\rho}) = \bar{U}$.

\begin{proposition}[\cite{zz}*{Proposition 3.1 and Theorem 3.3}] \label{prop:approx-KF}
    Let $\bar{U}$ be a Killing field on $(\mathcal{S},\bar{g}_\infty)$, and let $U$ be a smooth vector field on $M$ which extends $\bar{U}$.
    Then there is a smooth vector field $W$ on $M$ such that
    \begin{align}
        \Delta_f W &= 0, \\
        |W-U| &= \O(1), \\
        |\L_W g| &= \O(r^{-\frac{1}{2}}).
    \end{align}
\end{proposition}

The main result of this subsection is Proposition \ref{prop:exact-KFs}, which upgrades some of the approximate Killing fields $W$ from Proposition \ref{prop:approx-KF} to exact Killing fields. The key is to study a PDE satisfied by the Lie derivative of $W$. Recall the operator $L$ defined in \eqref{eq:L-defn}. As $\Ric = \nabla^2 f$, one easily checks that
\begin{align} \label{eq:L2}
    Lh = \Delta_L h + \L_{\nabla f}h,
\end{align}
where $\Delta_L$ is the Lichnerowicz Laplacian defined by
\begin{align} \label{eq:lich-lap}
    \Delta_L h_{ik} = \Delta h_{ik} + 2 \Rm_{ijkl} h^{jl} - \Ric_i^l h_{kl} - \Ric_k^l h_{il}.
\end{align}
Brendle made the following crucial observations:
\begin{theorem}[\cite{brendle2013rotational}*{\S 4}] \label{thm:brendle-PDE}
    Let $W$ be a vector field satisfying $\Delta_f W = 0$. Then $h = \L_W g$ satisfies $Lh = 0$. In particular, taking $W = \nabla f$ gives $L\Ric = 0$.
\end{theorem}

We will prove strong decay estimates for solutions of $Lh = 0$, starting with a computation:
\begin{lemma} \label{lem:h-subharmonic}
    For all sufficiently large $\ell > 0$, there exists a compact set $K_\ell \subset M$ such that for all symmetric 2-tensors $h$ on $M$ satisfying $Lh = 0$, we have
    \begin{align}
        \Delta_f (f^\ell |h|^2) \geq 0 \quad \text{on } M \setminus K_\ell.
    \end{align}
\end{lemma}
\begin{proof}
    Since $Lh = 0$, we have
    \begin{align} \label{eq:Deltafh2}
        \Delta_f |h|^2 &= 2 \inner{\Delta_f h}{h} + 2|\nabla h|^2 = -4\inner{\Rm(h)}{h} + 2|\nabla h|^2.
    \end{align}
    Let $\ell > 0$. Using this and the soliton identity $\Delta_f f = 1$, we get
    \begin{align}
        \Delta_f (f^\ell |h|^2) &= (\Delta_f f^\ell) |h|^2 + f^\ell \Delta_f |h|^2 + 2\inner{\nabla f^\ell}{\nabla|h|^2} \\
        &= \left( \ell f^{\ell-1} + \ell(\ell-1)f^{\ell-2} |\nabla f|^2 \right) |h|^2 \\
        &\quad + f^\ell \left( 2|\nabla h|^2 - 4\inner{\Rm(h)}{h} \right) + 2\ell f^{\ell-1} \inner{\nabla f}{\nabla|h|^2} \\
        &\geq \ell f^{\ell-1} |h|^2 + \ell(\ell-1) f^{\ell-2} |\nabla f|^2 |h|^2 + 2 f^\ell |\nabla h|^2 \\
        &\quad - 4f^\ell |{\Rm}| |h|^2 - 4\ell f^{\ell-1} |\nabla f| |h| |\nabla|h||. \label{eq:barrier-f}
    \end{align}
    By Proposition \ref{prop:roundness-estimates}, we have
    \begin{align} \label{eq:lin-curv-decay}
        |{\Rm}| \leq Cf^{-1}.
    \end{align}
    Moreover, by the absorbing inequality and Kato's inequality, we have
    \begin{align} \label{eq:absorbing-1}
        2|\nabla f||h||\nabla|h|| \leq \frac{f}{\ell} |\nabla h|^2 + \frac{\ell}{f} |\nabla f|^2 |h|^2.
    \end{align}
    Using these facts in \eqref{eq:barrier-f} gives
    \begin{align}
        \Delta_f (f^\ell |h|^2) &\geq (\ell - C) f^{\ell-1} |h|^2 - (\ell+\ell^2) f^{\ell-2} |\nabla f|^2 |h|^2 \\
        &= f^{\ell-2} |h|^2 \left[ (\ell-C) f - (\ell + \ell^2) |\nabla f|^2 \right].
    \end{align}
    Since $|\nabla f|^2 \leq 1$, and $f \to +\infty$ at infinity, it follows that if $\ell > C$, the right-hand side is nonnegative outside a compact set of $M$.
\end{proof}

The next two results are due to Brendle and Zhao--Zhu, respectively.

\begin{proposition}[\cite{brendle2014rotational}*{Lemma 4.1}] \label{prop:brendle-blowdown}
    Consider the shrinking cylinder $(\mathcal{S} \times \R, \bar{g}(t))$, $t \in (0,1)$, where $\bar{g}(t)$ is given by \eqref{eq:gbar-defn}. Let $\bar{h}(t)$, $t \in (0,1)$, be a one-parameter family of symmetric 2-tensors which solve the parabolic equation
    \begin{align}
        \frac{\del}{\del t} \bar{h}(t) = \Delta_{L,\bar{g}(t)} \bar{h}(t),
    \end{align}
    where $\Delta_{L,\bar{g}(t)}$ is the Lichnerowicz Laplacian with respect to $\bar{g}(t)$ (see \eqref{eq:lich-lap}). Moreover, suppose $\bar{h}(t)$ is invariant under translations along the axis of the cylinder, and
    \begin{align}
        |\bar{h}(t)|_{\bar{g}(t)} \leq (1-t)^{-\ell}
    \end{align}
    for all $t \in (0,\frac{1}{2}]$ and some $\ell > 0$. Then there exists $N > 0$ depending on $\ell$ such that
    \begin{align}
        \inf_{\lambda \in \R} \sup_{\mathcal{S} \times \R} |\bar{h}(t) - {\lambda \Ric_{\bar{g}(t)}}|_{\bar{g}(t)} \leq N (1-t)^{\frac{1}{2(n-2)}-\frac{1}{2}}
    \end{align}
    for all $t \in [\frac{1}{2},1)$.
\end{proposition}
\begin{proof}
    The proof of \cite{brendle2014rotational}*{Lemma 4.1} applies after replacing every occurrence of $(\Sph^{n-1},g_{\Sph^{n-1}})$ there with $(\mathcal{S}, \bar{g}_\infty)$. (The proof uses spectral properties of the Laplacian which still hold on $(\mathcal{S},\bar{g}_\infty)$.)
\end{proof}

\begin{proposition}[\cite{zz}*{Theorem 4.3}] \label{prop:superpoly-decay}
    If $Lh = 0$ on $M$ and $|h| = \O(r^{-\frac{1}{2}})$, then there exists $\lambda \in \R$ such that $|h - {\lambda \Ric}| = \O(r^{-\ell})$ for every $\ell > 0$.
\end{proposition}
\begin{proof}[Proof sketch]
    Due to differences in organization, we will outline the proof, omitting steps which can essentially be copied down from the original source.

    Let $h$ satisfy $Lh = 0$ and $|h| = \O(r^{-\frac{1}{2}})$. Consider the function
    \begin{align} \label{eq:Ar-defn}
        A(\rho) = \inf_{\lambda \in \R} \sup_{\{f=\rho\}} |h - {\lambda\Ric}|.
    \end{align}
    Let $r_m \to \infty$ be a sequence of numbers, and denote by $\lambda_m$ the minimizing $\lambda$'s associated with $A(r_m)$.

    By Theorem \ref{thm:brendle-PDE}, we have $L(h - \lambda_m \Ric) = 0$ for each $m$. Then by Lemma \ref{lem:h-subharmonic} and the maximum principle, for each large $\ell > 0$ there exists $\rho_\ell > 0$ such that for all $m$,
    \begin{equation} \label{eq:fmp}
        \sup_{\{\rho_\ell \leq f \leq r_m\}} f^{2\ell} |h - {\lambda_m\Ric}|^2 \leq \max\left\{ \rho_\ell^{2\ell} \sup_{\{f=\rho_\ell\}} |h - {\lambda_m \Ric}|^2, r_m^{2\ell} \sup_{\{f=r_m\}} |h - {\lambda_m \Ric}|^2 \right\}.
    \end{equation}
    We divide into two cases:

    \textbf{Case 1:} there exists $\ell > 0$ and a subsequence of $r_m$ (still denoted $r_m$) such that
    \begin{align}
        \rho_\ell^\ell \sup_{\{f=\rho_\ell\}} |h - {\lambda_m \Ric}| \leq r_m^\ell \sup_{\{f=r_m\}} |h - {\lambda_m \Ric}|.
    \end{align}
    This is Case 1 in the proof of \cite{zz}*{Theorem 4.3}. Following the argument there (which uses Proposition \ref{prop:brendle-blowdown}), we see that the $\lambda_m$'s are all equal to a constant $\lambda$, and
    \begin{align}
        h = \lambda\Ric \quad \text{on } \{f \geq \rho_\ell\}.
    \end{align}
    This clearly implies $|h - {\lambda\Ric}| = \O(r^{-\ell})$ for every $\ell > 0$.

    \textbf{Case 2:} for all $\ell > 0$, there exists $R_\ell > 0$ such that for all $r_m > R_\ell$,
    \begin{align}
        \rho_\ell^\ell \sup_{\{f=\rho_\ell\}} |h - {\lambda_m \Ric}| > r_m^\ell \sup_{\{f=r_m\}} |h - {\lambda_m \Ric}|.
    \end{align}
    Then by \eqref{eq:fmp}, for each $\ell > 0$, and for all $m$ such that $r_m > R_\ell$, one has
    \begin{align} \label{eq:asdf}
        \sup_{\{\rho_\ell \leq f \leq r_m\}} f^\ell |h - {\lambda_m \Ric}| \leq \rho_\ell^\ell \sup_{\{f=\rho_\ell\}} |h - {\lambda_m \Ric}|.
    \end{align}
    Following the arguments in Case 2 in the proof of \cite{zz}*{Theorem 4.3} (in particular Claim 4.4 there), we conclude that the sequence $\lambda_m$ is uniformly bounded, so a subsequence of the $\lambda_m$ converges to some $\lambda \in \R$. Fixing an arbitrary $\ell > 0$ in \eqref{eq:asdf} and taking $m \to \infty$, we get
    \begin{align}
        \sup_{\{f \geq \rho_\ell\}} f^\ell |h - {\lambda\Ric}| \leq \rho_\ell^\ell \sup_{\{f=\rho_\ell\}} |h-{\lambda\Ric}|.
    \end{align}
    Hence, for all $\rho \geq \rho_\ell$,
    \begin{align}
        \sup_{\{f=\rho\}} |h-{\lambda\Ric}| \leq \frac{\rho_\ell^\ell}{\rho^\ell} \sup_{\{f=\rho_\ell\}} |h - {\lambda\Ric}| \leq C_\ell \rho^{-\ell}
    \end{align}
    where $C_\ell > 0$ depends on $\ell$. So $|h - {\lambda\Ric}| = \O(r^{-\ell})$ for each $\ell > 0$.
\end{proof}

Finally, we record a simple application of soliton identities and the maximum principle.
\begin{lemma} \label{lem:chan-exp}
    Suppose $u$ is a function such that $\Delta_f u \geq 0$ outside a compact set and $u \to 0$ at infinity. Then $u \leq Ce^{-f}$ on $M$ for some $C>0$.
\end{lemma}
\begin{proof}
    Recall that $R \geq 0$ on any steady GRS. By this and \eqref{eq:identity-1},
    \begin{align}
        \Delta_f e^{-f} = -e^{-f} \Delta_f f + e^{-f} |\nabla f|^2 = e^{-f} (-1 + |\nabla f|^2) = -R e^{-f} \leq 0.
    \end{align}
    Choose $\rho > 0$ large so that $\Delta_f u \geq 0$ on $\{f > \rho\}$. Choose $C > 0$ large so that $u - Ce^{-f} < 0$ on $\{f=\rho\}$. We have
    \begin{align} \label{eq:068}
        \Delta_f(u - Ce^{-f}) \geq 0 \quad \text{on } \{f > \rho\},
    \end{align}
    and $u - Ce^{-f} \to 0$ at infinity. It follows by the maximum principle that $u - Ce^{-f} \leq 0$ on $\{f > \rho\}$. Suitably increasing $C$, the lemma follows.
\end{proof}

We can now prove exponential decay of solutions to $Lh = 0$ up to a multiple of the Ricci tensor.
\begin{lemma} \label{lem:h-in-L2}
    If $Lh = 0$ on $M$ and $|h| = \O(r^{-\frac{1}{2}})$, then $h - \lambda\Ric \in L^2(e^f)$ for some $\lambda \in \R$.
\end{lemma}
\begin{proof}
    Let $h$ be as in the lemma.
    By Proposition \ref{prop:superpoly-decay}, we can find $\lambda \in \R$ such that $\tilde{h} = h - \lambda\Ric$ satisfies $f^\ell |\tilde{h}|^2 \to 0$ at infinity for all $\ell > 0$. Moreover, $L\tilde{h} = 0$ by Theorem \ref{thm:brendle-PDE}. By Lemma \ref{lem:h-subharmonic}, it holds for each sufficiently large $\ell > 0$ that
    \begin{align}
        \Delta_f(f^\ell |\tilde{h}|^2) \geq 0 \quad \text{outside a compact set}.
    \end{align}
    Then Lemma \ref{lem:chan-exp} gives $f^\ell |\tilde{h}|^2 = \O(e^{-f})$ outside a compact set. Since $f \sim r$ and $(M,g)$ has volume growth of order $r^{\frac{n+1}{2}}$, this implies $\tilde{h} \in L^2(e^f)$.
\end{proof}

Now we find exact Killing fields on $(M,g)$, following \cite{brendle2013rotational}*{Theorem 7.1}; the number of these will depend on a dimension bound for $\ker_{L^2(e^f)}(L)$. Denote by $\mathfrak{I}(\mathcal{S},\bar{g}_\infty)$ the Lie algebra of Killing fields on $(\mathcal{S},\bar{g}_\infty)$, so that $\dim \mathfrak{I}(\mathcal{S},\bar{g}_\infty) = \dim \I(\mathcal{S},\bar{g}_\infty) = \bar{d}$.

\begin{proposition} \label{prop:exact-KFs}
    Suppose $d := \dim\ker_{L^2(e^f)}(L) < \infty$. Then there exists a vector space $\mathcal{V}$ of Killing fields on $(M,g)$ with $\dim \mathcal{V} \geq \bar{d}-d$ such that for each $W \in \mathcal{V}$, we have $[W,\nabla f] = 0$ and $\inner{W}{\nabla f} = 0$ everywhere. In particular, $W$ restricts to a Killing field on each regular level set of $f$.
\end{proposition}
\begin{proof}
    Let $\{\bar{U}_i\}_{i=1}^{\bar{d}}$ be a basis for $\mathfrak{I}(\mathcal{S},\bar{g}_\infty)$, and let $\{U_i\}_{i=1}^{\bar{d}}$ be smooth extensions to $M$. Proposition \ref{prop:approx-KF} gives vector fields $\{\tilde{W}_i\}_{i=1}^{\bar{d}}$ satisfying
    \begin{align}
        \Delta_f \tilde{W}_i = 0, \qquad |\tilde{W}_i - U_i| = \O(1), \qquad |\L_{\tilde{W}_i} g| = \O(r^{-\frac{1}{2}}).
    \end{align}
    By Theorem \ref{thm:brendle-PDE}, we have $L(\L_{\tilde{W}_i}g) = 0$ and $L\Ric = 0$. Lemma \ref{lem:h-in-L2} gives $\lambda_i \in \R$ such that
    \begin{align}
        \L_{\tilde{W}_i} g - \lambda_i \Ric \in \ker_{L^2(e^f)}(L).
    \end{align}
    Define
    \begin{align}
        W_i = \tilde{W}_i - \frac{\lambda_i}{2} \nabla f.
    \end{align}
    Since $\L_{\nabla f}g = 2\nabla^2 f = 2\Ric$, we have $\L_{W_i}g \in \ker_{L^2(e^f)}(L)$. Moreover, $\Delta_f \nabla f = \nabla \Delta_f f = 0$ by the weighted Bochner formula and \eqref{eq:dff}. As $|\nabla f| \leq 1$, this implies
    \begin{align} \label{eq:steady-Wi-estimates}
        \Delta_f W_i = 0, \qquad |W_i - U_i| = \O(1).
    \end{align}
    
    The vector fields $\{W_i\}_{i=1}^{\bar{d}}$ are linearly independent. Indeed, if $\sum_i a_i W_i = 0$, then $|\sum_i a_i U_i|_g = \O(1)$. Pulling back to $\mathcal{S}$ using $F_\rho$ and applying Lemma \ref{lem:g-closeness} gives
    \begin{align}
        \left| \sum_i a_i \bar{U}_i \right|_{\bar{g}_\infty} = \O(\rho^{-\frac{1}{2}})
    \end{align}
    as $\rho\to\infty$. Thus $\sum_i a_i \bar{U}_i = 0$, so all $a_i$ vanish.

    Set $\mathcal{V}' = \operatorname{span}\{W_i\}_{i=1}^{\bar{d}}$, so that $\dim\mathcal{V}' = \bar{d}$ by the last paragraph.
    Define a map $\mathcal{T}: \mathcal{V}' \to \ker_{L^2(e^f)}(L)$ by $W \mapsto \L_W g$. Let $\mathcal{V} = \ker\mathcal{T}$. By rank-nullity, $\dim\mathcal{V} \geq \bar{d}-d$. Every $W \in \mathcal{V}$ is a Killing field for $(M,g)$. The general identity
    \begin{align}
        \div(\L_W g) - \frac{1}{2} \nabla(\tr \L_W g) = \Delta W + \Ric(W)
    \end{align}
    therefore gives $\Delta W + \Ric(W) = 0$. Thus
    \begin{align} \label{eq:steady-bracket-proof}
        [W,\nabla f] = \Ric(W) - \nabla_{\nabla f}W = -\Delta_f W = 0.
    \end{align}
    The fact that $W$ is Killing also implies $0 = [W,\nabla f] = \nabla(Wf)$, so $Wf = \inner{W}{\nabla f}$ is constant. Since $f$ attains a minimum by Lemma \ref{lem:f-r-comparable}, evaluating at this minimum point shows that $\inner{W}{\nabla f}$ vanishes identically. Consequently, $W$ is tangent to the regular level sets of $f$, and restricts to a Killing field on each of them.
\end{proof}

\subsection{Killing fields on asymptotically conical expanding GRSs}

In this subsection, we fix an asymptotically $(\mathcal{S},\bar{g}_\infty)$-conical expanding GRS $(M^n,g,f)$. We will prove an analog of Proposition \ref{prop:exact-KFs} in this setting; see Proposition \ref{prop:exact-KFs-EGS}. The method of proof is essentially the same.

Recall the identifications
\begin{align}
    F_\rho: \mathcal{S} \to \Sigma_\rho := \{f=\rho\}
\end{align}
fixed towards the end of \S\ref{subsec:egs-conventions}. Given a Killing field $\bar{U}$ on $(\mathcal{S},\bar{g}_\infty)$, we will say that a vector field $U$ on $M$ \emph{extends $\bar{U}$} if for all sufficiently large $\rho$, $U$ is tangent to $\Sigma_\rho$ and $F_\rho^*(U|_{\Sigma_\rho}) = \bar{U}$.

The following adapts \cite{chodosh}*{Lemma 4.1} to our setting.
\begin{lemma} \label{lem:chodosh}
    Fix $a>0$ such that $F := f+a \geq 1$ on $M$. Then for all sufficiently small $\epsilon > 0$, the following holds. If $Q$ is a smooth vector field on $M$ such that $|Q| = \O(r^{-2\epsilon})$, and $V$ is a vector field satisfying $\Delta_f V - \frac{1}{2} V = Q$ in $\{f \leq \rho^2\}$, then
    \begin{align}
        \sup_{\{f \leq \rho^2\}} (|V| - BF^{-\epsilon}) \leq \max\left\{ \sup_{\{f=\rho^2\}} (|V|-BF^{-\epsilon}), 0 \right\}
    \end{align}
    for some $B>0$ independent of $\rho$.
\end{lemma}
\begin{proof}
    Using the soliton identities \eqref{eq:EGS-2} and \eqref{eq:EGS-3}, we compute
    \begin{align}
        \left( \Delta_f - \frac{1}{2} \right) F^{-\epsilon} &= -\epsilon F^{-\epsilon-1} \Delta_f F + \epsilon(\epsilon+1) F^{-\epsilon-2} |\nabla F|^2 - \frac{1}{2} F^{-\epsilon} \\
        &= F^{-\epsilon} \left( -\frac{\epsilon}{F}\left( f + \frac{n}{2} \right) + \frac{\epsilon(\epsilon+1)}{F^2} (f-R) - \frac{1}{2} \right).
    \end{align}
    Since $\frac{1}{F}(f+\frac{n}{2})$ and $\frac{1}{F^2}(f-R)$ are bounded on $M$, we can take $\epsilon>0$ small enough so that
    \begin{align} \label{eq:233}
        \left( \Delta_f - \frac{1}{2} \right) F^{-\epsilon} \leq -\frac{1}{4} F^{-\epsilon} \quad \text{on } M.
    \end{align}
    Since $|Q| = \O(r^{-2\epsilon})$, there exists $B>0$ such that $|Q| \leq \frac{B}{4} F^{-\epsilon}$ on $M$. A computation using Kato's inequality then implies that
    \begin{align} \label{eq:234}
        \left( \Delta_f - \frac{1}{2} \right)|V| \geq -|Q|
    \end{align}
    on $\{f \leq \rho^2\}$ and whenever $V \neq 0$. Combining \eqref{eq:233} and \eqref{eq:234}, it follows that
    \begin{align}
        \left( \Delta_f - \frac{1}{2} \right) (|V| - BF^{-\epsilon}) \geq 0
    \end{align}
    on $\{f \leq \rho^2\}$ and whenever $|V| - BF^{-\epsilon} \geq 0$. An application of the strong maximum principle to $|V|-BF^{-\epsilon}$ completes the proof.
\end{proof}

The next proposition follows from the argument of \cite{chodosh}*{Proposition 4.2}, using Lemma \ref{lem:chodosh} in place of the barrier estimate there.
\begin{proposition} \label{prop:EGS-solve-PDE}
    Let $Q$ be a smooth vector field on $M$ such that $|Q| = \O(r^{-\epsilon})$ for some $\epsilon > 0$. Then there exist a smooth vector field $V$ on $M$ and $\delta > 0$ such that $\Delta_f V - \frac{1}{2} V = Q$ and $|V| + |\nabla V| = \O(r^{-\delta})$.
\end{proposition}

We may now establish an analog of Proposition \ref{prop:approx-KF}:
\begin{proposition} \label{prop:approx-KF-EGS}
    Let $\bar{U}$ be a Killing field on $(\mathcal{S},\bar{g}_\infty)$, and let $U$ be a smooth vector field on $M$ which extends $\bar{U}$. Then there is a smooth vector field $W$ on $M$ and $\delta > 0$ such that
    \begin{align}
        \Delta_f W - \frac{1}{2}W &= 0, \\
        |W-U| &= \O(r^{-\delta}), \\
        |\L_W g| &= \O(r^{-\delta}).
    \end{align}
\end{proposition}
\begin{proof}
    In the asymptotic coordinate system $\Psi: (s_0,\infty) \times \mathcal{S} \to \{f > s_0^2/4\}$ from Proposition \ref{prop:EGS-deruelle-diffeo}, $U$ is radially constant. Writing $\bar{g} = ds^2 + s^2 \bar{g}_\infty$, we have $\L_U \bar{g} = 0$ and $|\bar{\nabla}^k U|_{\bar{g}} = \O(s^{1-k})$ for each $k \geq 0$. Since $s \sim r$, the decay estimates \eqref{eq:EGS-deruelle-coord-decay} give
    \begin{align}
        |\nabla^k U| = \O(r^{1-k}) \quad \text{for each } k \geq 0.
    \end{align}
    Moreover, using $\L_U g = \L_U(g-\bar{g})$, we have
    \begin{align}
        |\nabla^k \L_U g| = \O(r^{-2-k}) \quad \text{for each } k \geq 0.
    \end{align}
    The general identity
    \begin{align}
        \div(\L_U g) - \frac{1}{2} \nabla(\tr_g \L_U g) = \Delta U + \Ric(U)
    \end{align}
    therefore implies
    \begin{align} \label{eq:459594}
        |\Delta U + \Ric(U)| \leq C|\nabla \L_U g| = \O(r^{-3}).
    \end{align}

    Since $U(f) = 0$ and $[U, \frac{\nabla f}{|\nabla f|^2}] = 0$ outside a compact set, the soliton equation and \eqref{eq:EGS-2} give
    \begin{align}
        \nabla_{\nabla f} U - \Ric(U) - \frac{1}{2} U = -[U,\nabla f] = -U(|\nabla f|^2) \frac{\nabla f}{|\nabla f|^2} = \inner{U}{\nabla R} \frac{\nabla f}{|\nabla f|^2}.
    \end{align}
    By Lemma \ref{lem:EGS-curv-estimates}, $|\nabla R| = \O(r^{-3})$. Combining this with $|U| = \O(r)$ and $|\nabla f| \sim r/2$, we obtain
    \begin{align} \label{eq:110}
        \left| \nabla_{\nabla f} U - \Ric(U) - \frac{1}{2} U \right| = \O(r^{-3}).
    \end{align}
    Together, \eqref{eq:459594} and \eqref{eq:110} imply
    \begin{align}
        \left| \Delta_f U - \frac{1}{2} U \right| = \O(r^{-3}). \label{eq:vf-pde-estimate}
    \end{align}
    Proposition \ref{prop:EGS-solve-PDE} now gives a smooth vector field and $0 < \delta < 2$ satisfying
    \begin{align}
        \Delta_f V - \frac{1}{2} V = -\Delta_f U + \frac{1}{2} U, \qquad |V| + |\nabla V| = \O(r^{-\delta}).
    \end{align}
    Setting $W := U+V$, we have
    \begin{gather}
        \Delta_f W - \frac{1}{2} W = 0, \qquad |W-U| = |V| = \O(r^{-\delta}), \\
        |\L_W g| \leq |\L_U g| + |\L_V g| \leq |\L_U g| + C|\nabla V| = \O(r^{-\delta}),
    \end{gather}
    as claimed.
\end{proof}

\begin{theorem}[\cite{chodosh}*{Proposition 3.1}] \label{thm:brendle-PDE-EGS}
    Let $W$ be a vector field satisfying $\Delta_f W - \frac{1}{2}W = 0$. Then $h = \L_W g$ satisfies $Lh = 0$.
\end{theorem}

As in the previous subsection, we will prove strong decay estimates for solutions of $Lh = 0$. This occupies the next four lemmas.

\begin{lemma} \label{lem:h-subharmonic-expander}
    For each $\ell > 0$, there exists a compact set $K_\ell \subset M$ such that for all symmetric 2-tensors $h$ on $M$ satisfying $Lh = 0$, we have
    \begin{align}
        \Delta_f(f^\ell |h|^2) \geq 0 \quad \text{on } M \setminus K_\ell.
    \end{align}
\end{lemma}
\begin{proof}
    Since $Lh = 0$, we have
    \begin{align}
        \Delta_f |h|^2 = -4\inner{{\Rm}(h)}{h} + 2|\nabla h|^2.
    \end{align}
    Let $\ell > 0$. Mimicking the computation \eqref{eq:barrier-f} but using expander identities, we get
    \begin{align}
        \Delta_f(f^\ell |h|^2) &\geq \ell f^\ell |h|^2 + \frac{n\ell}{2} f^{\ell-1} |h|^2 + \ell(\ell-1) f^{\ell-2} |\nabla f|^2 |h|^2 + 2 f^\ell |\nabla h|^2 \\
        &\quad - 4f^\ell |{\Rm}| |h|^2 - 4\ell f^{\ell-1} |\nabla f| |h| |\nabla|h||. \label{eq:expander-deltaf}
    \end{align}
    By Lemma \ref{lem:EGS-curv-estimates}, we have
    \begin{align} \label{eq:Rm-expander}
        |{\Rm}| \leq Cf^{-1}, \quad |\nabla f|^2 \leq Cf.
    \end{align}
    Moreover, by the absorbing inequality and Kato's inequality, we have
    \begin{align} \label{eq:absorbing-2}
        2|\nabla f||h||\nabla|h|| \leq \frac{f}{\ell} |\nabla h|^2 + \frac{\ell}{f} |\nabla f|^2 |h|^2.
    \end{align}
    Using these facts in \eqref{eq:expander-deltaf} gives
    \begin{align}
        \Delta_f(f^\ell |h|^2) \geq \ell f^\ell |h|^2 + \left( \frac{n}{2} \ell - C - (\ell+\ell^2)C \right) f^{\ell-1} |h|^2.
    \end{align}
    Since $f \to +\infty$ at infinity, this is nonnegative outside a compact set of $M$.
\end{proof}

\begin{lemma} \label{lem:h-subharmonic-expander-2}
    For each $\epsilon \in (0,1)$, there exists a compact set $K_\epsilon \subset M$ such that for all symmetric 2-tensors $h$ on $M$ satisfying $Lh = 0$, we have
    \begin{align}
        \Delta_f(e^{\epsilon f}|h|^2) \geq 0 \quad \text{on } M \setminus K_\epsilon.
    \end{align}
\end{lemma}
\begin{proof}
    Compute
    \begin{align}
        \Delta_f(e^{\epsilon f}|h|^2) &= (\epsilon e^{\epsilon f} \Delta_f f + \epsilon^2 e^{\epsilon f} |\nabla f|^2) |h|^2 \\
        &\quad + e^{\epsilon f} (2|\nabla h|^2 - 4\inner{{\Rm}(h)}{h}) + 2\epsilon e^{\epsilon f} \inner{\nabla f}{\nabla |h|^2} \\
        &\geq e^{\epsilon f}|h|^2 \left(\epsilon \left(f+\frac{n}{2}\right) + \epsilon^2 |\nabla f|^2 \right) + 2e^{\epsilon f}|\nabla h|^2 \\
        &\quad - 4e^{\epsilon f}|{\Rm}| |h|^2 - 4\epsilon e^{\epsilon f} |\nabla f| |h| |\nabla|h||. \label{eq:EE1}
    \end{align}
    By the absorbing inequality and Kato's inequality, we have
    \begin{align}
        2|\nabla f| |h| |\nabla|h|| \leq \epsilon^{-1} |\nabla h|^2 + \epsilon |\nabla f|^2 |h|^2.
    \end{align}
    Using this, \eqref{eq:Rm-expander}, and \eqref{eq:EGS-2} in \eqref{eq:EE1}, we obtain
    \begin{align}
        \Delta_f(e^{\epsilon f}|h|^2) &\geq e^{\epsilon f} |h|^2 \left( \epsilon \left(f + \frac{n}{2}\right) - \epsilon^2 |\nabla f|^2 - Cf^{-1} \right) \\
        &= e^{\epsilon f} |h|^2 \left( (\epsilon - \epsilon^2)f + \epsilon \frac{n}{2} + \epsilon^2 R - Cf^{-1} \right).
    \end{align}
    If $\epsilon \in (0,1)$, then this is nonnegative outside a compact set of $M$.
\end{proof}

\begin{lemma} \label{lem:chan-exp-expander}
    Suppose $u$ is a function such that $\Delta_f u \geq 0$ outside a compact set and $u \to 0$ at infinity. Then $u \leq Ce^{-f}$ on $M$ for some $C > 0$.
\end{lemma}
\begin{proof}
    The proof is the same as Lemma \ref{lem:chan-exp}, except that we use the estimate
    \begin{align}
        \Delta_f e^{-f} = e^{-f}(-\Delta_f f + |\nabla f|^2) = -e^{-f} \Delta f = -e^{-f} \left(R+\frac{n}{2}\right) \leq 0,
    \end{align}
    where we have used \eqref{eq:EGS-1} as well as the trace of \eqref{eq:EGS-equation}.
\end{proof}

\begin{lemma} \label{lem:h-in-L2-EGS}
    If $Lh = 0$ on $M$ and $|h| = \O(r^{-\delta})$ for some $\delta > 0$, then $h \in L^2(e^f)$.
\end{lemma}
\begin{proof}
    Let $\ell \in (0, \delta)$, so that $f^\ell |h|^2 \to 0$ at infinity. By Lemma \ref{lem:h-subharmonic-expander}, we have
    \begin{align}
        \Delta_f(f^\ell |h|^2) \geq 0 \quad \text{outside a compact set}.
    \end{align}
    Then Lemma \ref{lem:chan-exp-expander} gives $f^\ell |h|^2 = \O(e^{-f})$. Hence, $e^{\frac{1}{2}f} |h|^2 \to 0$ at infinity. By Lemma \ref{lem:h-subharmonic-expander-2}, we have
    \begin{align}
        \Delta_f(e^{\frac{1}{2}f} |h|^2) \geq 0 \quad \text{outside a compact set}.
    \end{align}
    Then Lemma \ref{lem:chan-exp-expander} gives $e^{\frac{1}{2}f} |h|^2 = \O(e^{-f})$. This implies $\int_M |h|^2 e^f < \infty$ since $(M,g)$ is asymptotically $(\mathcal{S}, \bar{g}_\infty)$-conical and thus has volume growth of order $r^n$.
\end{proof}

We can now prove an analog of Proposition \ref{prop:exact-KFs}. As before, $\mathfrak{I}(\mathcal{S},\bar{g}_\infty)$ denotes the space of Killing fields on $(\mathcal{S},\bar{g}_\infty)$, which has dimension $\bar{d}$.
\begin{proposition} \label{prop:exact-KFs-EGS}
    Suppose $d := \dim\ker_{L^2(e^f)}(L) < \infty$. Then there exists a vector space $\mathcal{V}$ of Killing fields on $(M,g)$ with $\dim\mathcal{V} \geq \bar{d}-d$ such that for each $W \in \mathcal{V}$, we have $[W,\nabla f] = 0$ and $\inner{W}{\nabla f} = 0$ everywhere. In particular, $W$ restricts to a Killing field on each regular level set of $f$.
\end{proposition}
\begin{proof}
    The proof is similar to that of Proposition \ref{prop:exact-KFs}. Let $\{\bar{U}_i\}_{i=1}^{\bar{d}}$ be a basis for $\mathfrak{I}(\mathcal{S},\bar{g}_\infty)$ and let $\{U_i\}_{i=1}^{\bar{d}}$ be smooth extensions to $M$. Proposition \ref{prop:approx-KF-EGS} gives vector fields $\{W_i\}_{i=1}^{\bar{d}}$ satisfying
    \begin{align} \label{eq:expander-WU-estimate}
        \Delta_f W_i - \frac{1}{2} W_i = 0, \qquad |W_i - U_i| = \O(r^{-\delta}), \qquad |\L_{W_i} g| = \O(r^{-\delta}),
    \end{align}
    for some $\delta > 0$. By Theorem \ref{thm:brendle-PDE-EGS}, $L(\L_{W_i} g) = 0$. Thus, by Lemma \ref{lem:h-in-L2-EGS},
    \begin{align}
        \L_{W_i} g \in \ker_{L^2(e^f)}(L).
    \end{align}
    Linear independence of the vector fields $\{W_i\}_{i=1}^{\bar{d}}$ follows from the same argument used in Proposition \ref{prop:exact-KFs}, except that we use Lemma \ref{lem:g-closeness-EGS} instead of Lemma \ref{lem:g-closeness}.

    Set $\mathcal{V}' = \operatorname{span}\{W_i\}_{i=1}^{\bar{d}}$, so that $\dim\mathcal{V}' = \bar{d}$ by the last paragraph. Define a map $\mathcal{T}: \mathcal{V}' \to \ker_{L^2(e^f)}(L)$ by $W \mapsto \L_W g$. Let $\mathcal{V} = \ker\mathcal{T}$. By rank-nullity, $\dim\mathcal{V} \geq \bar{d}-d$. Every $W \in \mathcal{V}$ is a Killing field for $(M,g)$.

    The remainder of the proof follows that of Proposition \ref{prop:exact-KFs}, except that \eqref{eq:steady-bracket-proof} becomes
    \begin{align}
        [W,\nabla f] = \Ric(W) + \frac{1}{2}W - \nabla_{\nabla f}W = -\Delta_f W + \frac{1}{2}W = 0,
    \end{align}
    and $f$ attains a minimum by Corollary \ref{cor:f-asymptotics-EGS}.
\end{proof}

\subsection{Integrating Killing fields to global symmetries} \label{subsec:integrate-KFs}

In this subsection, $(M^n,g,f)$ is either an asymptotically $\mathcal{S}$-cylindrical steady GRS or an asymptotically $(\mathcal{S},\bar{g}_\infty)$-conical expanding GRS. We also assume
\begin{align}
    d := \dim\ker_{L^2(e^f)}(L) < \infty.
\end{align}
We now use the results from the last two subsections to prove Theorem \ref{thm:general-symmetry-intro}.

Let $\mathcal{V}$ be the space of Killing fields given by Proposition \ref{prop:exact-KFs} in the steady case, or Proposition \ref{prop:exact-KFs-EGS} in the expanding case. Let $\mathfrak{g}$ be the Lie subalgebra of $\mathfrak{I}(M,g)$ generated by $\mathcal{V}$. Then
\begin{align}
    \dim\mathfrak{g} \geq \dim\mathcal{V} \geq \bar{d}-d.
\end{align}

\begin{corollary} \label{cor:g-jacobi}
    Every $X \in \mathfrak{g}$ satisfies $Xf = 0$ and $[X,\nabla f] = 0$.
\end{corollary}
\begin{proof}
    Propositions \ref{prop:exact-KFs} and \ref{prop:exact-KFs-EGS} give $Wf= 0$ for every $W \in \mathcal{V}$. This orthogonality is preserved under linear combinations and Lie brackets, so $Xf=0$ for every $X \in \mathfrak{g}$. Since $X$ is Killing, $[X,\nabla f] = \nabla(Xf) = 0$.
\end{proof}

By the Myers--Steenrod theorem, the isometry group $\I(M,g)$ is a Lie group acting smoothly on $M$; see \cite{isom-orb} for the orbifold case. Denote by $\I^0(M,g)$ its identity component, and let $G_* \subset \I^0(M,g)$ be the unique connected Lie subgroup with Lie algebra $\mathfrak{g}$. For sufficiently large $\rho$, let $g_\rho$ be the induced metric on the regular hypersurface $\Sigma_\rho = \{f=\rho\}$. We use the identifications $F_\rho: \mathcal{S} \to \Sigma_\rho$ fixed towards the end of \S\ref{subsec:grs-conventions} and \S\ref{subsec:egs-conventions}.

\begin{proposition} \label{prop:G-symmetry-group}
    There is a compact, connected Lie group $G$ acting smoothly, effectively, and isometrically on $(M,g)$, with $\dim G \geq \bar{d}-d$. Its action preserves $f$, and thus preserves $\Sigma_\rho$ for all sufficiently large $\rho$. Restriction to any such $\Sigma_\rho$ identifies $G$ with a closed Lie subgroup of $\I^0(\Sigma_\rho,g_\rho)$. Moreover, $G$ embeds as a closed Lie subgroup of $\I^0(\mathcal{S},\bar{g}_\infty)$, and every sufficiently far level set $\Sigma_\rho$ is $G$-equivariantly diffeomorphic to $\mathcal{S}$.
\end{proposition}
\begin{proof}
    We take $G$ to be the closure
    \begin{align}
        G := \overline{G_*} \subset \I(M,g),
    \end{align}
    which is a closed Lie subgroup of $\I(M,g)$. Thus its action on $(M,g)$ is smooth, effective, and isometric. Moreover,
    \begin{align}
        \dim G \geq \dim G_* = \dim\mathfrak{g} \geq \dim\mathcal{V} \geq \bar{d}-d.
    \end{align}
    Since $G_*$ is connected, so is $G$. By Corollary \ref{cor:g-jacobi}, every $X \in \mathfrak{g}$ satisfies $Xf=0$. Thus, the action of $G_*$ preserves $f$, and by continuity, so does the action of $G$.

    To prove compactness, fix $p \in \Sigma_\rho$ for some sufficiently large $\rho$. Recall that the action of a closed subgroup of $\I(M,g)$ is proper (see \cite{alexandrino-bettiol}*{Proposition 3.62}, also \cite{isom-orb} for the orbifold version); that is, the map
    \begin{align}
        G \times M \to M \times M, \qquad (\gamma,x) \mapsto (\gamma(x),x)
    \end{align}
    is proper. Since $G = \{\gamma \in G : \gamma \cdot p \in \Sigma_\rho \}$ is homeomorphic to the preimage of $\Sigma_\rho \times \{p\}$, it follows that $G$ is compact.

    Since $G$ preserves $g$ and $f$, it also preserves the vector field $\frac{\del}{\del\rho} F_\rho = \frac{\nabla f}{|\nabla f|^2} \circ F_\rho$. Thus the transported action on $\mathcal{S}$, given by
    \begin{align}
        \gamma \mapsto \bar{\gamma} := F_\rho^{-1} \circ \gamma \circ F_\rho, \qquad \gamma \in G,
    \end{align}
    is independent of $\rho$. This furnishes a $G$-equivariant diffeomorphism between $\Sigma_\rho$ and $\mathcal{S}$. The transported action preserves each metric $F_\rho^*g_\rho$, and hence its rescaled limit $\bar{g}_\infty$ by Lemmas \ref{lem:g-closeness} and \ref{lem:g-closeness-EGS}.

    If $\gamma \in G$ restricts to the identity on one $\Sigma_\rho$, then $\bar{\gamma} = \mathrm{id}_{\mathcal{S}}$. Independence of $\rho$ implies that $\gamma$ acts as the identity on the end. Thus it acts as the identity on all of $M$, because $M$ is connected and so isometries are uniquely determined by their value and differential at a single point. Effectiveness of the action then implies $\gamma$ is the identity element in $G$. Thus, both the restriction map $G \to \I(\Sigma_\rho, g_\rho)$ and the transported action $G \to \I(\mathcal{S},\bar{g}_\infty)$ are injective. Since $G$ is compact and connected, these maps identify $G$ with closed Lie subgroups of their respective identity components.
\end{proof}

For homogeneous links $(\mathcal{S}, \bar{g}_\infty)$, recall the number $d_*(\mathcal{S},\bar{g}_\infty)$ defined in Theorem \ref{thm:general-symmetry-intro}.
\begin{corollary} \label{cor:G-transitive-action}
    Suppose $(\mathcal{S},\bar{g}_\infty)$ is homogeneous. If $d \leq d_*(\mathcal{S}, \bar{g}_\infty)$, then $G$ acts transitively on $\Sigma_\rho$ for all sufficiently large $\rho$.
\end{corollary}
\begin{proof}
    Proposition \ref{prop:G-symmetry-group} identifies $G$ with a closed connected subgroup of $\I^0(\mathcal{S},\bar{g}_\infty)$ satisfying
    \begin{align}
        \dim G \geq \bar{d}-d \geq \bar{d}-d_*(\mathcal{S},\bar{g}_\infty).
    \end{align}
    Thus, $G$ acts transitively on $\mathcal{S}$ by the definition of $d_*(\mathcal{S},\bar{g}_\infty)$. Since every sufficiently far level set $\Sigma_\rho$ is $G$-equivariantly diffeomorphic to $\mathcal{S}$, $G$ also acts transitively on $\Sigma_\rho$.
\end{proof}

\begin{proof}[Proof of Theorem \ref{thm:general-symmetry-intro}]
    Part (a) of the theorem follows from Proposition \ref{prop:G-symmetry-group}, and part (b) follows from Corollary \ref{cor:G-transitive-action}.
\end{proof}

\subsection{Global structure of cohomogeneity one orbifolds}

In applications, we will use Theorem \ref{thm:general-symmetry-intro} in conjunction with the structure theory of cohomogeneity one actions. We now recall the relevant facts from that theory; for a more thorough account, see \cite{alexandrino-bettiol}*{\S 6.3} for manifolds and \cite{gonzalez-alvaro-phd}*{\S 2} for orbifolds.

Let $G$ be a compact Lie group acting smoothly on an orbifold $M$. The action is \emph{effective} if the kernel of $G \to \operatorname{Diff}(M)$ is trivial, and \emph{almost effective} if its kernel is finite. The \emph{isotropy subgroup} of a point is its stabilizer. On an open dense subset of $M$, the isotropy subgroups are conjugate to a fixed subgroup $H$, called a \emph{principal isotropy subgroup}. Orbits with isotropy subgroup conjugate to $H$ are called \emph{principal}; all other orbits are called \emph{non-principal}. Principal orbits have maximal dimension among all orbits. A non-principal orbit is \emph{exceptional} if it has the same dimension as the principal orbits. The action is said to have \emph{cohomogeneity one} if the principal orbits have codimension one in $M$.

\begin{proposition} \label{prop:coho-one-global-structure}
    Let $(M,g)$ be a complete, connected, noncompact Riemannian orbifold, with one end and with orbifold singular set of codimension at least $4$. Suppose a compact, connected Lie group $G$ acts smoothly, almost effectively, and isometrically on $M$ with cohomogeneity one. Then:
    \begin{enumerate}[label=(\alph*)]
        \item The orbit space $M/G$ is homeomorphic to the interval $[0,\infty)$.
        \item There are closed subgroups $H \subsetneq K \subseteq G$, such that $H$ is a principal isotropy subgroup, the orbit corresponding to $0$ is $B \cong G/K$, while all other orbits are principal and equivariantly diffeomorphic to $G/H$. In particular, there is a $G$-equivariant diffeomorphism
        \begin{align} \label{eq:coho-one-metric-principal}
            M \setminus B \cong (0,\infty) \times G/H.
        \end{align}
        This identification may be chosen so that
        \begin{align}
            g = ds^2 + g_s,
        \end{align}
        where $s$ is the distance from $B$ and $g_s$ is a smooth family of $G$-invariant metrics on $G/H$.
        \item Set $q = \operatorname{codim}_M B \geq 1$. Then $K/H$ is a spherical space form of dimension $q-1$.
    \end{enumerate}
\end{proposition}

We also need the following result which determines the diffeomorphism type of $M$ in our cases of interest.

\begin{proposition} \label{prop:coho-one-diffeotype}
    In the setting of Proposition \ref{prop:coho-one-global-structure}, the following hold.
    \begin{enumerate}[label=(\alph*)]
        \item If $K/H$ is diffeomorphic to $\Sph^{q-1}$, where $q = \operatorname{codim}_M B$, then $M$ is a smooth manifold $G$-equivariantly diffeomorphic to the total space of the associated vector bundle $G \times_K \R^q$ over $G/K$.
        Here $K$ acts on $\R^q$ by radial extension of its left-multiplication action on $K/H \cong \Sph^{q-1}$.
        \item Suppose that $n = \dim M \geq 3$, $K=G$, and $G/H$ is diffeomorphic to a spherical space form $\Sph^{n-1}/\Gamma$. Then $M$ is diffeomorphic to $\R^n/\Gamma$.
    \end{enumerate}
\end{proposition}
\begin{proof}
    Suppose $K/H$ is diffeomorphic to $\Sph^{q-1}$. We claim that $M$ is a smooth manifold. First, every point $x \in M \setminus B$ belongs to a principal orbit, thus cannot be orbifold-singular. Otherwise, since $G$ acts by diffeomorphisms, the entire orbit of $x$ would consist of orbifold-singular points. As this orbit has codimension one, this would contradict the regularity assumption on $M$. Thus every point outside $B$ must be orbifold-regular. Meanwhile, since $B$ is a $G$-orbit, either every point of $B$ is orbifold-regular or every point of $B$ is orbifold-singular.

    If $q \leq 2$, then if every point of $B$ is orbifold-singular, this would again contradict the regularity assumption. So $M$ has no orbifold-singular points and thus is a smooth manifold.

    If $q \geq 3$, then by the orbifold slice theorem \cite{gonzalez-alvaro-phd}*{Theorem 2.25}, the normal orbifold slice at $p \in B$ is $\R^q/\Lambda$, where $\Lambda \subset O(q)$ is the local orbifold group and acts freely on $\Sph^{q-1}$, and $\Sph^{q-1}/\Lambda \cong K/H$. Since $K/H \cong \Sph^{q-1}$ by assumption, we have $\Sph^{q-1}/\Lambda \cong \Sph^{q-1}$. As $\Sph^{q-1}$ is simply-connected, this implies $\Lambda=\{1\}$. Thus $B$ consists only of orbifold-regular points, so $M$ is a smooth manifold.

    Having now established that $M$ is a smooth manifold, the standard structure theory of cohomogeneity-one manifolds now gives the conclusion of part (a).

    To establish (b), suppose $n = \dim M \geq 3$, $K=G$, and $G/H \cong \Sph^{n-1}/\Gamma$. Then $B$ consists of a single point $p$. Let $\Lambda \subset O(n)$ be its local orbifold group. For sufficiently small $\epsilon > 0$, the exponential map gives an orbifold chart $\exp: B_{2\epsilon}(0)/\Lambda \to M$ with $\exp(0) = p$, whose radial coordinate equals $s$ from Proposition \ref{prop:coho-one-global-structure}(b). Flowing out by the vector field $\nabla s$ extends this to a global orbifold diffeomorphism
    \begin{align} \label{eq:exp-global}
        \exp: \R^n/\Lambda \to M.
    \end{align}
    Moreover, coordinate spheres in $\R^n/\Lambda$ map diffeomorphically to level sets of $s$, so $\Sph^{n-1}/\Lambda \cong G/H \cong \Sph^{n-1}/\Gamma$. Since $n \geq 3$, a theorem of de Rham \cite{derham} shows that $\Sph^{n-1}/\Lambda$ and $\Sph^{n-1}/\Gamma$ are isometric when equipped with their standard round quotient metrics. A covering argument, using the fact that $\Sph^{n-1}$ is simply-connected, implies that $\Lambda$ and $\Gamma$ are conjugate in $O(n)$, i.e. $A\Lambda A^{-1} = \Gamma$ for some $A \in O(n)$. The linear map $A$ therefore induces an orbifold diffeomorphism $\R^n/\Lambda \cong \R^n/\Gamma$. Composing this with \eqref{eq:exp-global} proves (b).
\end{proof}

\section{Characterizations of Bryant solitons} \label{sec:round-link-classification}

In this section, we specialize our symmetry principle to round spherical asymptotic links and prove Theorem \ref{thm:bryant-classification-intro}.

Recall that $g_{\Sph^{n-1}}$ denotes the round metric of constant sectional curvature $1$. Its isometry group has identity component $\I^0(\Sph^{n-1},g_{\Sph^{n-1}}) = SO(n)$, which acts transitively with isotropy subgroup $SO(n-1)$. We begin with a well-known fact about its Lie algebra; see for instance \cite{kobayashi-groups}*{\S II.3}.

\begin{lemma} \label{lem:large-subalgebras-so}
    Let $n \geq 3$. The maximal dimension of a proper Lie subalgebra of $\mathfrak{so}(n)$ is $(n-1)(n-2)/2$ if $n=3$ or $n \geq 5$, and $4$ if $n=4$.
\end{lemma}

Recall the definition of $d_*$ from Theorem \ref{thm:general-symmetry-intro}.

\begin{corollary} \label{cor:round-sphere-d1}
    We have
    \begin{align} \label{eq:round-sphere-d1}
        d_*(\Sph^{n-1},g_{\Sph^{n-1}}) \geq \begin{cases}
            n-2 & \text{if } n=3 \text{ or } n \geq 5, \\
            1 & \text{if } n=4.
        \end{cases}
    \end{align}
\end{corollary}
\begin{proof}
    Write $\bar{d} = \dim SO(n) = \frac{n(n-1)}{2}$, and let $j$ be the quantity on the right-hand side of \eqref{eq:round-sphere-d1}. Let $G'$ be a closed connected Lie subgroup of $\I^0(\Sph^{n-1},g_{\Sph^{n-1}}) = SO(n)$ with dimension at least $\bar{d}-j$. Then its Lie algebra $\mathfrak{g}' \subseteq \mathfrak{so}(n)$ has dimension at least $\bar{d}-j$. By Lemma \ref{lem:large-subalgebras-so}, we must have $\mathfrak{g}' = \mathfrak{so}(n)$. Hence $G'=SO(n)$, which acts transitively on $\Sph^{n-1}$.
\end{proof}

\subsection{Proof of Theorem \ref{thm:bryant-classification-intro}} \label{subsec:bryant-classification-proof}

Throughout the remainder of this section, let $(M^n,g,f)$, $n \geq 3$, satisfy the hypotheses of Theorem \ref{thm:bryant-classification-intro}. Thus $(M,g,f)$ is either an asymptotically $\Sph^{n-1}$-cylindrical steady GRS or an asymptotically $(\Sph^{n-1}, \alpha g_{\Sph^{n-1}})$-conical expanding GRS with $\alpha>0$, and the kernel bound \eqref{eq:bryant-kernel-bound-intro} holds. We first identify the symmetry group provided by Theorem \ref{thm:general-symmetry-intro}, and its principal isotropy subgroup.

\begin{corollary} \label{cor:round-link-SOn-action}
    The group $SO(n)$ acts smoothly, effectively, and isometrically with cohomogeneity one on $(M,g)$. The action preserves $f$, and every sufficiently far level set of $f$ is an orbit with isotropy subgroup conjugate to the standard $SO(n-1) \subset SO(n)$.
\end{corollary}
\begin{proof}
    By Corollary \ref{cor:round-sphere-d1}, the bound \eqref{eq:bryant-kernel-bound-intro} implies $\dim\ker_{L^2(e^f)}(L) \leq d_*(\Sph^{n-1},\alpha g_{\Sph^{n-1}})$ for every $\alpha > 0$. Theorem \ref{thm:general-symmetry-intro} therefore gives a closed connected subgroup $G \subseteq \I^0(\Sph^{n-1}, \bar{g}_\infty) = SO(n)$ acting smoothly, effectively, and isometrically with cohomogeneity one on $(M,g)$, preserving $f$ and acting transitively on every sufficiently far level set $\Sigma_\rho$. Moreover,
    \begin{align} \label{eq:G-dimension-bound-bryant}
        \dim G \geq \frac{n(n-1)}{2} - \dim\ker_{L^2(e^f)}(L) \geq \begin{cases}
            \frac{(n-1)(n-2)}{2}+1 &\text{if } n=3 \text{ or } n \geq 5, \\
            5 &\text{if } n=4.
        \end{cases}
    \end{align}
    By Lemma \ref{lem:large-subalgebras-so}, we must have $G = SO(n)$. Theorem \ref{thm:general-symmetry-intro} also asserts that every sufficiently far level set is $G$-equivariantly diffeomorphic to the standard action on $(\mathcal{S},\bar{g}_\infty)$. Its isotropy subgroup is therefore conjugate to $SO(n-1)$.
\end{proof}

Let $G = SO(n)$ and $H = SO(n-1)$. By Corollary \ref{cor:round-link-SOn-action} and Proposition \ref{prop:coho-one-global-structure}, there is a unique non-principal orbit $B \cong G/K$, where $H \subsetneq K \subseteq G$ and $K/H$ is a spherical space form, such that
\begin{align}
    M \setminus B \cong (0,\infty) \times G/H, \qquad g = ds^2 + g_s, \qquad f = f(s),
\end{align}
where $s$ is the distance from $B$, and each $g_s$ is a $G$-invariant metric on $G/H \cong \Sph^{n-1}$. Such metrics bijectively correspond to $\operatorname{Ad}(H)$-invariant inner products on $\mathfrak{p}$, where $\mathfrak{p}$ is an $\operatorname{Ad}(H)$-invariant subspace satisfying $\mathfrak{so}(n) = \mathfrak{so}(n-1) \oplus \mathfrak{p}$. We can take
\begin{align}
    \mathfrak{p} = \left\{ \begin{pmatrix} 0 & v \\ -v^T & 0 \end{pmatrix} : v \in \R^{n-1} \right\},
\end{align}
and verify that $\operatorname{Ad}(H)$ restricts to the standard irreducible representation on $\mathfrak{p} \cong \R^{n-1}$. By Schur's lemma, $\mathfrak{p}$ admits a unique $\operatorname{Ad}(H)$-invariant inner product up to scaling, so every $G$-invariant metric on $G/H$ is a multiple of the round metric. Thus,
\begin{align} \label{eq:warped-rotationally-symmetric}
    g = ds^2 + a(s)^2 g_{\Sph^{n-1}}
\end{align}
for a smooth positive function $a$ on $(0,\infty)$. We next identify the possibilities for $K$.

\begin{lemma} \label{lem:round-K-classification}
    Let $n \geq 3$. Up to conjugacy, the closed subgroups $K \subseteq G$ such that $H \subsetneq K$ are as follows:
    \[
    \begin{array}{c|c|c}
    K & K/H & B\cong G/K\\ \hline \noalign{\vskip 1pt}
    S(O(n-1) \times O(1)) & \Sph^0 & \mathbb{RP}^{n-1} \\
    SO(n) & \Sph^{n-1} & \{\mathrm{point}\}
    \end{array}
    \]
\end{lemma}
\begin{proof}
    Write $\mathfrak{h} \subseteq \mathfrak{k} \subseteq \mathfrak{so}(n)$ for the corresponding Lie algebras. As above, write
    \begin{align}
        \mathfrak{so}(n) = \mathfrak{h} \oplus \mathfrak{p}, \qquad \mathfrak{p} = \left\{ \begin{pmatrix} 0 & v \\ -v^T & 0 \end{pmatrix} : v \in \R^{n-1} \right\},
    \end{align}
    so $\operatorname{Ad}(H)$ restricts to the standard irreducible representation on $\mathfrak{p} \cong \R^{n-1}$. Since $\mathfrak{h} \subseteq \mathfrak{k}$, we have $\mathfrak{k} = \mathfrak{h} \oplus (\mathfrak{k} \cap \mathfrak{p})$. Since $\mathfrak{k} \cap \mathfrak{p}$ is $\operatorname{Ad}(H)$-invariant, it is either zero or $\mathfrak{p}$. Thus $\mathfrak{k} = \mathfrak{h}$ or $\mathfrak{k} = \mathfrak{so}(n)$.

    If $\mathfrak{k} = \mathfrak{so}(n)$, then $K = G = SO(n)$. If $\mathfrak{k} = \mathfrak{so}(n-1)$, then $K^0 = H$, so $K$ normalizes $H$. Thus $H \subsetneq K \subseteq N_G(H)$. Since $N_G(H) = S(O(n-1) \times O(1))$ contains $H$ as an index $2$ subgroup, these inclusions force $K=N_G(H)$.
\end{proof}

We now finish the proof of Theorem \ref{thm:bryant-classification-intro}. By the above discussion and Lemma \ref{lem:round-K-classification}, either $K = S(O(n-1) \times O(1))$ or $K = SO(n)$.

First suppose that $K = S(O(n-1) \times O(1))$.
By the smoothness conditions for tensors across an exceptional orbit (see for instance \cites{eschenburg-wang,verdiani-ziller}), we need $a$ and $f$ to extend evenly across $s=0$, with $a(0) > 0$. In particular,
\begin{align}
    a'(0) = f'(0) = 0.
\end{align}
Recall the GRS equation $\Ric = \nabla^2 f + \lambda g$; we take $\lambda = 0$ in the steady case, and $\lambda = -\frac{1}{2}$ in the expanding case after normalization. For the metric \eqref{eq:warped-rotationally-symmetric}, the tangential and radial components of the GRS equation at $s=0$ are computed to be
\begin{align}
    \lambda = \frac{n-2}{a(0)^2} - \frac{a''(0)}{a(0)}, \qquad
    f''(0) + \lambda = -(n-1) \frac{a''(0)}{a(0)}.
\end{align}
Combining these, we get
\begin{align} \label{eq:fpp-leq-0}
    f''(0) = (n-2)\lambda - \frac{(n-1)(n-2)}{a(0)^2} < 0.
\end{align}
On the other hand, since $f$ is constant on $B$ and $f'(0) = 0$, the traced GRS equation $\Delta f = R-n\lambda$ and the inequality $R \geq n\lambda$ (by \eqref{eq:R0} and \eqref{eq:EGS-1}) give
\begin{align}
    0 \leq R|_B - n\lambda = \Delta f|_B = f''(0) + (n-1)\frac{a'(0)}{a(0)} f'(0) = f''(0),
\end{align}
contradicting \eqref{eq:fpp-leq-0}.

Thus $K = SO(n)$, and $B$ is a point. By Proposition \ref{prop:coho-one-diffeotype}(b), $M$ is a smooth manifold diffeomorphic to $\R^n$, and $g$ is rotationally symmetric. By the uniqueness of complete rotationally symmetric solitons on $\R^n$ \cite{bryant}, it follows that $(M^n,g,f)$ is the Bryant steady soliton in the steady case, and the Bryant expanding soliton with asymptotic link $(\Sph^{n-1}, \alpha g_{\Sph^{n-1}})$ in the expanding case. This concludes the proof of Theorem \ref{thm:bryant-classification-intro}.

\section{Symmetry and classification with quotient-Berger links} \label{sec:berger-link-classification}

In this section, we prove Theorem \ref{thm:berger-symmetry-intro} and Corollary \ref{cor:appleton-characterization-intro}. We first determine the symmetry groups and invariant metrics of quotient-Berger spheres, then apply the symmetry principle and identify the possible non-principal orbits.

For a Lie group $G$ and a subgroup $H \subset G$, write $N_G(H)$ and $N_G^0(H)$ for the normalizer of $H$ in $G$, and its identity component. We will use the following standard fact without proof:
\begin{lemma} \label{lem:quotient-isom-group}
    Let $(\tilde{M},\tilde{g})$ be a simply-connected Riemannian manifold, and $\Gamma \subset \I(\tilde{M},\tilde{g})$ a finite subgroup of isometries acting freely. Then the quotient manifold $(M,g)$ has isometry group
    \begin{align}
        \I(M,g) = N_{\I(\tilde{M},\tilde{g})}(\Gamma)/\Gamma.
    \end{align}
\end{lemma}

\subsection{Isometry groups and invariant metrics of quotient-Berger spheres}

Recall from \S\ref{subsec:uniq-appl} the Hopf fibration $\pi: \Sph^{2m-1} \to \CP^{m-1}$ and the lens spaces $L_{m,k} = \Sph^{2m-1}/\Z_k$, where $\Z_k$ acts by scalar multiplication on $\Sph^{2m-1} \subset \C^m$. The Hopf 1-form $\sigma$ and the Fubini--Study metric $g_{\mathrm{FS}}$ of holomorphic sectional curvature $4$ yield
\begin{align}
    g_{\Sph^{2m-1}} = \sigma^2 + \pi^*g_{\mathrm{FS}}.
\end{align}
We write $\tilde{g}_{a,b} = a^2 \sigma^2 + b^2 \pi^*g_{\mathrm{FS}}$ for the Berger metric on $\Sph^{2m-1}$ and $g_{a,b}$ for its quotient on $L_{m,k}$.

\begin{lemma} \label{lem:berger-isometry-groups}
    Let $m \geq 2$ and $a,b>0$. Then
    \begin{align} \label{eq:berger-sphere-isom}
        \I(\Sph^{2m-1}, \tilde{g}_{a,b}) \cong \begin{cases}
            O(2m) &\text{if } a=b, \\
            U(m) \rtimes \Z_2 &\text{if } a \neq b,
        \end{cases}
    \end{align}
    where the additional involution in the second case is complex conjugation.

    If $k \geq 3$, then for every $a,b > 0$,
    \begin{align} \label{eq:Lmk-gab-isometry-group}
        \I(L_{m,k}, g_{a,b}) \cong (U(m) \rtimes \Z_2)/\Z_k,
    \end{align}
    so in particular,
    \begin{align}
        \I^0(L_{m,k}, g_{a,b}) = U(m)/\Z_k.
    \end{align}

    For every $k \geq 1$, the natural actions of $U(m)$ and $SU(m)$ on $(L_{m,k}, g_{a,b})$ are transitive and isometric, with isotropy subgroups $U(m-1) \times \Z_k$ and $S(U(m-1) \times \Z_k)$ respectively. These subgroups are embedded block-diagonally.
\end{lemma}
\begin{proof}
    Define
    \begin{align}
        \hat{\I} = \{A \in O(2m) : AJ = \pm JA\},
    \end{align}
    where $J \in O(2m)$ is the matrix associated with multiplication by $i$ in $\C^m \cong \R^{2m}$. One checks that $\hat{\I}$ is the disjoint union
    \begin{align}
        \hat{\I} = U(m) \sqcup K U(m),
    \end{align}
    where $K$ is complex conjugation. Since $KAK^{-1} = \bar{A}$ for $A \in U(m)$, we obtain
    \begin{align}
        \hat{\I} \cong U(m) \rtimes \Z_2.
    \end{align}

    If $a \neq b$, it is well-known (e.g. \cite{torralbo-urbano}*{\S 2}) that $\I(\Sph^{2m-1}, \tilde{g}_{a,b}) = \hat{\I}$. Also recall that $(L_{m,k},g_{a,b})$ is $(\Sph^{2m-1},\tilde{g}_{a,b})$ quotiented by the scalar multiplication action of $\Z_k$, which is central in the identity component of $\hat{\I}$. Lemma \ref{lem:quotient-isom-group} thus gives
    \begin{align}
        \I(L_{m,k}, g_{a,b}) \cong N_{\hat{\I}}(\Z_k)/\Z_k.
    \end{align}
    The central copy of $\Z_k$ in $U(m)$ is normalized by $K$, since $KAK^{-1} = \bar{A} = A^{-1}$ whenever $A$ belongs to this central $\Z_k$. It follows that $\hat{\I}$ normalizes $\Z_k$, so
    \begin{align}
        \I(L_{m,k},g_{a,b}) \cong \hat{\I}/\Z_k = (U(m) \rtimes \Z_2)/\Z_k,
    \end{align}
    proving \eqref{eq:Lmk-gab-isometry-group} in the $a \neq b$ case.

    Now assume $a=b$, so $(\Sph^{2m-1}, \tilde{g}_{a,b})$ is a round sphere and $\I(\Sph^{2m-1},\tilde{g}_{a,b}) = O(2m)$. Since $(L_{m,k},g_{a,b})$ is $(\Sph^{2m-1},\tilde{g}_{a,b})$ quotiented by the scalar action of $\Z_k$, Lemma \ref{lem:quotient-isom-group} gives
    \begin{align} \label{eq:Lmk-gab-isometry-round}
        \I(L_{m,k}, g_{a,b}) \cong N_{O(2m)}(\Z_k)/\Z_k.
    \end{align}
    Set $\theta = 2\pi/k$. The matrix $Z = \exp(\theta J) = \cos(\theta) I + \sin(\theta) J$ generates this copy of $\Z_k$. If $A \in N_{O(2m)}(\Z_k)$, then $AZA^{-1} = Z^\alpha$ for some $\alpha \in \Z$. Since $\tr J = 0$, taking traces gives
    \begin{align}
        2m \cos(\theta) = \tr(Z) = \tr(AZA^{-1}) = \tr(Z^\alpha) = 2m\cos(\alpha\theta).
    \end{align}
    Therefore $\alpha \equiv \pm 1 \pmod k$, and $AZA^{-1} = Z^{\pm 1}$. Writing this equation out in terms of $I$ and $J$, then simplifying, we get $\sin(\theta)(AJ \pm JA) = 0$. Since $k \geq 3$, $\sin(\theta) \neq 0$, so $AJ = \pm JA$. Thus, $N_{O(2m)}(\Z_k) = \widehat{\I}$. By \eqref{eq:Lmk-gab-isometry-round}, it follows that
    \begin{align}
        \I(L_{m,k},g_{a,b}) \cong \hat{\I}/\Z_k = (U(m) \rtimes \Z_2)/\Z_k,
    \end{align}
    proving \eqref{eq:Lmk-gab-isometry-group} in the $a=b$ case.

    Finally, the actions of $U(m)$ and $SU(m)$ on $\Sph^{2m-1}$ are transitive and, by \eqref{eq:berger-sphere-isom}, are isometric for every Berger metric. The $U(m)$-action commutes with the scalar $\Z_k$-action and thus descends to a transitive isometric action on $L_{m,k}$. The stabilizer of $[(0,\ldots,0,1)] \in L_{m,k}$ is
    \begin{align}
        \{\operatorname{diag}(B,\zeta) : B \in U(m-1), \, \zeta^k = 1\} = U(m-1) \times \Z_k.
    \end{align}
    Intersecting this subgroup with $SU(m)$ gives $S(U(m-1) \times \Z_k)$, as claimed.
\end{proof}

We next record the analog of Lemma \ref{lem:large-subalgebras-so} that will be needed.
\begin{lemma} \label{lem:large-subalgebras-u}
    Let $m \geq 2$. If $\mathfrak{g}$ is a Lie subalgebra of $\mathfrak{u}(m)$ satisfying $\dim\mathfrak{g} \geq m^2-2m+3$, then either $\mathfrak{g} = \mathfrak{su}(m)$ or $\mathfrak{g}=\mathfrak{u}(m)$. If $m=4$, we also allow $\mathfrak{g} = \mathfrak{sp}(2) \oplus \mathfrak{u}(1)$, up to conjugacy in $\mathfrak{u}(4)$.
\end{lemma}
\begin{proof}
    Consider the canonical representation of $\mathfrak{g}$ on $\C^m$. If this were reducible, then after conjugating by an element of $U(m)$, $\mathfrak{g}$ would be contained in $\mathfrak{u}(r) \oplus \mathfrak{u}(m-r)$ for some $1 \leq r \leq m-1$. Thus, $\dim\mathfrak{g} \leq r^2 + (m-r)^2 \leq (m-1)^2 + 1 = m^2-2m+2$, contradicting the assumed dimension bound. Therefore, $\mathfrak{g}$ acts irreducibly on $\C^m$.

    Since $\mathfrak{g}$ is compact and thus reductive, $\mathfrak{g} = [\mathfrak{g},\mathfrak{g}] \oplus Z(\mathfrak{g})$. We have $[\mathfrak{g},\mathfrak{g}] \subseteq [\mathfrak{u}(m),\mathfrak{u}(m)] = \mathfrak{su}(m)$, whereas $Z(\mathfrak{g}) \subseteq \R i I_m$ by irreducibility and Schur's lemma. Thus,
    \begin{align} \label{eq:g-subalgebra-decomp}
        \mathfrak{g} = \mathfrak{g}^{\mathrm{tf}} \oplus \mathfrak{z} := [\mathfrak{g},\mathfrak{g}] \oplus (\mathfrak{g} \cap \R i I_m).
    \end{align}
    Since $\dim\mathfrak{z} \leq 1$, we have $\dim\mathfrak{g}^{\mathrm{tf}} \geq \dim\mathfrak{g} - 1$. Also note that $\mathfrak{g}^{\mathrm{tf}}$ is compact, hence reductive, and acts irreducibly on $\C^m$. Hence, its complexification $\mathfrak{g}^{\mathrm{tf}}_{\C} \subset \mathfrak{sl}(m,\C)$ is also reductive, acts irreducibly on $\C^m$, and satisfies
    \begin{align} \label{eq:g-tf-dimension-lb}
        \dim_{\C}\mathfrak{g}^{\mathrm{tf}}_{\C} \geq m^2 - 2m + 2.
    \end{align}
    By a dimensional argument, one checks that the canonical representation of $\mathfrak{g}^{\mathrm{tf}}_{\C}$ on $\C^m$ does not decompose as a tensor product. Combined with its reductivity, this implies $\mathfrak{g}^{\mathrm{tf}}_{\C}$ is a simple Lie algebra over $\C$ possessing a faithful irreducible representation of dimension $m$. Also being a subalgebra of $\mathfrak{sl}(m,\C)$ and satisfying \eqref{eq:g-tf-dimension-lb}, one checks using the classification of simple Lie algebras and their smallest irreducible representations that the only possibilities are $\mathfrak{g}^{\mathrm{tf}}_{\C} \cong \mathfrak{sl}(m,\C)$, and if $m=4$, $\mathfrak{g}^{\mathrm{tf}}_{\C} \cong \mathfrak{sp}(2,\C)$. We divide into these two cases.

    If $\mathfrak{g}^{\mathrm{tf}}_{\C} \cong \mathfrak{sl}(m,\C)$, then $\mathfrak{g}^{\mathrm{tf}}$ is a compact real form of $\mathfrak{sl}(m,\C)$, which is necessarily $\mathfrak{su}(m)$. Using \eqref{eq:g-subalgebra-decomp}, we see that $\mathfrak{g} = \mathfrak{su}(m)$ if $\dim\mathfrak{z} = 0$, or $\mathfrak{g} = \mathfrak{u}(m)$ if $\dim\mathfrak{z} = 1$.

    If $m=4$ and $\mathfrak{g}^{\mathrm{tf}}_{\C} \cong \mathfrak{sp}(2,\C)$, then $\mathfrak{g}^{\mathrm{tf}}$ is a compact real form of $\mathfrak{sp}(2,\C)$, which is necessarily $\mathfrak{sp}(2)$. Using \eqref{eq:g-subalgebra-decomp}, we see that up to conjugacy in $\mathfrak{u}(4)$, $\mathfrak{g} = \mathfrak{sp}(2)$ if $\dim\mathfrak{z} = 0$, or $\mathfrak{g} = \mathfrak{sp}(2) \oplus \mathfrak{u}(1)$ if $\dim\mathfrak{z} = 1$. However, we can rule out the former case since $\dim\mathfrak{sp}(2) = 10$ while we have assumed that $\dim\mathfrak{g} \geq m^2 - 2m + 3 = 11$.
\end{proof}

The next lemma records the invariant metrics for the homogeneous presentations of $L_{m,k}$ given in Lemma \ref{lem:berger-isometry-groups}.
\begin{lemma} \label{lem:invariant-berger-metrics}
    \begin{enumerate}[label=(\alph*)]
        \item For $m \geq 2$ and $k \geq 1$, every $U(m)$-invariant metric on $U(m)/(U(m-1) \times \Z_k) \cong L_{m,k}$ is $g_{a,b}$ for some $a,b > 0$.
        \item For $m \geq 3$ and $k \geq 1$, every $SU(m)$-invariant metric on $SU(m)/S(U(m-1) \times \Z_k) \cong L_{m,k}$ is $g_{a,b}$ for some $a,b>0$.
    \end{enumerate}
\end{lemma}
\begin{proof}
    We handle case (a); case (b) is similar.

    Let $G = U(m)$ and $H = U(m-1) \times \Z_k$. Then $G$-invariant metrics on $G/H$ bijectively correspond to $\operatorname{Ad}(H)$-invariant inner products on $\mathfrak{p}$, where $\mathfrak{g} = \mathfrak{h} \oplus \mathfrak{p}$ is a fixed $\operatorname{Ad}(H)$-invariant decomposition. We can take
    \begin{align}
        \mathfrak{p} = \mathfrak{p}_1 \oplus \mathfrak{p}_2, \qquad \mathfrak{p}_1 = \left\{ \begin{pmatrix} 0 & 0 \\ 0 & it \end{pmatrix} : t \in \R \right\}, \qquad \mathfrak{p}_2 = \left\{ \begin{pmatrix} 0 & z \\ -z^* & 0 \end{pmatrix} : z \in \C^{m-1} \right\}.
    \end{align}
    The adjoint action of $\operatorname{diag}(B,\zeta) \in H$ fixes $\mathfrak{p}_1$ pointwise, hence $\mathfrak{p}_1$ is irreducible, while on $\mathfrak{p}_2$ it acts as $z \mapsto \zeta^{-1}Bz$. Since $U(m-1) \subset H$ acts irreducibly on $\C^{m-1}$ by left-multiplication, it follows that $\mathfrak{p}_2$ is also irreducible. Under the standard $G/H \cong L_{m,k}$, these summands correspond to the vertical and horizontal distributions of the Hopf fibration, respectively.

    By Schur's lemma, all $\operatorname{Ad}(H)$-invariant inner products on $\mathfrak{p}$ are thus of the form
    \begin{align}
        a^2 \inner{\cdot}{\cdot}_{\mathfrak{p}_1} + b^2 \inner{\cdot}{\cdot}_{\mathfrak{p}_2},
    \end{align}
    where $a,b>0$, and each $\inner{\cdot}{\cdot}_{\mathfrak{p}_i}$ is induced by the round quotient metric $g_{1,1}$. Tracing back identifications, this corresponds to the metric $g_{a,b}$ on $L_{m,k}$. The lemma follows.
\end{proof}

We conclude this subsection with a lower bound for $d_*(L_{m,k},g_{a,b})$, as defined in the statement of Theorem \ref{thm:general-symmetry-intro}.
\begin{corollary} \label{cor:berger-link-d1}
    Let $m \geq 2$, $k \geq 3$, and $a,b>0$. Then $d_*(L_{m,k},g_{a,b}) \geq 2m-3$.
\end{corollary}
\begin{proof}
    By Lemma \ref{lem:berger-isometry-groups}, the identity component of the isometry group is $\I^0(L_{m,k},g_{a,b}) = U(m)/\Z_k$, which has dimension $m^2$, and its Lie algebra is $\mathfrak{u}(m)$. Let $G'$ be a closed connected Lie subgroup of $\I^0(L_{m,k},g_{a,b})$ satisfying
    \begin{align}
        \dim G' \geq m^2-(2m-3).
    \end{align}
    Then its Lie algebra $\mathfrak{g}' \subseteq \mathfrak{u}(m)$ has dimension $\dim\mathfrak{g}' \geq m^2-(2m-3)$. By Lemma \ref{lem:large-subalgebras-u}, it follows that the group $G' \subseteq U(m)/\Z_k$ is conjugate to the image of $SU(m)$ or $U(m)$, or, when $m=4$, of $Sp(2) \cdot U(1)$, under the quotient map $U(m) \to U(m)/\Z_k$. Each of these groups acts transitively on $\Sph^{2m-1}$, so $G'$ acts transitively on the quotient $L_{m,k}$. Thus, the assertion follows from the definition of $d_*(L_{m,k},g_{a,b})$.
\end{proof}

\subsection{Proofs of Theorem \ref{thm:berger-symmetry-intro} and Corollary \ref{cor:appleton-characterization-intro}} \label{subsec:berger-cohomogeneity-one}

Throughout this subsection, let $(M^{2m},g,f)$ satisfy the hypotheses of Theorem \ref{thm:berger-symmetry-intro}. We first identify the symmetry group provided by Theorem \ref{thm:general-symmetry-intro}, and its principal isotropy subgroup.

\begin{corollary} \label{cor:berger-global-action}
    Set
    \begin{align}
        G = \begin{cases}
            U(2) &\text{if } m=2, \\
            SU(m) &\text{if } m \geq 3.
        \end{cases}
    \end{align}
    Then $G$ acts smoothly, almost effectively, and isometrically with cohomogeneity one on $(M,g)$. The action preserves $f$, and every sufficiently far level set of $f$ is an orbit with isotropy subgroup conjugate to
    \begin{align}
        H = \begin{cases}
            U(1) \times \Z_k &\text{if } m=2, \\
            S(U(m-1) \times \Z_k) &\text{if } m \geq 3.
        \end{cases}
    \end{align}
\end{corollary}
\begin{proof}
    Recall that $\bar{g}_\infty$ from Theorem \ref{thm:general-symmetry-intro} is specified as part of the asymptotically conical hypothesis in the expanding case; thus $\bar{g}_\infty = g_{\alpha,\beta}$. In the steady case, $\bar{g}_\infty = g_{1,1}$ is the metric on $L_{m,k}$ with constant sectional curvature $1$. In both cases, Lemma \ref{lem:berger-isometry-groups} gives
    \begin{align}
        \I^0(L_{m,k}, \bar{g}_\infty) = U(m)/\Z_k, \qquad \bar{d}=m^2.
    \end{align}
    By Corollary \ref{cor:berger-link-d1} and \eqref{eq:berger-kernel-bound-intro}, we have $\dim\ker_{L^2(e^f)}(L) \leq d_*(L_{m,k}, \bar{g}_\infty)$. Theorem \ref{thm:general-symmetry-intro} therefore gives a closed connected subgroup $G_0 \subseteq U(m)/\Z_k$ acting smoothly, effectively, and isometrically with cohomogeneity one on $(M,g)$, preserving $f$. Every sufficiently far level set is $G_0$-equivariantly diffeomorphic to $L_{m,k}$. Moreover,
    \begin{align} \label{eq:berger-global-group-dimension}
        \dim G_0 \geq m^2 - \dim\ker_{L^2(e^f)}(L) \geq \begin{cases}
            4 &\text{if } m=2, \\
            14 &\text{if } m=4, \\
            m^2 - 2m + 3 &\text{if } m=3 \text{ or } m \geq 5.
        \end{cases}
    \end{align}
    Writing $\mathfrak{g}_0 \subset \mathfrak{u}(m)$ for the Lie algebra of $G_0$, Lemma \ref{lem:large-subalgebras-u} implies
    \begin{align}
        \mathfrak{g}_0 = \begin{cases}
            \mathfrak{u}(2) &\text{if } m=2, \\
            \mathfrak{su}(m) \text{ or } \mathfrak{u}(m) &\text{if } m \geq 3.
        \end{cases}
    \end{align}
    When $m=4$, the additional possibility $\mathfrak{sp}(2) \oplus \mathfrak{u}(1)$ is excluded because it has dimension $11$, whereas $\dim\mathfrak{g}_0 = \dim G_0 \geq 14$.

    In the $m=2$ case, since $G_0 \subseteq U(2)/\Z_k$ is connected and $\mathfrak{g}_0 = \mathfrak{u}(2)$, it follows that $G_0$ is the full $U(2)/\Z_k$. Composing the quotient map $U(2) \to U(2)/\Z_k$ with the action of $G_0$, it follows that $U(2)$ acts smoothly, almost effectively, and isometrically on $(M,g)$. Its action on every sufficiently far level set is equivariantly diffeomorphic to the transitive $U(2)$-action on $L_{m,k}$, whose isotropy subgroup is $H = U(1) \times \Z_k$ as indicated in Lemma \ref{lem:berger-isometry-groups}.

    In the $m \geq 3$ case, since $G_0 \subseteq U(m)/\Z_k$ is connected and $\mathfrak{g}_0 \supseteq \mathfrak{su}(m)$, it follows that $G_0$ contains the image of $SU(m)$ under the map $U(m) \to U(m)/\Z_k$. Composing the restricted quotient map $SU(m) \to U(m)/\Z_k$ (whose image is contained in $G_0$) with the action of $G_0$, and arguing similarly as above, the claims for $m \geq 3$ follow.
\end{proof}

Let $G$ and $H$ be given by Corollary \ref{cor:berger-global-action}. By the corollary and Proposition \ref{prop:coho-one-global-structure}, there is a unique non-principal orbit $B \cong G/K$, where $H \subsetneq K \subseteq G$ and $K/H$ is a spherical space form, such that
\begin{align} \label{eq:berger-principal-metric-1}
    M \setminus B \cong (0,\infty) \times G/H, \qquad g = ds^2 + g_s, \qquad f=f(s),
\end{align}
where $s$ is the distance from $B$, and each $g_s$ is a $G$-invariant metric on $G/H \cong L_{m,k}$. By Lemma \ref{lem:invariant-berger-metrics},
\begin{align} \label{eq:berger-principal-metric-2}
    g = ds^2 + g_{a(s),b(s)},
\end{align}
for smooth positive functions $a,b$ on $(0,\infty)$. The asymptotically cylindrical and conical assumptions on $(M,g,f)$ imply that as $s \to \infty$,
\begin{equation} \label{eq:ab-asymptotics}
    \frac{a(s)}{b(s)} \to 1 \text{ in the steady case}, \qquad \frac{a(s)}{s} \to \alpha, \quad \frac{b(s)}{s} \to \beta \text{ in the expanding case}.
\end{equation}
We next identify the possibilities for $K$.

\begin{lemma} \label{lem:berger-endpoint-isotropy}
    Let $m \geq 2$ and $k \geq 3$, and let $H \subset G$ be the groups in Corollary \ref{cor:berger-global-action}, so that
    \begin{align}
        (G,H) = \begin{cases}
            (U(2), U(1) \times \Z_k) &\text{if } m=2, \\
            (SU(m), S(U(m-1) \times \Z_k)) &\text{if } m \geq 3.
        \end{cases}
    \end{align}
    Up to conjugacy, the closed subgroups $K \subseteq G$ such that $H \subsetneq K$ and $K/H$ is a spherical space form are as follows:
    \begin{enumerate}[label=(\alph*)]
        \item If $m=2$:
        \begin{align*}
            \begin{array}{c|c|c}
                K & K/H & B \cong G/K \\ \hline \noalign{\vskip 1pt}
                U(1) \times \Z_{2k}
                    & \Sph^0 & L_{2,2k} \\
                U(1) \times U(1)
                    & \Sph^1 & \CP^1 \\
                U(2)
                    & L_{2,k} & \{\mathrm{point}\}
            \end{array}
        \end{align*}
        \item If $m \geq 3$:
        \begin{align*}
            \begin{array}{c|c|c}
                K & K/H & B \cong G/K \\ \hline \noalign{\vskip 1pt}
                S(U(m-1) \times \Z_{2k})
                    & \Sph^0 & L_{m,2k} \\
                S(U(m-1) \times U(1))
                    & \Sph^1 & \CP^{m-1} \\
                SU(m)
                    & L_{m,k} & \{\mathrm{point}\}
            \end{array}
        \end{align*}
    \end{enumerate}
\end{lemma}
\begin{proof}
    If $K = G$, then $K/H \cong L_{m,k}$ and $G/K$ is a point, giving the last possibility in each case. Henceforth suppose that $K$ is a proper subgroup of $G$.

    We first focus on part (b) of the lemma, so $G = SU(m)$ with $m \geq 3$, and $H = S(U(m-1) \times \Z_k)$. The identity component $K^0$ of $K$ splits $\C^m$ into irreducible modules which are orthogonal with respect to the standard Hermitian metric. As $K^0$ contains the identity component of $H$, it contains the upper-left $SU(m-1)$, so there are two possibilities:
    \begin{enumerate}[label=(\roman*)]
        \item The decomposition is $\C^{m-1} \oplus \C^1$.
        \item The decomposition is $\C^m$, i.e. $K^0$ acts irreducibly.
    \end{enumerate}
    
    \textbf{Case (i):} in this case, $K^0$ contains the upper-left $SU(m-1)$ and is a subgroup of $S(U(m-1) \times U(1))$. Thus, $K^0 = S(U(m-1) \times C)$ for some subgroup $C$ of $U(1)$. As $K^0$ is connected, the only possibilities are $K^0 = SU(m-1)$ or $K^0 = S(U(m-1) \times U(1))$. Since $K^0$ is a normal subgroup of $K$, we have $K \subseteq N_{SU(m)}(K^0)$. Using this, we find the possibilities for $K$:
    \begin{itemize}
        \item If $K^0 = SU(m-1)$, then $N_{SU(m)}(K^0) = S(U(m-1) \times U(1))$ so $K$ is of the form $S(U(m-1) \times C)$.
        \item If $K^0 = S(U(m-1) \times U(1))$, then $N_{SU(m)}(K^0) = K^0$, so $K= K^0$.
    \end{itemize}

    \textbf{Case (ii):} since $K^0$ contains $SU(m-1)$, its Lie algebra $\mathfrak{k} \subset \mathfrak{su}(m)$ contains the upper-left $\mathfrak{su}(m-1)$. Since $K^0$ is connected and acts irreducibly on $\C^m$, the induced representation of $\mathfrak{k}$ on $\C^m$ is also irreducible. Hence, there exists $X \in \mathfrak{k}$ with nonzero off-diagonal entries with respect to the vector space decomposition $\C^m = \C^{m-1} \oplus \C^1$, i.e.
    \begin{align}
        X = \begin{pmatrix}
            A & u \\ -u^* & \alpha
        \end{pmatrix}
    \end{align}
    where $A \in \mathfrak{u}(m-1)$, $u \in \C^{(m-1) \times 1}$, and $\alpha \in i\R$, with $\mathrm{diag}(A,\alpha) \in \mathfrak{su}(m)$ and $u \neq 0$. Using this and the fact that $\mathfrak{k}$ contains $\mathfrak{su}(m-1)$, it can be checked by computing commutators that $\mathfrak{k} = \mathfrak{su}(m)$. We omit the details. It follows that $K^0 = SU(m)$ and $K = SU(m)$.

    From these two cases, we have either $K = S(U(m-1) \times C)$ for some subgroup $C$ of $U(1)$, or $K = SU(m)$. Since $K/H$ is a spherical space form, in the former case we must take either $C = \Z_{2k}$ or $C = U(1)$. This establishes part (b) of the lemma; the diffeomorphism types of $K/H$ and $B$ follow.

    For part (a), we have $G = U(2)$ and $H = U(1) \times \Z_k$. The breakdown into cases (i) and (ii) is entirely analogous to part (b). The only substantive difference lies in the analog of case (i), where in the case $K^0 = U(1) \times U(1)$ we get $N_{U(2)}(K^0) = K^0 \rtimes \Z_2$, so either $K = K^0$ or $K = K^0 \rtimes \Z_2$. Thus, gathering the two cases narrows the possibilities for $K$ to
    \begin{align}
        U(1) \times C, \qquad (U(1) \times U(1)) \rtimes \Z_2, \qquad U(2),
    \end{align}
    where $C$ is a subgroup of $U(1)$. Since $K/H$ is a spherical space form, the middle case is ruled out (this would make $K/H$ the disjoint union of two circles), and in the first case we must take either $C = \Z_{2k}$ or $C = U(1)$. This establishes part (a) of the lemma; the diffeomorphism types of $K/H$ and $B$ follow.
\end{proof}

We next express the GRS equation $\Ric = \nabla^2 f + \lambda g$ in terms of the functions $a$, $b$, and $f$ in \eqref{eq:berger-principal-metric-1} and \eqref{eq:berger-principal-metric-2}. Recall that $\lambda=0$ in the steady case, and $\lambda=-\frac{1}{2}$ in the expanding case after normalization. A computation using the curvature formulas in \cite{appleton}*{Appendix A} gives
\begin{align}
    f'' &= -\frac{a''}{a} - (2m-2) \frac{b''}{b} - \lambda, \label{eq:berger-soliton-f} \\
    a'' &= (2m-2)\left( \frac{a^3}{b^4} - \frac{a'b'}{b} \right) - a'f' - \lambda a, \label{eq:berger-soliton-a} \\
    b'' &= \frac{2m}{b} - 2\frac{a^2}{b^3} - \frac{a'b'}{a} - (2m-3)\frac{(b')^2}{b} - b'f' - \lambda b. \label{eq:berger-soliton-b}
\end{align}
It will be useful to consider the ratio
\begin{align}
    Q = \frac{a}{b}.
\end{align}
Substituting $a=Qb$ into \eqref{eq:berger-soliton-a} and using \eqref{eq:berger-soliton-b}, we obtain (compare with \cite{appleton}*{equation (3.1)})
\begin{align}
    Q'' + \left( f' + (2m-1) \frac{b'}{b} \right) Q' = \frac{2m}{b^2} Q(Q^2-1). \label{eq:berger-ratio}
\end{align}
Equivalently,
\begin{align}
    (e^f b^{2m-1} Q')' = 2m e^f b^{2m-3} Q(Q^2-1). \label{eq:berger-ratio-divergence}
\end{align}
The next observation follows directly from \eqref{eq:berger-ratio} and \eqref{eq:berger-ratio-divergence}.
\begin{corollary} \label{cor:Q-local-extrema-1}
    $Q$ cannot have a local maximum with value strictly greater than $1$ or a local minimum with value strictly less than $1$. Moreover, if for some $s_0 > 0$ we have $Q(s_0) > 1$ and $Q'(s_0) \geq 0$, then $Q(s) > Q(s_0)$ and $Q'(s) > 0$ for every $s > s_0$.
\end{corollary}

We now rule out the exceptional orbit in Lemma \ref{lem:berger-endpoint-isotropy}.

\begin{lemma} \label{lem:berger-no-exceptional-orbit}
    The non-principal orbit $B$ is not exceptional.
\end{lemma}
\begin{proof}
    Suppose that $B$ is exceptional. By the smoothness conditions for metrics and tensors across a non-principal orbit (see for instance \cites{eschenburg-wang,verdiani-ziller}), we need $a,b,f$ to extend evenly across $s=0$, with $a(0) > 0$, $b(0) > 0$. In particular,
    \begin{align}
        a'(0) = b'(0) = f'(0) = 0.
    \end{align}
    Evaluating \eqref{eq:berger-soliton-a} and \eqref{eq:berger-soliton-b} at $s=0$ gives
    \begin{align}
        \frac{a''(0)}{a(0)} &= (2m-2) \frac{Q(0)^2}{b(0)^2} - \lambda, \\
        \frac{b''(0)}{b(0)} &= \frac{2m-2Q(0)^2}{b(0)^2} - \lambda. \label{eq:bpp-ratio-exceptional}
    \end{align}
    Substituting these into \eqref{eq:berger-soliton-f} yields
    \begin{align}
        f''(0) = (2m-2) \left( \lambda + \frac{Q(0)^2-2m}{b(0)^2} \right).
    \end{align}
    Meanwhile, since $f$ is constant on $B$ and $f'(0) = 0$, the traced GRS equation $\Delta f = R-n\lambda$ and the inequality $R \geq n\lambda$ (by \eqref{eq:R0} and \eqref{eq:EGS-1}) give
    \begin{align}
        0 \leq R|_B - 2m\lambda = \Delta f|_B = \left( f'' + \left(\frac{a'}{a} + (2m-2)\frac{b'}{b} \right)f' \right)\Big|_{s=0} = f''(0).
    \end{align}
    Consequently,
    \begin{align} \label{eq:berger-exceptional-Q-bound}
        Q(0)^2 \geq 2m - \lambda b(0)^2 > 1.
    \end{align}
    Since $Q'(0)=0$, \eqref{eq:berger-ratio} now gives $Q''(0) = \frac{2m}{b(0)^2} Q(0) (Q(0)^2 - 1) > 0$. Thus $Q'>0$ and $Q > Q(0) > 1$ for all sufficiently small $s>0$. Corollary \ref{cor:Q-local-extrema-1} then implies
    \begin{align} \label{eq:berger-exceptional-Q-increasing}
        Q(s) > Q(0) > 1, \qquad Q'(s) > 0 \qquad \text{for every } s > 0.
    \end{align}
    In the steady case, this contradicts the asymptotic condition $Q(s) \to 1$ from \eqref{eq:ab-asymptotics}. It remains to consider the expanding case.

    By \eqref{eq:bpp-ratio-exceptional} and \eqref{eq:berger-exceptional-Q-bound}, we have
    \begin{align}
        \frac{b''(0)}{b(0)} = \frac{2m-2Q(0)^2}{b(0)^2} - \lambda \leq -\frac{2m}{b(0)^2} + \lambda < 0.
    \end{align}
    Hence $b'<0$ for all sufficiently small $s>0$. We claim that $b'<0$ on all of $(0,\infty)$. Otherwise, let $s_* > 0$ be its first subsequent zero. Then $b(s_*) < b(0)$ and $b''(s_*) \geq 0$. However, evaluating \eqref{eq:berger-soliton-b} at $s_*$ and using \eqref{eq:berger-exceptional-Q-increasing}, we get
    \begin{align}
        \frac{b''(s_*)}{b(s_*)} &= \frac{2(m-Q(s_*)^2)}{b(s_*)^2} - \lambda \\
        &\leq \frac{2(m-Q(0)^2)}{b(0)^2} - \lambda \\
        &\leq -\frac{2m}{b(0)^2} + \lambda < 0,
    \end{align}
    a contradiction.

    Thus, $b$ is strictly decreasing and bounded above by $b(0)$, but this contradicts the expanding asymptotic condition $b(s)/s \to \beta > 0$. Therefore $B$ cannot be exceptional.
\end{proof}

\begin{proof}[Proof of Theorem \ref{thm:berger-symmetry-intro}]
    The metric ansatz and asymptotic conclusions have already been established in \eqref{eq:berger-principal-metric-1}, \eqref{eq:berger-principal-metric-2}, and \eqref{eq:ab-asymptotics}. It remains to identify the underlying orbifold and the $U(m)$-action.

    By Lemmas \ref{lem:berger-endpoint-isotropy} and \ref{lem:berger-no-exceptional-orbit}, either $B \cong \CP^{m-1}$ or $B$ is a point. If $B$ is a point, then Proposition \ref{prop:coho-one-diffeotype}(b) gives $M \cong \C^m/\Z_k$. If $B \cong \CP^{m-1}$, then in the case $m=2$, this corresponds to
    \begin{align}
        G = U(2), \qquad K = U(1) \times U(1), \qquad H = U(1) \times \Z_k, \qquad K/H \cong \Sph^1.
    \end{align}
    Proposition \ref{prop:coho-one-diffeotype}(a) gives $M \cong G \times_K \R^2$, where $K$ acts on $\R^2$ by radial extension of its left-multiplication action on $K/H \cong \Sph^1$. This gives $M \cong \O_{\CP^{m-1}}(-k)$ (see the end of Appendix \ref{subsec:line-bundle-geometry}). Similarly, in the case $m \geq 3$, this corresponds to
    \begin{align}
        G = SU(m), \quad K = S(U(m-1) \times U(1)), \quad H = S(U(m-1) \times \Z_k), \quad K/H \cong \Sph^1,
    \end{align}
    which gives $M \cong \O_{\CP^{m-1}}(-k)$ as well.

    In both cases for $B$, the natural $U(m)$-action preserves $g$ on $M \setminus B$, is transitive on each $L_{m,k}$ cross-section, and extends smoothly over $B$. Thus this action is isometric on all of $M$ and has cohomogeneity one.
\end{proof}

To prove Corollary \ref{cor:appleton-characterization-intro}, let us further narrow down the conclusion in case (i)(b) of Theorem \ref{thm:berger-symmetry-intro}.

\begin{corollary} \label{cor:berger-point-bryant}
    In case (i)(b) of Theorem \ref{thm:berger-symmetry-intro}, $(M,g)$ is isometric to the $\Z_k$-quotient of the $2m$-dimensional Bryant steady soliton, where the $\Z_k$-action is generated by scalar multiplication by a $k$-th root of unity on $\R^{2m} \cong \C^m$.
\end{corollary}
\begin{proof}
    Lift the metric and potential to $\C^m$ using the orbifold covering map $\C^m \to \C^m/\Z_k$. Away from the origin, the lifted metric is
    \begin{align}
        \tilde{g} = ds^2 + \tilde{g}_{a(s),b(s)} = ds^2 + a(s)^2 \sigma^2 + b(s)^2 \pi^*g_{\mathrm{FS}}.
    \end{align}
    Smoothness at the origin requires $a(0)=b(0)=0$ and $a'(0)=b'(0)=1$. Thus, the ratio $Q = a/b$ satisfies
    \begin{align}
        \lim_{s \to 0} Q(s) = \lim_{s \to \infty} Q(s) = 1,
    \end{align}
    where the second limit follows from Theorem \ref{thm:berger-symmetry-intro}. By Corollary \ref{cor:Q-local-extrema-1}, these two limits imply that $Q \equiv 1$ everywhere, so $a \equiv b$ and
    \begin{align}
        \tilde{g} = ds^2 + a(s)^2 g_{\Sph^{2m-1}}.
    \end{align}
    The lifted soliton $(\C^m, \tilde{g}, \tilde{f})$ is complete, nonflat, and rotationally symmetric, so it is the Bryant steady soliton. The assertion follows.
\end{proof}

\begin{proof}[Proof of Corollary \ref{cor:appleton-characterization-intro}]
    Let $k \geq 3$, and let $(M^4,g,f)$ be an asymptotically $L_{2,k}$-cylindrical steady GRS.

    Suppose $\ker_{L^2(e^f)}(L) = \{0\}$. Then Theorem \ref{thm:berger-symmetry-intro} applies, giving two possibilities: $M \cong \O_{\CP^1}(-k)$ or $M \cong \C^2/\Z_k$. The soliton metric and potential reduce to the ansatz shown in that theorem, with the asymptotic condition $\lim_{s \to \infty} \frac{a(s)}{b(s)} = 1$. If $M \cong \O_{\CP^1}(-k)$, we can therefore apply \cite{law-appleton}*{Theorem 1.1} to obtain that $(M,g)$ is isometric to the Appleton soliton on $\O_{\CP^1}(-k)$. If $M \cong \C^2/\Z_k$, Corollary \ref{cor:berger-point-bryant} shows that $(M,g)$ is isometric to the $\Z_k$-quotient of the $4$-dimensional Bryant soliton.

    Conversely, \cite{law-appleton}*{Theorem 1.3} and \cite{deruelle-stability}*{Proposition 6.6} show that the Appleton soliton on $\O_{\CP^1}(-k)$ and the Bryant steady soliton (and therefore its $\Z_k$-quotient) are both strictly linearly stable. In particular, they satisfy $\ker_{L^2(e^f)}(L) = \{0\}$.
\end{proof}

\section{Steady Ricci solitons asymptotic to cylinders over $SU(2)/\Gamma$} \label{sec:spherical-space-form-classification}

In this section, we complete the proof of Theorem \ref{thm:spherical-space-form-classification}. The cyclic cases with $\Gamma = \{1\}$ or $\Gamma = \Z_k$, $k \geq 3$, follow from Theorem \ref{thm:bryant-classification-intro} and Corollary \ref{cor:appleton-characterization-intro} respectively. Here we handle the other cases: $\Gamma = \Z_2$, $\Gamma = \tilde{D}_k$, and $\Gamma \in \{\tilde{T},\tilde{O},\tilde{I}\}$.

\subsection{Subgroups of $SU(2)$ and identifications} \label{subsec:SU2-preliminaries}

Throughout this section, let $\mathbb{H}$ be the quaternions and identify $SU(2)$ with the unit group
\begin{align}
    SU(2) = \{ x_0 + x_1 \bfi + x_2 \bfj + x_3 \bfk \in \mathbb{H} : x_0^2 + x_1^2 + x_2^2 + x_3^2 = 1 \}.
\end{align}
The finite subgroups of $SU(2)$, up to conjugacy, are listed in Table \ref{tab:finite-SU2-subgroups}, together with their orders, generators, and normalizers.
The closed subgroups of $SU(2)$ up to conjugacy consist of the finite subgroups listed in Table \ref{tab:finite-SU2-subgroups}, the circle $U(1) = \{e^{\bfi \theta} : \theta \in \R\}$, its normalizer
\begin{align}
    \operatorname{Pin}(2) = U(1) \sqcup \bfj U(1),
\end{align}
and $SU(2)$ itself. These facts are standard; see e.g. \cite{wolf}*{Theorem 2.6.7}.

\begin{table}[htbp]
    \centering
    \begin{tabular}{c|c|c|c}
        $\Gamma$
        & $|\Gamma|$
        & Generators
        & $N_{SU(2)}(\Gamma)$ \\
        \hline \noalign{\vskip 1pt}
        $\{1\}$
        & $1$ & $1$ & $SU(2)$ \\

        $\Z_2$
        & $2$ & $-1$ & $SU(2)$ \\

        $\Z_k$, $k\geq3$
        & $k$ & $e^{2\pi\bfi/k}$
        & $\operatorname{Pin}(2)$ \\

        $\tilde{D}_2$
        & $8$ & $\bfi,\bfj$ & $\tilde{O}$ \\

        $\tilde{D}_k$, $k\geq3$
        & $4k$ & $e^{\pi\bfi/k},\bfj$
        & $\tilde{D}_{2k}$ \\

        $\tilde{T}$
        & $24$ & $\bfj$, $e^{\pi \bfi/4}\left( \frac{1}{\sqrt{2}} \bfi + \frac{1}{\sqrt{2}} \bfj \right)$ & $\tilde{O}$ \\

        $\tilde{O}$
        & $48$ & $\frac{1}{\sqrt{2}} \bfi + \frac{1}{\sqrt{2}} \bfj$, $e^{\pi \bfi/4}\left( \frac{1}{\sqrt{2}} \bfi + \frac{1}{\sqrt{2}} \bfj \right)$ & $\tilde{O}$ \\

        $\tilde{I}$
        & $120$ & $\bfj$, $\frac{1}{2} + \frac{1}{2} \sqrt{1-2\cos(2\pi/5)} \, \bfi + \cos(\pi/5) \bfj$ & $\tilde{I}$ \\
    \end{tabular}
    \caption{Finite subgroups of $SU(2)$ up to conjugacy,
    with their orders, generators, and normalizers.}
    \label{tab:finite-SU2-subgroups}
\end{table}

Given a finite subgroup $\Gamma \subset SU(2)$, we write $SU(2)/\Gamma$ for the quotient by the right action of $\Gamma$. Left multiplication induces a transitive $SU(2)$-action on $SU(2)/\Gamma$ with isotropy subgroups conjugate to $\Gamma$.

Let $\{X_1,X_2,X_3\}$ be the left-invariant frame on $SU(2)$ agreeing with $\{\bfi,\bfj,\bfk\}$ at the identity, and let $\{\sigma_1,\sigma_2,\sigma_3\}$ be the dual coframe. Then
\begin{align} \label{eq:bracket-normalization}
    [X_1, X_2] = 2X_3, \qquad [X_2, X_3] = 2X_1, \qquad [X_3, X_1] = 2X_2.
\end{align}

Identify $\C^2$ with $\mathbb{H}$ by the map $(z_1,z_2) \mapsto z_1 + \bfj z_2$. This restricts to $\Sph^3 \cong SU(2)$. Under this identification, the round metric on $\Sph^3$ of constant sectional curvature $1$ is
\begin{align} \label{eq:round-identification}
    g_{\Sph^3} = \sigma_1^2 + \sigma_2^2 + \sigma_3^2.
\end{align}
Recall that the Hopf circle action on $\Sph^3 \subset \C^2$ is generated by scalar multiplications by $e^{i\theta} \in U(1)$. This corresponds to right-multiplications by $e^{\bfi\theta}$ on $SU(2)$, which has infinitesimal generator $X_1$. Thus, with $\sigma$ the Hopf $1$-form on $\Sph^3$ as defined in \S\ref{subsec:uniq-appl}, we have the identifications
\begin{align} \label{eq:identifications-S3-SU2}
    \sigma = \sigma_1, \qquad \pi^*g_{\mathrm{FS}} = g_{\Sph^3} - \sigma^2 = \sigma_2^2 + \sigma_3^2.
\end{align}

\subsection{The real projective case}

We first address $\Gamma = \Z_2$, for which the asymptotic link is the round $\RP^3$.

\begin{proposition} \label{prop:RP3-bryant}
    Let $(M^4,g,f)$ be an asymptotically $\RP^3$-cylindrical steady GRS satisfying $\ker_{L^2(e^f)}(L) = \{0\}$. Then $(M,g)$ is isometric to the $\Z_2$-quotient of the $4$-dimensional Bryant steady soliton.
\end{proposition}
\begin{proof}
    Theorem \ref{thm:general-symmetry-intro} gives a closed connected subgroup
    \begin{align}
        G \subseteq \I^0(\RP^3, g_{\RP^3}) = SO(4)/\{\pm I\}, \qquad \dim G \geq 6,
    \end{align}
    acting smoothly, effectively, and isometrically with cohomogeneity one on $(M,g)$ and preserving $f$. Thus $G = SO(4)/\{\pm I\}$. Moreover, every sufficiently far level set of $f$ is $G$-equivariantly diffeomorphic to $\RP^3$.
    
    Composing with the covering homomorphism $SO(4) \to SO(4)/\{\pm I\}$ therefore gives an almost effective isometric cohomogeneity-one action of $SO(4)$ preserving $f$, with every sufficiently far level set being $SO(4)$-equivariantly diffeomorphic to $\RP^3$. Since the $SO(4)$-action on $\RP^3$ is transitive, with isotropy subgroups conjugate to $S(O(3) \times O(1))$, it follows that the $SO(4)$-action on $M$ has principal isotropy subgroup $H = S(O(3) \times O(1))$.
    
    Thus, Proposition \ref{prop:coho-one-global-structure} gives a unique non-principal orbit $B \cong SO(4)/K$, where $H \subsetneq K \subseteq SO(4)$ and $K/H$ is a spherical space form, such that
    \begin{align}
        M \setminus B \cong (0,\infty) \times \RP^3, \qquad g = ds^2 + g_s, \qquad f=f(s),
    \end{align}
    where $g_s$ is an $SO(4)$-invariant metric on $SO(4)/H$. By the same reasoning as the discussion before Lemma \ref{lem:round-K-classification}, each $g_s$ is a multiple of the round metric $g_{\RP^3}$. Thus
    \begin{align}
        g = ds^2 + a(s)^2 g_{\RP^3},
    \end{align}
    where $a$ is a smooth positive function on $(0,\infty)$.
    
    By Lemma \ref{lem:round-K-classification}, $K = SO(4)$ and $B$ is a point. Proposition \ref{prop:coho-one-diffeotype}(b) yields $M \cong \R^4/\Z_2$. Lifting the soliton structure to $\R^4$ gives a smooth, complete, nonflat, rotationally symmetric steady GRS, which must be the $4$-dimensional Bryant steady soliton.
\end{proof}

\subsection{Global symmetry for non-cyclic quotients} \label{subsec:homog-metrics-noncyclic}

Now turning to finite non-cyclic subgroups $\Gamma \subset SU(2)$, we first record the isometry groups of $SU(2)/\Gamma$ with their round metrics.
\begin{lemma} \label{lem:round-noncyclic-isometries}
    Let $\Gamma \subset SU(2)$ be a finite non-cyclic subgroup, and equip $SU(2)/\Gamma \cong \Sph^3/\Gamma$ with the round metric $g_{\Sph^3/\Gamma}$ of sectional curvature $1$. Then $\I^0(SU(2)/\Gamma, g_{\Sph^3/\Gamma})$ is a finite quotient of the standard left $SU(2)$-action. In particular, $\dim\I^0(SU(2)/\Gamma, g_{\Sph^3/\Gamma}) = 3$.
\end{lemma}
\begin{proof}
    Denote the groups of left and right multiplications on $SU(2)$ by $SU(2)_L$ and $SU(2)_R$ respectively, and write $\Gamma_R$ for the right action of $\Gamma$. The identity component of the isometry group of $SU(2)$ equipped with the round metric $g_{\Sph^3}$ is
    \begin{align}
        \I^0(SU(2),g_{\Sph^3}) = (SU(2)_L \times SU(2)_R)/\{(1,1) \sim (-1,-1)\} \cong SO(4),
    \end{align}
    and the full isometry group is isomorphic to $O(4)$. Since the left and right factors commute, we have
    \begin{align}
        N_{SO(4)}(\Gamma_R) = SU(2)_L \cdot N_{SU(2)_R}(\Gamma_R).
    \end{align}
    Since $\Gamma$ is non-cyclic, Table \ref{tab:finite-SU2-subgroups} shows that $N_{SU(2)_R}(\Gamma_R)$ is discrete, so taking identity components gives $N_{SO(4)}^0(\Gamma_R) = SU(2)_L$. In particular,
    \begin{align}
        N_{O(4)}^0(\Gamma_R) = N_{SO(4)}^0(\Gamma_R) = SU(2)_L.
    \end{align}
    By Lemma \ref{lem:quotient-isom-group}, the isometry group of the round $SU(2)/\Gamma$ is
    \begin{align} \label{eq:round-su2gamma-i}
        \I(SU(2)/\Gamma, g_{\Sph^3/\Gamma}) \cong N_{O(4)}(\Gamma_R)/\Gamma_R.
    \end{align}
    The identity component of the quotient \eqref{eq:round-su2gamma-i} is the image of $N_{O(4)}^0(\Gamma_R) = SU(2)_L$ under the quotient map $N_{O(4)}(\Gamma_R) \to N_{O(4)}(\Gamma_R)/\Gamma_R$. This map has finite kernel, so the identity component of \eqref{eq:round-su2gamma-i} is a finite quotient of $SU(2)_L$, giving the lemma.
\end{proof}

We now apply the symmetry principle to solitons with non-cyclic asymptotic links.
\begin{corollary} \label{cor:noncyclic-global-action}
    Let $\Gamma \subset SU(2)$ be a finite non-cyclic subgroup, and let $(M^4,g,f)$ be an asymptotically $SU(2)/\Gamma$-cylindrical steady GRS satisfying $\ker_{L^2(e^f)}(L) = \{0\}$. Then $SU(2)$ acts smoothly, almost effectively, and isometrically with cohomogeneity one on $(M,g)$. The action preserves $f$, and every sufficiently far level set of $f$ is an orbit with isotropy subgroup conjugate to $\Gamma$.
\end{corollary}
\begin{proof}
    By Lemma \ref{lem:round-noncyclic-isometries}, $\I^0(SU(2)/\Gamma, g_{\Sph^3/\Gamma})$ is a finite quotient of $SU(2)_L$. In particular, $\dim\I^0(SU(2)/\Gamma, g_{\Sph^3/\Gamma}) = 3$. Theorem \ref{thm:general-symmetry-intro} therefore gives a closed connected subgroup $G \subseteq \I^0(SU(2)/\Gamma, g_{\Sph^3/\Gamma})$ with $\dim G \geq 3$ acting effectively and isometrically with cohomogeneity one on $M$. The $G$-action preserves $f$, and each sufficiently far level set of $f$ is $G$-equivariantly diffeomorphic to $SU(2)/\Gamma$.

    Since $G$ is connected and $3 \leq \dim G \leq \dim \I^0(SU(2)/\Gamma,g_{\Sph^3/\Gamma}) \leq 3$, we must have $G = \I^0(SU(2)/\Gamma,g_{\Sph^3/\Gamma})$. Thus, $G$ is a finite quotient of $SU(2)_L$.
    Composing the $G$-action with $SU(2)_L \to G$ gives the required almost effective action. Every sufficiently far level set of $f$ is an orbit $SU(2)_L$-equivariantly diffeomorphic to $SU(2)/\Gamma$, so the isotropy subgroups in that orbit are conjugate to $\Gamma$.
\end{proof}

The next lemma rules out exceptional orbits. Since the solitons considered here are steady, we can use the topological assertion of one-endedness instead of the curvature-based arguments in the proofs of Theorem \ref{thm:bryant-classification-intro} and Lemma \ref{lem:berger-no-exceptional-orbit} above.

\begin{lemma} \label{lem:noncyclic-no-exceptional-orbit}
    The $SU(2)$-action in Corollary \ref{cor:noncyclic-global-action} has no exceptional orbits.
\end{lemma}
\begin{proof}
    Suppose there is an exceptional orbit, and let $K \subseteq SU(2)$ be its isotropy subgroup as in Proposition \ref{prop:coho-one-global-structure}. Exceptionality means $\dim SU(2)/K = 3$, so the closed subgroup $K$ must be finite. Moreover, $K/\Gamma \cong \Sph^0$.
    
    By Proposition \ref{prop:coho-one-diffeotype}(a), $M$ is a smooth manifold diffeomorphic to the total space of a real line bundle over $B \cong SU(2)/K$. Pulling this bundle back by the finite covering $SU(2) \to SU(2)/K$ gives a covering $\tilde{M} \to M$ where $\tilde{M}$ is a real line bundle over $SU(2)$. Every real line bundle over $SU(2) \cong \Sph^3$ is trivial, so
    \begin{align}
        \tilde{M} \cong \Sph^3 \times \R.
    \end{align}
    In particular, $\tilde{M}$ has two ends. Since $(M,g)$ has positive scalar curvature by \eqref{eq:R0}, pulling back the soliton structure gives a complete steady GRS on $\tilde{M}$ with positive scalar curvature. By \cite{mw-smms}*{Corollary 1.1}, $\tilde{M}$ must have one end, a contradiction.
\end{proof}

The next two subsections separately treat the cases $\Gamma \in \{\tilde{T},\tilde{O},\tilde{I}\}$ and $\Gamma = \tilde{D}_k$, $k \geq 2$.

\subsection{The binary tetrahedral, octahedral, and icosahedral cases} \label{subsec:binary-polyhedral-cases}

We show that if $\Gamma \in \{\tilde{T},\tilde{O},\tilde{I}\}$, then rotational symmetry is forced:
\begin{proposition} \label{prop:polyhedral-bryant}
    Let $\Gamma \in \{\tilde{T},\tilde{O},\tilde{I}\}$, and let $(M^4,g,f)$ be an asymptotically $SU(2)/\Gamma$-cylindrical steady GRS satisfying $\ker_{L^2(e^f)}(L) = \{0\}$. Then $(M,g)$ is isometric to the $\Gamma$-quotient of the $4$-dimensional Bryant steady soliton.
\end{proposition}
\begin{proof}
    By Corollary \ref{cor:noncyclic-global-action} and Proposition \ref{prop:coho-one-global-structure}, there is a unique non-principal orbit $B \cong SU(2)/K$, where $\Gamma \subsetneq K \subseteq SU(2)$ and $K/\Gamma$ is a spherical space form, such that
    \begin{align}
        M \setminus B \cong (0,\infty) \times SU(2)/\Gamma, \qquad g = ds^2 + g_s, \qquad f=f(s),
    \end{align}
    where $s$ is the distance from $B$, and each $g_s$ is invariant under the standard left $SU(2)$-action. Such invariant metrics correspond to $\operatorname{Ad}(\Gamma)$-invariant inner products on $\mathfrak{su}(2)$. However, $\operatorname{Ad}(\Gamma)$ is irreducible on $\mathfrak{su}(2)$, so by Schur's lemma, every such inner product is a positive multiple of the standard inner product. Hence,
    \begin{align}
        g = ds^2 + a(s)^2 g_{\Sph^3/\Gamma}
    \end{align}
    for some smooth positive function $a$ on $(0,\infty)$, where $g_{\Sph^3/\Gamma}$ is the quotient of the unit round metric.

    We next identify the possibilities for $K$. By the classification of closed subgroups of $SU(2)$ (see \S\ref{subsec:SU2-preliminaries}), we have the following possibilities:
    \begin{align*}
        \begin{array}{c|c|c|c}
            \Gamma & K & K/\Gamma & B\cong SU(2)/K \\ \hline
            \noalign{\vskip 1pt}
            \tilde{T}
                & \tilde{O}
                & \Sph^0
                & SU(2)/\tilde{O} \\
            \tilde{T},\tilde{O},\tilde{I}
                & SU(2)
                & SU(2)/\Gamma
                & \{\mathrm{point}\}.
        \end{array}
    \end{align*}
    By Lemma \ref{lem:noncyclic-no-exceptional-orbit}, the case $\Gamma = \tilde{T}$ and $K = \tilde{O}$ cannot occur. Thus, we have $K = SU(2)$, and $B$ is a point. Proposition \ref{prop:coho-one-diffeotype}(b) yields $M \cong \R^4/\Gamma$. Lifting the soliton structure to $\R^4$ gives a smooth, complete, nonflat, rotationally symmetric steady GRS, which must be the $4$-dimensional Bryant steady soliton.
\end{proof}

\subsection{The binary dihedral cases} \label{subsec:binary-dihedral-case}

We now consider the cases $\Gamma = \tilde{D}_k$, $k \geq 2$. Throughout this subsection, let $(M^4,g,f)$ be an asymptotically $SU(2)/\tilde{D}_k$-cylindrical steady GRS satisfying $\ker_{L^2(e^f)}(L) = \{0\}$.

By Corollary \ref{cor:noncyclic-global-action} and Proposition \ref{prop:coho-one-global-structure}, there is a unique non-principal orbit $B \cong SU(2)/K$, where $\tilde{D}_k \subsetneq K \subseteq SU(2)$, and an identification
\begin{align}
    M \setminus B \cong (0,\infty) \times SU(2)/\tilde{D}_k, \qquad g = ds^2 + g_s, \qquad f=f(s).
\end{align}
Here each $g_s$ is $SU(2)$-invariant, and $K/\tilde{D}_k$ is a spherical space form. The classification of closed subgroups of $SU(2)$ (see \S\ref{subsec:SU2-preliminaries}) gives the following possibilities for $K$ and hence for $B$.

\begin{lemma} \label{lem:dihedral-endpoint-isotropy}
    Let $H = \tilde{D}_k \subset SU(2)$, where $k \geq 2$. Up to conjugacy, the closed subgroups $K \subseteq SU(2)$ such that $H \subsetneq K$ and $K/H$ is a spherical space form are as follows:
    \begin{align*}
        \begin{array}{c|c|c|c}
            & K & K/H & B \cong SU(2)/K \\ \hline
            \noalign{\vskip 1pt}
            \mathrm{(i)}
                & \tilde{D}_{2k}
                & \Sph^0
                & SU(2)/\tilde{D}_{2k} \\
            \mathrm{(ii)}
                & \operatorname{Pin}(2)
                & \Sph^1
                & \RP^2 \\
            \mathrm{(iii)}
                & SU(2)
                & SU(2)/\tilde{D}_k
                & \{\mathrm{point}\}
        \end{array}
    \end{align*}
\end{lemma}

By Lemma \ref{lem:noncyclic-no-exceptional-orbit}, we have:
\begin{lemma} \label{lem:dihedral-no-exceptional-orbit}
    Case (i) in Lemma \ref{lem:dihedral-endpoint-isotropy} cannot occur.
\end{lemma}

Since the metric $g_s$ on $SU(2)/\tilde{D}_k$ is $SU(2)$-invariant, it pulls back to a left-invariant metric $\tilde{g}_s$ on $SU(2)$ whose value at the identity is an $\operatorname{Ad}(\tilde{D}_k)$-invariant inner product on $\mathfrak{su}(2)$. Recall from Table \ref{tab:finite-SU2-subgroups} that $\tilde{D}_k$ is generated by $e^{\pi \bfi/k}$ and $\bfj$. The adjoint action of $e^{\pi \bfi/k}$ fixes $\bfi$ and rotates the plane $\operatorname{span}\{\bfj,\bfk\}$. Thus, $\operatorname{Ad}(\tilde{D}_k)$-invariance implies $\bfi$ is $\tilde{g}_s$-orthogonal to $\operatorname{span}\{\bfj,\bfk\}$. Meanwhile the adjoint action of $\bfj$ fixes $\bfj$ and maps $\bfk \mapsto -\bfk$, so $\operatorname{Ad}(\tilde{D}_k)$-invariance implies $\tilde{g}_s(\bfj,\bfk) = 0$. Therefore, in the fixed coframe introduced in \S\ref{subsec:SU2-preliminaries},
\begin{align}
    \tilde{g}_s = a(s)^2 \sigma_1^2 + b(s)^2 \sigma_2^2 + c(s)^2 \sigma_3^2.
\end{align}

\begin{lemma} \label{lem:triaxial-round-asymptotics}
    Suppose a steady GRS on $(0,\infty) \times SU(2)$ has the form
    \begin{align}
        g = ds^2 + a(s)^2 \sigma_1^2 + b(s)^2 \sigma_2^2 + c(s)^2 \sigma_3^2, \qquad f=f(s),
    \end{align}
    where $a,b,c$ are positive functions on $(0,\infty)$ with nonnegative limits as $s \to 0$, and
    \begin{align}
        \lim_{s \to \infty} \frac{a}{b} = \lim_{s \to \infty} \frac{b}{c} = 1.
    \end{align}
    Then
    \begin{enumerate}[label=(\alph*)]
        \item If $a/b \to 1$ and $b/c \to 1$ as $s \to 0$, then $a=b=c$.
        \item If $a(0)=0$ and $b(0)=c(0) > 0$, then $b=c$.
    \end{enumerate}
\end{lemma}
\begin{proof}
    With our convention $\Ric = \nabla^2 f$ and normalization \eqref{eq:bracket-normalization}, the soliton equations in \cite{buttsworth}*{equation (9)} give, in particular,
    \begin{align}
        \frac{a''}{a} + \frac{a'}{a} \left( f' + \frac{b'}{b} + \frac{c'}{c} \right) &= \frac{2(a^4 - (b^2-c^2)^2)}{a^2 b^2 c^2}, \label{eq:triaxial-soliton-a} \\
        \frac{b''}{b} + \frac{b'}{b} \left( f' + \frac{a'}{a} + \frac{c'}{c} \right) &= \frac{2(b^4-(a^2-c^2)^2)}{a^2 b^2 c^2}, \label{eq:triaxial-soliton-b} \\
        \frac{c''}{c} + \frac{c'}{c} \left( f' + \frac{a'}{a} + \frac{b'}{b} \right) &= \frac{2(c^4 - (a^2-b^2)^2)}{a^2 b^2 c^2}. \label{eq:triaxial-soliton-c}
    \end{align}
    Subtracting \eqref{eq:triaxial-soliton-b} from \eqref{eq:triaxial-soliton-a} and simplifying, we obtain
    \begin{align}
        \frac{a''}{a} - \frac{b''}{b} + \left( f' + \frac{c'}{c} \right) \left( \frac{a'}{a} - \frac{b'}{b} \right) &= \frac{4(a^2-b^2)(a^2 + b^2 - c^2)}{a^2 b^2 c^2}. \label{eq:60086}
    \end{align}
    Setting $u_{ab} = \log(a/b)$, we have $u_{ab}' = \frac{a'}{a} - \frac{b'}{b}$ and $\frac{a''}{a} - \frac{b''}{b} = u_{ab}'' + \left( \frac{a'}{a} + \frac{b'}{b} \right) u_{ab}'$. Substituting into \eqref{eq:60086} yields
    \begin{align}
        u_{ab}'' + \left( f' + \frac{a'}{a} + \frac{b'}{b} + \frac{c'}{c} \right) u_{ab}' = \frac{4(a^2-b^2)(a^2+b^2-c^2)}{a^2 b^2 c^2}.
    \end{align}
    Similarly, for any permutation $(x,y,z)$ of $(a,b,c)$, we have
    \begin{align} \label{eq:triaxial-ratio}
        u_{xy}'' + \left( f' + \frac{a'}{a} + \frac{b'}{b} + \frac{c'}{c} \right) u_{xy}' &= \frac{4(x^2-y^2)(x^2+y^2-z^2)}{x^2 y^2 z^2}.
    \end{align}
    We will use these identities to prove the lemma. In case (a), set
    \begin{align}
        v(s) = \frac{\max\{a(s),b(s),c(s)\}}{\min\{a(s),b(s),c(s)\}},
    \end{align}
    while in case (b), set
    \begin{align}
        v(s) = \frac{\max\{a(s),b(s),c(s)\}}{\min\{b(s),c(s)\}}.
    \end{align}
    In either case, $v \geq 1$ and $\lim_{s \to 0} v(s) = \lim_{s \to \infty} v(s) = 1$. We will show that $v \equiv 1$, which implies $a=b=c$ in case (a) and $b=c$ in case (b).

    Suppose $v$ is not identically $1$. Then it attains a maximum greater than $1$ at some $s_* > 0$. Then for some permutation $(x,y,z)$ of $(a,b,c)$ we have $v(s_*) = x(s_*)/y(s_*)$ and $x(s_*) = \max\{a(s_*), b(s_*), c(s_*)\}$. Since $x/y \leq v$ everywhere, the function $x/y$ attains a local maximum greater than $1$ at $s_*$, and hence $u_{xy} = \log(x/y)$ also attains a local maximum. Thus $u_{xy}'(s_*) = 0$ and $u_{xy}''(s_*) \leq 0$. However, \eqref{eq:triaxial-ratio} gives, at $s_*$,
    \begin{align}
        u_{xy}''(s_*) = \frac{4(x^2-y^2)(x^2+y^2-z^2)}{x^2 y^2 z^2} > 0,
    \end{align}
    since $x(s_*) > y(s_*)$ and $x(s_*) \geq z(s_*)$. This is a contradiction.
\end{proof}

We now address the remaining two cases in Lemma \ref{lem:dihedral-endpoint-isotropy}. The Appleton quotients $A_{2k}/\Z_2$ in the next proposition are detailed in Appendix \ref{subsec:appleton-Z2-quotients}.
\begin{proposition} \label{prop:dihedral-RP2-cover}
    In case (ii) of Lemma \ref{lem:dihedral-endpoint-isotropy}, $(M,g)$ is isometric to $A_{2k}/\Z_2$.
\end{proposition}
\begin{proof}
    By Proposition \ref{prop:coho-one-diffeotype}(a), $M$ is the total space of a $2$-plane bundle over $B \cong \RP^2$ whose sphere bundles are diffeomorphic to $SU(2)/\tilde{D}_k$. This is the non-orientable $2$-plane bundle over $\RP^2$ with Euler number $\pm k$.

    A smooth double cover of $M$ is obtained by pulling back this bundle by the standard double covering $\Sph^2 \to \RP^2$. This is a $2$-plane bundle $\hat{M}$ over $\Sph^2 \cong \CP^1$ with Euler number $\pm 2k$, i.e. the total space of $\O_{\CP^1}(-2k)$, whose sphere bundles are diffeomorphic to $L_{2,2k}$.
    
    Lift the soliton structure to $(\hat{M},\hat{g},\hat{f})$. This gives an asymptotically $L_{2,2k}$-cylindrical steady GRS. By the preceding discussion, on the complement of the zero section, the pullback of $\hat{g}$ under the natural covering $p:(0,\infty) \times SU(2) \to (0,\infty) \times L_{2,2k}$ has the form
    \begin{align}
        p^*\hat{g} = ds^2 + a(s)^2 \sigma_1^2 + b(s)^2 \sigma_2^2 + c(s)^2 \sigma_3^2.
    \end{align}
    Without loss of generality, assume the collapsing direction is $\sigma_1$. Smoothness of the metric across the zero section requires $a(0) = 0$ and $b(0) = c(0) > 0$. Asymptotic cylindricity gives
    \begin{align}
        \lim_{s \to \infty} \frac{a}{b} = \lim_{s \to \infty} \frac{b}{c} = 1.
    \end{align}
    Lemma \ref{lem:triaxial-round-asymptotics} therefore implies $b=c$. By \eqref{eq:identifications-S3-SU2}, it follows that $\hat{g}$ and $\hat{f}$ satisfy the ansatz \eqref{eq:appleton-ansatz-intro} on $\O_{\CP^1}(-2k)$. We may now apply \cite{law-appleton}*{Theorem 1.1} to conclude that $(\hat{M},\hat{g},\hat{f})$ is the Appleton soliton $A_{2k}$.
\end{proof}

\begin{proposition} \label{prop:binary-dihedral-point}
    In case (iii) of Lemma \ref{lem:dihedral-endpoint-isotropy}, $M$ is isometric to the $\tilde{D}_k$-quotient of the $4$-dimensional Bryant steady soliton.
\end{proposition}
\begin{proof}
    In this case, $B$ is a point. Proposition \ref{prop:coho-one-diffeotype}(b) yields $M \cong \R^4/\tilde{D}_k$. Lifting the soliton structure to $\R^4$ gives an asymptotically $\Sph^3$-cylindrical steady GRS. The preceding discussion gives, outside the origin,
    \begin{align}
        \hat{g} = ds^2 + a(s)^2 \sigma_1^2 + b(s)^2 \sigma_2^2 + c(s)^2 \sigma_3^2, \qquad \hat{f} = \hat{f}(s).
    \end{align}
    Smoothness at the origin requires
    \begin{align}
        a(0)=b(0)=c(0)=0, \qquad a'(0)=b'(0)=c'(0) = 1.
    \end{align}
    Hence,
    \begin{align}
        \lim_{s \to 0} \frac{a}{b} = \lim_{s \to 0} \frac{b}{c} = 1.
    \end{align}
    The same limits hold as $s \to \infty$ by asymptotic cylindricity. Lemma \ref{lem:triaxial-round-asymptotics} therefore implies $a=b=c$, so the lifted soliton is smooth, complete, nonflat, and rotationally symmetric. Hence, it is the $4$-dimensional Bryant steady soliton.
\end{proof}

Combining Lemma \ref{lem:dihedral-endpoint-isotropy}, Lemma \ref{lem:dihedral-no-exceptional-orbit}, Proposition \ref{prop:dihedral-RP2-cover}, and Proposition \ref{prop:binary-dihedral-point}, we obtain a classification in the binary dihedral cases:
\begin{corollary} \label{cor:binary-dihedral-classification}
    Let $(M^4,g,f)$ be an asymptotically $SU(2)/\tilde{D}_k$-cylindrical steady GRS, where $k \geq 2$, satisfying $\ker_{L^2(e^f)}(L) = \{0\}$. Then $(M,g)$ is isometric to either $A_{2k}/\Z_2$ or the $\tilde{D}_k$-quotient of the $4$-dimensional Bryant steady soliton.
\end{corollary}

\appendix
\section{Finite quotients of Appleton solitons} \label{app:appleton-quotients}

\subsection{The line bundles $\O_{\CP^{m-1}}(-k)$} \label{subsec:line-bundle-geometry}

For $m \geq 2$, the tautological line bundle $\O_{\CP^{m-1}}(-1)$ has fiber $\ell$ over $\ell \in \CP^{m-1}$. For each integer $k \geq 1$, its $k$-th tensor power is denoted by $\O_{\CP^{m-1}}(-k)$, with fiber $\ell^{\otimes k}$. The natural action of $U(m)$ on this bundle is
\begin{align} \label{eq:natural-Um-action}
    A \cdot (\ell,v) = (A\ell, A^{\otimes k} v), \qquad v \in \ell^{\otimes k}.
\end{align}
The standard Hermitian inner product on $\C^m$ induces a Hermitian norm on the fibers, preserved by this action.

Recall that $L_{m,k} = \Sph^{2m-1}/\Z_k$, where $\Sph^{2m-1} \subset \C^m$ and the $\Z_k$-action is generated by scalar multiplication by a $k$-th root of unity. The map
\begin{align}
    [u] \mapsto (\C u, u^{\otimes k}), \qquad |u|=1,
\end{align}
defines a $U(m)$-equivariant diffeomorphism from $L_{m,k}$ to the unit circle bundle of $\O_{\CP^{m-1}}(-k)$. This extends to a $U(m)$-equivariant identification
\begin{align}
    (0,\infty) \times L_{m,k} \to \O_{\CP^{m-1}}(-k) \setminus B, \qquad (t,[u]) \mapsto (\C u, t u^{\otimes k}), \label{eq:line-bundle-polar-identification}
\end{align}
where $B \cong \CP^{m-1}$ is the zero section and $t$ is the fiber norm. Thus each positive level set of $t$ is a $U(m)$-orbit diffeomorphic to $L_{m,k}$, while the zero section $\{t=0\}$ is an orbit diffeomorphic to $\CP^{m-1}$.

We also record the associated bundle description used in the proof of Theorem \ref{thm:berger-symmetry-intro}. For a Lie group $G$ and closed subgroups $H \subset K \subset G$, there is a $G$-equivariant diffeomorphism
\begin{align}
    G/H \cong G \times_K (K/H).
\end{align}
Applying this with either
\begin{align}
    (G,K,H) &= (U(m), \, U(m-1) \times U(1), \, U(m-1) \times \Z_k), \\
    (G,K,H) &= (SU(m), \, S(U(m-1) \times U(1)), \, S(U(m-1) \times \Z_k)),
\end{align}
where the subgroups are embedded block-diagonally, identifies $L_{m,k} \cong G/H \cong G \times_K (K/H)$. Since $L_{m,k}$ is also $G$-equivariantly diffeomorphic to the unit circle bundle of $\O_{\CP^{m-1}}(-k)$, we can extend the identification $L_{m,k} \cong G \times_K (K/H)$ across all circle bundles in $\O_{\CP^{m-1}}(-k)$ and across the zero section, giving a $G$-equivariant diffeomorphism
\begin{align}
    \O_{\CP^{m-1}}(-k) \cong G \times_K \R^2.
\end{align}
Here $K$ acts on $\R^2$ by radial extension of its left-multiplication action on $K/H \cong \Sph^1$.

\subsection{\texorpdfstring{$\Z_2$}{Z2}-quotients
of Appleton solitons with binary dihedral asymptotic links}
\label{subsec:appleton-Z2-quotients}

For $k \geq 2$, we construct the quotients $A_{2k}/\Z_2$ appearing in Theorem \ref{thm:spherical-space-form-classification}.

As in \S\ref{subsec:SU2-preliminaries}, consider $SU(2)$ as the group of unit quaternions, and identify $\C^2$ with $\mathbb{H}$ by $(z_1,z_2) \mapsto z_1 + \bfj z_2$. This restricts to $\Sph^3 \cong SU(2)$. Scalar multiplication by $e^{i\theta}$ on $\C^2$ corresponds to right multiplication by $e^{\bfi\theta}$ on $\mathbb{H}$. Right multiplication by $\bfj$ on $\mathbb{H}$ corresponds to the map $J: \C^2 \to \C^2$ defined by
\begin{align}
    J(z_1,z_2) = (-\bar{z}_2, \bar{z}_1),
\end{align}
which preserves $\Sph^3$. The map $J$ is real-orthogonal, so it preserves the round metric on $\Sph^3$. Since $J(e^{i\theta}z) = e^{-i\theta}J(z)$, it satisfies $J^*\sigma = -\sigma$. Thus $J$ preserves $\pi^*g_{\mathrm{FS}} = g_{\Sph^3} - \sigma^2$, as well as the Berger metrics $\tilde{g}_{a,b} = a^2 \sigma^2 + b^2 \pi^*g_{\mathrm{FS}}$. Moreover, $J$ descends to $L_{2,2k} = \Sph^3/\Z_{2k}$ and preserves the quotient-Berger metrics $g_{a,b}$.

Using the bundle description in Appendix \ref{subsec:line-bundle-geometry}, define a smooth map on $\O_{\CP^1}(-2k)$ by
\begin{align}
    \tau(\ell,v) = (J\ell, J^{\otimes 2k} v).
\end{align}
Since $J^2 = -I$, we have $\tau^2 = \mathrm{id}$. On the zero section $\CP^1 \cong \Sph^2$, $\tau$ acts as the antipodal map, so $\tau$ is fixed-point-free. Under \eqref{eq:line-bundle-polar-identification}, $\tau$ preserves the fiber norm $t$ and acts by $J$ on each level positive level set of $t$. Hence, by Appleton's ansatz \eqref{eq:appleton-ansatz-intro}, $\tau$ preserves both the Appleton metric and the soliton potential on $\O_{\CP^1}(-2k)$.

Denote the quotient soliton by $A_{2k}/\Z_2$. The principal orbits of the quotient are
\begin{align}
    L_{2,2k}/\langle \tau \rangle \cong SU(2)/\langle e^{\pi\bfi/k}, \bfj \rangle = SU(2)/\tilde{D}_k.
\end{align}
Thus, $A_{2k}/\Z_2$ is a complete asymptotically $SU(2)/\tilde{D}_k$-cylindrical steady GRS. Its underlying manifold is the total space of the non-orientable $2$-plane bundle over $\RP^2$ with Euler number $\pm k$.

Finally, \cite{law-appleton}*{Theorem 1.3} and a lifting argument imply that $A_{2k}/\Z_2$ is strictly linearly stable, and in particular satisfies $\ker_{L^2(e^f)}(L) = \{0\}$.

\subsection{Cyclic quotients of Appleton solitons with lens-space asymptotic links} \label{subsec:appleton-cyclic-quotients}

Write $\zeta_r = e^{2\pi i/r}$. We use the lens space convention
\begin{align}
    L(p,q) = \Sph^3/ \langle \operatorname{diag}(\zeta_p, \zeta_p^q) \rangle, \qquad \gcd(p,q) = 1.
\end{align}
Under this convention, Appleton's solitons on $\O_{\CP^1}(-k)$, $k \geq 3$, are asymptotically $L(k,1)$-cylindrical. In the next proposition, we quotient these solitons to produce asymptotically $L(p,q)$-cylindrical steady GRSs with $q \not\equiv \pm 1 \pmod p$. These examples are invariant under the action of a $2$-torus, and satisfy the conditions in Corollary \ref{cor:lens-toric-intro}.

\begin{proposition} \label{prop:appleton-cyclic-quotients}
    Let $k \geq 3$, $p \geq 2$, and $1 \leq q \leq p-1$ satisfy $\gcd(p,q) = \gcd(p,qk+1) = 1$. Then there is an effective $\Z_p$-action on $\O_{\CP^1}(-k)$ whose quotient has exactly two isolated orbifold singularities. The image of the unit circle bundle under the quotient is a smooth manifold diffeomorphic to $L(pk,qk+1)$.

    The action preserves the Appleton soliton metric. Thus, the quotient is a complete, asymptotically $L(pk,qk+1)$-cylindrical steady GRS, denoted by $A_{k;p,q}$. It admits an effective isometric $\mathbb{T}^2$-action preserving the soliton potential, and is strictly linearly stable.
\end{proposition}
\begin{proof}
    Consider the action on $\O_{\CP^1}(-k)$ generated by the element
    \begin{align}
        A = \operatorname{diag}(\zeta_{pk}, \zeta_{pk}^{qk+1}) \in U(2),
    \end{align}
    and defined via the formula \eqref{eq:natural-Um-action}. In particular, $A$ scales the fiber over $[1:0] \in \CP^1$ by a factor of $\zeta_p$. Thus $A^r$ acts nontrivially for all $1 \leq r < p$. On the other hand, $A^p = \zeta_k I_2$ acts trivially. So the action of $A$ defines an effective action of the cyclic group $\Z_p$ on $\O_{\CP^1}(-k)$.

    By \eqref{eq:line-bundle-polar-identification}, the quotient of the unit circle bundle is
    \begin{align} \label{eq:unit-appleton-cyclic}
        L_{2,k}/\Z_p \cong \Sph^3/\langle A \rangle = L(pk,qk+1).
    \end{align}
    The action of $A$ on $\Sph^3$ is free because $\gcd(pk,qk+1) = \gcd(p,qk+1) = 1$. Thus the $\Z_p$-action on the bundle is free outside the zero section.

    To examine the action on the zero section, first consider a point $([z_1:z_2], 0) \in \O_{\CP^1}(-k)$ in the zero section such that $z_1 z_2 \neq 0$. The $r$-th power of $A$ fixes this point if and only if
    \begin{align}
        [\zeta_{pk}^r z_1: \zeta_{pk}^{r(qk+1)} z_2] = [z_1:z_2] \quad \Leftrightarrow \quad \zeta_{pk}^{rqk} = 1 \quad \Leftrightarrow \quad p \mid rq.
    \end{align}
    Since $\gcd(p,q) = 1$, the smallest positive $r$ for which this holds is $r=p$. It follows that the $\Z_p$-action on the zero section is free away from the points $[1:0]$ and $[0:1]$.

    It remains to look at the action at the points $([1:0],0)$ and $([0:1],0)$. In a neighborhood of $([1:0],0)$, use the coordinates
    \begin{align}
        \C^2 \ni (z,w) \mapsto ([1:z], w(1,z)^{\otimes k}) \in \O_{\CP^1}(-k).
    \end{align}
    In these coordinates, a computation shows that the generator acts by
    \begin{align}
        (z,w) \mapsto (\zeta_p^q z, \zeta_p w).
    \end{align}
    Similarly, in a neighborhood of $([0:1],0)$, using the coordinates
    \begin{align}
        \C^2 \ni (z,w) \mapsto ([z:1], w(z,1)^{\otimes k}) \in \O_{\CP^1}(-k),
    \end{align}
    we see that the generator acts by
    \begin{align}
        (z,w) \mapsto (\zeta_p^{-q}z, \zeta_p^{qk+1}w).
    \end{align}
    Thus, the points $[1:0]$ and $[0:1]$ in the zero section are fixed points of the action, and after forming the quotient they are isolated orbifold singularities, with local orbifold groups of order $p$ generated respectively by
    \begin{align}
        \operatorname{diag}(\zeta_p^q, \zeta_p) \in U(2) \subset SO(4), \qquad \operatorname{diag}(\zeta_p^{-q}, \zeta_p^{qk+1}) \in U(2) \subset SO(4).
    \end{align}
    The assumption $\gcd(p,q) = \gcd(p,qk+1) = 1$ ensures that these groups act freely on $\Sph^3$.

    Since the Appleton soliton metric on $\O_{\CP^1}(-k)$ is $U(2)$-invariant, and the soliton potential depends only on the fiber norm (hence is $U(2)$-invariant), both are preserved by $A$ and hence the $\Z_p$-action. This gives a complete steady GRS $A_{k;p,q}$ on the quotient. It is asymptotically $L(pk,qk+1)$-cylindrical by the identification \eqref{eq:unit-appleton-cyclic}.

    The diagonal torus $U(1) \times U(1) \subset U(2)$ commutes with $A$, so its action descends to $A_{k;p,q}$ with kernel equal to the order $pk$ subgroup generated by $A$. Thus, $(U(1) \times U(1))/\langle A \rangle \cong \mathbb{T}^2$ acts effectively and isometrically on $A_{k;p,q}$, preserving the soliton potential. Finally, $A_{k;p,q}$ is strictly linearly stable by \cite{law-appleton}*{Theorem 1.3} and a lifting argument.
\end{proof}

For example, the Appleton soliton quotient $A_{3;5,1}$ is a strictly linearly stable, asymptotically $L(15,4)$-cylindrical steady GRS with two isolated orbifold singularities of order $5$.

\bibliography{Refs}

% \bib, bibdiv, biblist are defined by the amsrefs package.
\begin{bibdiv}
\begin{biblist}

\bib{alexandrino-bettiol}{book}{
      author={Alexandrino, M.~M.},
      author={Bettiol, R.~G.},
       title={Lie groups and geometric aspects of isometric actions},
   publisher={Springer Cham},
        date={2015},
}

\bib{andersson-dahl-nilsson}{article}{
      author={Andersson, L.},
      author={Dahl, M.},
      author={Nilsson, G.},
       title={Toric gravitational instantons with four fixed points},
        date={2026},
     journal={Preprint, \href{https://arxiv.org/abs/2611.05333}{\texttt{arXiv:2611.05333}}},
}

\bib{appleton}{article}{
      author={Appleton, A.},
       title={A family of non-collapsed steady {R}icci solitons in even dimensions greater or equal to four},
        date={2017},
     journal={Preprint, \href{https://arxiv.org/abs/1708.00161}{\texttt{arXiv:1708.00161}}},
}

\bib{araneda-lucietti}{article}{
      author={Araneda, B.},
      author={Lucietti, J.},
       title={All toric {H}ermitian {ALE} gravitational instantons},
        date={2026},
     journal={J. Geom. Anal.},
      volume={36},
      number={345},
}

\bib{isom-orb}{article}{
      author={Bagaev, A.~V.},
      author={Zhukova, N.~I.},
       title={The isometry groups of {R}iemannian orbifolds},
        date={2007},
     journal={Sib. Math. J.},
      volume={48},
       pages={579\ndash 592},
}

\bib{bamler-a}{article}{
      author={Bamler, R.~H.},
       title={Entropy and heat kernel bounds on a {R}icci flow background},
        date={2020},
     journal={Preprint, \href{https://arxiv.org/abs/2008.07093}{\texttt{arXiv:2008.07093}}},
}

\bib{bamler-c}{article}{
      author={Bamler, R.~H.},
       title={Structure theory of non-collapsed limits of {R}icci flows},
        date={2020},
     journal={Preprint, \href{https://arxiv.org/abs/2009.03243}{\texttt{arXiv:2009.03243}}},
}

\bib{bamler-b}{article}{
      author={Bamler, R.~H.},
       title={Compactness theory of the space of {S}uper {R}icci flows},
        date={2023},
     journal={Invent. Math.},
      volume={233},
       pages={1121\ndash 1277},
}

\bib{bamler-3}{article}{
      author={Bamler, R.~H.},
      author={Chow, B.},
      author={Deng, Y.},
      author={Ma, Z.},
      author={Zhang, Y.},
       title={Four-dimensional steady gradient {R}icci solitons with 3-cylindrical tangent flows at infinity},
        date={2022},
     journal={Adv. Math.},
      volume={401},
      number={108285},
}

\bib{brendle2013rotational}{article}{
      author={Brendle, S.},
       title={Rotational symmetry of self-similar solutions to the {R}icci flow},
        date={2013},
     journal={Invent. Math.},
      volume={194},
       pages={731\ndash 764},
}

\bib{brendle2014rotational}{article}{
      author={Brendle, S.},
       title={Rotational symmetry of {R}icci solitons in higher dimensions},
        date={2014},
     journal={J. Differential Geom.},
      volume={97},
      number={2},
       pages={191\ndash 214},
}

\bib{brendle-acta}{article}{
      author={Brendle, S.},
       title={Ancient solutions to the {R}icci flow in dimension 3},
        date={2020},
     journal={Acta Math.},
      volume={225},
       pages={1\ndash 102},
}

\bib{bds}{article}{
      author={Brendle, S.},
      author={Daskalopoulos, P.},
      author={Sesum, N.},
       title={Uniqueness of compact ancient solutions to three-dimensional {R}icci flow},
        date={2021},
     journal={Invent. Math.},
      volume={226},
       pages={579\ndash 651},
}

\bib{brendle-naff}{article}{
      author={Brendle, S.},
      author={Naff, K.},
       title={Rotational symmetry of ancient solutions to the {R}icci flow in higher dimensions},
        date={2023},
     journal={Geom. Topol.},
      volume={27},
      number={1},
       pages={153\ndash 226},
}

\bib{bryant}{article}{
      author={Bryant, R.~L.},
       title={Ricci flow solitons in dimension three with ${SO}(3)$-symmetries},
        date={2005},
     journal={Preprint, \url{https://sites.math.duke.edu/~bryant/3DRotSymRicciSolitons.pdf}},
}

\bib{buttsworth}{article}{
      author={Buttsworth, T.},
       title={{$SU(2)$}-invariant steady gradient {R}icci solitons on four-manifolds},
        date={2025},
     journal={Comm. Anal. Geom.},
      volume={33},
      number={2},
       pages={365\ndash 401},
}

\bib{chodosh}{article}{
      author={Chodosh, O.},
       title={Expanding {R}icci solitons asymptotic to cones},
        date={2014},
     journal={Calc. Var. Partial Differential Equations},
      volume={51},
       pages={1\ndash 15},
}

\bib{mcf-generic-ccms}{article}{
      author={Chodosh, O.},
      author={Choi, K.},
      author={Mantoulidis, C.},
      author={Schulze, F.},
       title={Mean curvature flow with generic initial data},
        date={2024},
     journal={Invent. Math.},
      volume={237},
       pages={121\ndash 220},
}

\bib{chodosh-fong}{article}{
      author={Chodosh, O.},
      author={Fong, F. T.-H.},
       title={Rotational symmetry of conical {K}\"ahler--{R}icci solitons},
        date={2016},
     journal={Math. Ann.},
      volume={364},
       pages={777\ndash 792},
}

\bib{chhw}{article}{
      author={Choi, K.},
      author={Haslhofer, R.},
      author={Hershkovits, O.},
      author={White, B.},
       title={Ancient asymptotically cylindrical flows and applications},
        date={2022},
     journal={Invent. Math.},
      volume={229},
       pages={139\ndash 241},
}

\bib{chow2023ricci}{book}{
      author={Chow, B.},
       title={Ricci solitons in low dimensions},
      series={Graduate Studies in Mathematics},
   publisher={American Mathematical Society},
        date={2023},
      volume={235},
}

\bib{CDN}{article}{
      author={Chow, B.},
      author={Deng, Y.},
      author={Ma, Z.},
       title={On four-dimensional steady gradient {R}icci solitons that dimension reduce},
        date={2022},
     journal={Adv. Math.},
      volume={403},
      number={108367},
}

\bib{chowluni}{book}{
      author={Chow, B.},
      author={Lu, P.},
      author={Ni, L.},
       title={Hamilton's {R}icci flow},
      series={Graduate Studies in Mathematics},
   publisher={American Mathematical Society},
        date={2006},
      volume={77},
}

\bib{cm-generic}{article}{
      author={Colding, T.~H.},
      author={Minicozzi~II, W.~P.},
       title={Generic mean curvature flow {I}; generic singularities},
        date={2012},
     journal={Ann. of Math.},
      volume={175},
       pages={755\ndash 833},
}

\bib{cds}{article}{
      author={Conlon, R.~J.},
      author={Deruelle, A.},
      author={Sun, S.},
       title={Classification results for expanding and shrinking gradient {K}ähler--{R}icci solitons},
        date={2024},
     journal={Geom. Topol.},
      volume={28},
      number={1},
       pages={267\ndash 351},
}

\bib{derham}{article}{
      author={de~Rham, G.},
       title={Complexes à automorphismes et homéomorphie différentiable},
        date={1950},
     journal={Ann. Inst. Fourier},
      volume={2},
       pages={51\ndash 67},
}

\bib{deng-orbifold}{article}{
      author={Deng, Y.},
       title={Rigidity of positively curved steady gradient {R}icci solitons on orbifolds},
        date={2025},
     journal={Preprint, \href{https://arxiv.org/abs/2504.14525}{\texttt{arXiv:2504.14525}}},
}

\bib{deng-wang}{article}{
      author={Deng, Y.},
      author={Wang, G.},
       title={Classification of cohomogeneity one steady gradient {R}icci solitons in dimension four and five},
        date={2026},
     journal={Preprint, arXiv:2611.08206},
}

\bib{deng-zhu}{article}{
      author={Deng, Y.},
      author={Zhu, X.},
       title={Classification of gradient steady {R}icci solitons with linear curvature decay},
        date={2020},
     journal={Sci. China Math.},
      volume={63},
       pages={135\ndash 154},
}

\bib{deruelle-asymptotic}{article}{
      author={Deruelle, A.},
       title={Asymptotic estimates and compactness of expanding gradient {R}icci solitons},
        date={2017},
     journal={Ann. Sc. Norm. Super. Pisa Cl. Sci. (5)},
      volume={17},
      number={2},
       pages={485\ndash 530},
}

\bib{deruelle-stability}{article}{
      author={Deruelle, A.},
       title={Stability of non compact steady and expanding gradient {R}icci solitons},
        date={2015},
     journal={Calc. Var. Partial Differential Equations},
      volume={54},
       pages={2367\ndash 2405},
}

\bib{donovan}{article}{
      author={Donovan, P.},
       title={Cohomogeneity one 4-dimensional gradient {R}icci solitons},
        date={2025},
     journal={J. Geom. Anal.},
      volume={35},
      number={401},
}

\bib{eschenburg-wang}{article}{
      author={Eschenburg, J.-H.},
      author={Wang, M.~Y.},
       title={The initial value problem for cohomogeneity one {E}instein metrics},
        date={2000},
     journal={J. Geom. Anal.},
      volume={10},
       pages={109\ndash 137},
}

\bib{fik}{article}{
      author={Feldman, M.},
      author={Ilmanen, T.},
      author={Knopf, D.},
       title={Rotationally symmetric shrinking and expanding gradient {K}\"ahler--{R}icci solitons},
        date={2003},
     journal={J. Differential Geom.},
      volume={65},
       pages={169\ndash 209},
}

\bib{gonzalez-alvaro-phd}{thesis}{
      author={Gonz\'alez~\'Alvaro, D.},
       title={Topological and geometric consequences of curvature and symmetry},
        type={Ph.D. thesis},
     address={Universidad Autónoma de Madrid},
        date={2016},
        note={Available at \url{https://dcain.etsin.upm.es/~david/tesis_main.pdf}},
}

\bib{hershkovits}{article}{
      author={Hershkovits, O.},
       title={Translators asymptotic to cylinders},
        date={2020},
     journal={J. Reine Angew. Math.},
      volume={766},
       pages={61\ndash 71},
}

\bib{kobayashi-groups}{book}{
      author={Kobayashi, S.},
       title={Transformation groups in differential geometry},
      series={Classics in Mathematics},
   publisher={Springer Berlin, Heidelberg},
        date={1995},
}

\bib{kunduri-lucietti}{article}{
      author={Kunduri, H.~K.},
      author={Lucietti, J.},
       title={Existence and uniqueness of asymptotically flat toric gravitational instantons},
        date={2021},
     journal={Lett. Math. Phys.},
      volume={111},
      number={133},
}

\bib{law}{article}{
      author={Law, M.~B.},
       title={An analytical characterization of {E}guchi--{H}anson space and its higher-dimensional analogs},
        date={2026},
     journal={Preprint, \href{https://arxiv.org/abs/2604.23410}{\texttt{arXiv:2604.23410}}},
}

\bib{law-appleton}{article}{
      author={Law, M.~B.},
       title={Uniqueness and stability of {A}ppleton's solitons},
        date={2026},
     journal={Preprint},
}

\bib{li-sun}{article}{
      author={Li, M.},
      author={Sun, S.},
       title={Gravitational instantons and harmonic maps},
        date={2025},
     journal={Preprint, \href{https://arxiv.org/abs/2507.15284}{\texttt{arXiv:2507.15284}}},
}

\bib{li-wang}{article}{
      author={Li, Y.},
      author={Wang, B.},
       title={Rigidity of the round cylinders in {R}icci shrinkers},
        date={2024},
     journal={J. Differential Geom.},
      volume={127},
      number={2},
       pages={817\ndash 897},
}

\bib{isoms-3}{article}{
      author={McCullough, D.},
       title={Isometries of elliptic 3-manifolds},
        date={2002},
     journal={J. London Math. Soc.},
      volume={65},
      number={1},
       pages={167\ndash 182},
}

\bib{mw-smms}{article}{
      author={Munteanu, O.},
      author={Wang, J.},
       title={Smooth metric measure spaces with non-negative curvature},
        date={2011},
     journal={Comm. Anal. Geom.},
      volume={19},
      number={3},
       pages={451\ndash 486},
}

\bib{orlik-raymond}{article}{
      author={Orlik, P.},
      author={Raymond, F.},
       title={Actions of the torus on 4-manifolds. {I}},
        date={1970},
     journal={Trans. Amer. Math. Soc.},
      volume={152},
      number={2},
       pages={531\ndash 559},
}

\bib{torralbo-urbano}{article}{
      author={Torralbo, F.},
      author={Urbano, F.},
       title={Index of compact minimal submanifolds of the {B}erger spheres},
        date={2022},
     journal={Calc. Var. Partial Differential Equations},
      volume={61},
      number={104},
}

\bib{verdiani-ziller}{article}{
      author={Verdiani, L.},
      author={Ziller, W.},
       title={Smoothness conditions in cohomogeneity one manifolds},
        date={2022},
     journal={Transform. Groups},
      volume={27},
       pages={311\ndash 342},
}

\bib{wolf}{book}{
      author={Wolf, J.~A.},
       title={Spaces of constant curvature, 6th ed.},
   publisher={AMS Chelsea Publishing},
        date={2011},
}

\bib{zz}{article}{
      author={Zhao, Z.},
      author={Zhu, X.},
       title={Rigidity of the {B}ryant {R}icci soliton},
        date={2022},
     journal={Preprint, \href{https://arxiv.org/abs/2212.02889}{\texttt{arXiv:2212.02889}}},
}

\end{biblist}
\end{bibdiv}
\end{document}